%% file: main_arxiv.tex
\documentclass[aos,preprint]{imsart}
\setattribute{journal}{name}{}

\usepackage{mathtools}

\makeatletter
\DeclareRobustCommand\widecheck[1]{{\mathpalette\@widecheck{#1}}}
\def\@widecheck#1#2{%
    \setbox\z@\hbox{\m@th$#1#2$}%
    \setbox\tw@\hbox{\m@th$#1%
       \widehat{%
          \vrule\@width\z@\@height\ht\z@
          \vrule\@height\z@\@width\wd\z@}$}%
    \dp\tw@-\ht\z@
    \@tempdima\ht\z@ \advance\@tempdima2\ht\tw@ \divide\@tempdima\thr@@
    \setbox\tw@\hbox{%
       \raise\@tempdima\hbox{\scalebox{1}[-1]{\lower\@tempdima\box
\tw@}}}%
    {\ooalign{\box\tw@ \cr \box\z@}}}
\makeatother

\usepackage{amsthm,amsmath,verbatim}
\usepackage[square,sort,comma,numbers]{natbib}

\RequirePackage[colorlinks,citecolor=blue,urlcolor=blue]{hyperref}
\usepackage{amssymb,graphicx,subfig}
\usepackage{enumitem}
\usepackage{graphicx}
\arxiv{arXiv:1608.07014}

\startlocaldefs
\newcommand{\MF}{\mathcal{F}}
\newcommand{\MI}{\mathcal{I}}
\newcommand{\MU}{\mathcal{U}}

\newcommand{\cI}{\mathcal{I}}
\newcommand{\cD}{\mathcal{D}}

\newcommand{\cB}{\mathcal{B}}
\newcommand{\cF}{\mathcal{F}}
\newcommand{\cC}{\mathcal{C}}

\newcommand{\Hyp}{\sf{H}}

\newcommand{\Pro}{{\sf P}}
\newcommand{\Exp}{{\sf E}}

\newcommand{\bN}{\mathbb{N}}
\newcommand{\bR}{{\mathbb{R}}}

\newcommand{\procGF}{Leap rule}
\newcommand{\procKM}{\text{Sum-Intersection rule}}
\newcommand{\procInt}{Intersection rule}

\newtheorem{theorem}{Theorem}[section]
\newtheorem{lemma}[theorem]{Lemma}
\newtheorem{remark}{Remark}[section]
\newtheorem{corollary}[theorem]{Corollary}
\newtheorem{example}[theorem]{Example}
\newtheorem{problem}{Problem}[section]
\newtheorem{definition}{Definition}[section]

\usepackage{upgreek}
\newcommand{\butheta}{\boldsymbol{\uptheta}}
\endlocaldefs

\begin{document}

\input{a_Intro_prob_formulation}

\input{a_Misclassification_rate}

\input{a_Familywise_rates}

\input{a_FWER_Simulation}

\input{a_Composite}
\input{a_Conclusion}

\input{a_Supp_file}

\section*{Acknowledgements}
This work was supported by the National Science Foundation under Grants  CCF 1514245, DMS 1737962 and in part by the Simons Foundation under Grant C3663.

\bibliographystyle{imsart-number}
\bibliography{GE}

\end{document}

%% file: a_Intro_prob_formulation.tex
\begin{frontmatter}

\title{Sequential multiple testing with generalized error control: \\
 an asymptotic optimality theory}
\runtitle{Sequential multiple testing}


\begin{aug}
\author{\fnms{Yanglei} \snm{Song}\ead[label=e1]{ysong44@illinois.edu}}
\and
\author{\fnms{Georgios} \snm{Fellouris}\ead[label=e2]{fexllouri@illinois.edu}}\\

\address{Department of Statistics, \\
University of Illinois, Urbana-Champaign, \\ 
725 S. Wright Street, Champaign 61820, USA\\ 
Email: \printead*{e1} and \printead*{e2}}

\runauthor{Y. Song and G. Fellouris}
\affiliation{University of Illinois, Urbana-Champaign}
\end{aug}

\begin{abstract}
The sequential multiple testing problem is considered   under two generalized error metrics. Under the first one, the probability of at least $k$ mistakes, of any kind,  is controlled. Under the second,  the probabilities of at least $k_1$ false positives and at least $k_2$ false negatives are simultaneously controlled. For each formulation,  the optimal expected sample size is characterized, to a first-order asymptotic approximation as the error probabilities go to 0,  
and  a novel multiple testing  procedure is proposed  and  shown to be asymptotically efficient  under every signal configuration. These results are  established when the data streams for the various hypotheses are independent and each  local log-likelihood ratio statistic satisfies a certain Strong Law of Large Numbers. In the special case of i.i.d. observations in  each stream, the gains of the proposed sequential procedures over fixed-sample size schemes are quantified. 
\end{abstract}

\begin{keyword}[class=MSC]
\kwd[Primary ]{62L10}
\end{keyword}

\begin{keyword}
\kwd{multiple testing}
\kwd{sequential analysis}
\kwd{asymptotic optimality}
\kwd{generalized familywise error rates}
\kwd{mis-classification rate}
\end{keyword}

\end{frontmatter}

\section{Introduction} \label{intro} 
In the early development of multiple testing,
the focus was on  procedures that control  the probability of at least \textit{one} false positive, i.e., falsely rejected null \citep{marcus1976closed,hommel1988stagewise,holm1979simple}. As this requirement can be prohibitive 
when the number of hypotheses is large,  the emphasis gradually shifted to the control of  less stringent error metrics, such as (i)  the expectation \citep{benjamini1995controlling} or the quantiles \citep{LehmannRomano2005} of the 
\textit{false discovery proportion}, i.e.,   the proportion of false positives among the rejected nulls, and (ii) 
the \textit{generalized familywise  error rate},  i.e., 
 the probability of at least $k \geq 1$ false positives \cite{hommel1988controlled,LehmannRomano2005}.
During the last two decades, various procedures have been proposed to control the above error metrics \cite{benjamini2001,romano2006stepup, romano2007control,guo2014further}.
Further,  the problem of maximizing the number of true positives  subject to a  generalized  control on false positives has been  studied in \cite{lehmann_optimal2005, storey2007optimal, sun2009large, pena2011power}, whereas in \cite{bogdan2011asymptotic} the  false negatives are incorporated into the risk function in a  Bayesian decision theoretic framework .


In all previous references it is assumed that the sample size is deterministic. However, in many applications  data are collected in real time and a reliable decision needs to be made as quickly as possible. Such applications  fall into  the
framework of \textit{sequential hypothesis testing},  which was introduced in the ground-breaking work of
Wald \cite{wald1945sequential} and has been studied extensively since then (see, e.g., \citep{tartakovsky2014sequential}). 

When testing simultaneously \textit{multiple} hypotheses with data collected from a different stream for each hypothesis,  there are two natural generalizations of Wald's sequential framework.  In the first one, sampling can be terminated earlier in some data streams \cite{bartroff2014sequential, malloy2014sequential,2014arXiv_Bartroff}. 
In the second, which is the focus of this paper, sampling is terminated at the same time in all streams \cite{de2012sequential,de2012step}. The latter  setup is  motivated by applications such as  multichannel signal detection \citep{tartakovsky2003sequential}, multiple access wireless network \citep{rappaport1996wireless} and multisensor surveillance systems \citep{foresti2003multisensor},   where a centralized decision maker needs to make a decision regarding the presence or absence of signal, e.g., an intruder, in multiple channels/areas monitored by a number of sensors. This framework is also motivated by
online surveys and  crowdsourcing tasks \citep{kittur2008crowdsourcing}, 
where the goal is to find ``correct'' answers to a fixed number of questions, e.g., regarding some product or service, by asking the smallest necessary number of  people.



In this paper we focus on two related, yet distinct, generalized   error metrics.  The first one is  a generalization of the usual mis-classification rate \cite{malloy2014sequential,li2014}, where   the  probability of at least $k\geq 1$ mistakes, of any kind, is controlled.  The second one 
controls  generalized familywise error rates of both types \cite{2014arXiv_Bartroff,De201588}, i.e., the  probabilities  of  at least $k_1\geq 1$ false positives and  at least $k_2 \geq 1$ false negatives.  

Various sequential procedures have been proposed recently to  control  such generalized  familywise error rates  \citep{bartroff2010multistage,de2012sequential,de2012step, bartroff2014sequential,De201588,2014arXiv_Bartroff}. 
To the best of our knowledge, the  efficiency of these procedures is understood only in the case of \textit{classical} familywise error rates, i.e., when $k_1=k_2=1$. Specifically, in the case of independent streams with i.i.d. observations, an asymptotic lower bound was obtained in \cite{song2017asymptotically} for the optimal expected sample size (ESS) as the error probabilities go to 0,  and was shown to be attained, under any signal configuration, by several existing procedures.
However, the results in \cite{song2017asymptotically}  do not extend to \textit{generalized} error metrics, since the technique for the proof of the asymptotic lower bound requires that the probability of not identifying the correct subset of signals goes to 0. Further, as we shall see, existing procedures fail to be asymptotically optimal, in general, under generalized error metrics.

The lack of an  optimality theory under such generalized error control also implies that it is not well understood how the  best possible ESS depends on the  user-specified parameters. This limits the applicability of generalized error metrics, as it is not clear for the practitioner how to select the number of hypotheses to be ``sacrificed" for the sake of a faster decision.  

In this paper, we address this research gap by  developing  an asymptotic optimality theory for the sequential multiple testing problem under the two generalized error metrics mentioned  above.  Specifically, for each formulation   we  characterize the  optimal ESS as the error probabilities go to 0, and   propose a novel, feasible  sequential multiple testing procedure that achieves the optimal ESS under every signal configuration.   These results are established under the assumption of independent data streams, and require that the  log-likelihood ratio statistic in each stream satisfies a certain Strong Law of Large Numbers. Thus, even in the case of classical familywise error rates, we extend the corresponding  results in \cite{song2017asymptotically} by relaxing the i.i.d. assumption in each stream.


Finally, whenever sequential testing procedures are utilized, it is of interest to quantify the savings in the ESS over fixed-sample size schemes with the same error control guarantees.
In the  case of  i.i.d. data streams, we obtain an  asymptotic lower bound for the gains of sequential sampling over \textit{any} fixed-sample size scheme,  
 and also characterize the asymptotic gains  over a specific fixed-sample size procedure.

In order to convey the main ideas and results with the maximum clarity, we first consider the case that  the local hypotheses are simple, and then extend our results to the case of composite  hypotheses. Thus, the remainder of the paper is organized as follows: 
in Section \ref{sec:formulation}  we formulate the
two problems of interest in  the case of simple hypotheses. The case of generalized mis-classification rate is presented in Section \ref{sec:KM}, and  the case of
 generalized familywise error rates in Section \ref{sec:GF}.
In Section \ref{sec:simulation_GF} we present two simulation studies under the second error metric. In Section \ref{sec:composite} we extend our results to the case of composite hypotheses. We conclude and discuss potential extensions of this work in  Section \ref{sec:conclusion}.  Proofs are presented in the Appendix, where we also present more simulation studies and a detailed analysis of the case of composite hypotheses.  For convenience, we list in Table~\ref{tab:roadmap} the procedures that are considered in this work.

\begin{table}[htbp!]
\centering
\caption{Procedures marked with $\dagger$  are new. Procedures in bold font are asymptotically optimal (AO) without  requiring a special testing structure. GMIS is short for generalized mis-classification rate, and  GFWER  for generalized familywise error rates.}
\label{tab:roadmap}
\begin{tabular}{|c|c|c|c|c|}
\hline
Procedure                 & Metric & Section  & Main results & Conditions for AO \\ \hline
\textbf{Sum-Intersection}$\dagger $        &  GMIS                 & \ref{sum_int_def}      & Thrm~\ref{order_sum_first_opt}             &  \eqref{cond_for_asym_opt}                        \\ \hline
\textbf{Leap}$\dagger$                    &  GFWER                & \ref{subsec:leap_rule}                   & Thrm~\ref{main}             &   \eqref{cond_for_asym_opt}                        \\ \hline
{Asym. Sum-Intersection}$\dagger$ &    GFWER              &   \ref{subsec:asym_sum_int_def}                 & Cor~\ref{delta_0_cor}             & \eqref{cond_for_asym_opt} + \eqref{sym_hom_KM} + \eqref{sym_hom_GF}                          \\ \hline
Intersection             &   both              &  \ref{subsec:inter_rule_def}                  & Cor~\ref{inter_sym}/~\ref{delta_0_cor}              &     \eqref{cond_for_asym_opt} + \eqref{sym_hom_KM} / \eqref{sym_hom_GF}                                             \\ \hline
MNP (fixed-sample)                       &  both                & \ref{subsec:fix_rules}                   & Thrm~\ref{KM_comp_fix_sample_rule}/~\ref{GF_comp_fix_sample_rule}             &          Not optimal                 \\ \hline
\end{tabular}
\end{table}

\section{Problem formulation} \label{sec:formulation}
Consider \emph{independent} streams of observations, $X^j := \{X^j(n): n \in \bN\}$, where $j \in [J] := \{1, \ldots, J\}$ and $\bN :=\{1,2, \ldots\}$. For each $j \in [J]$, we denote  by $\Pro^{j}$   the  distribution of $X^j$ and consider  two simple hypotheses for it,
\begin{align}\label{hypo}
\Hyp_0^j\;: \Pro^j = \Pro_0^j \; \text{ versus }\; \Hyp_1^j\;: \Pro^j = \Pro_1^j.
\end{align}
We denote by $\Pro_A$ the distribution of  $(X^1, \ldots,  X^J)$ when $A \subset [J]$ is the subset of  data streams with signal, i.e., in which the alternative hypothesis is correct.  Due to the assumption of independence  among streams, $\Pro_A$ is the following product measure:
\begin{equation} \label{product}
\Pro_{A} := \bigotimes_{j=1}^J \Pro^j ; \quad\; \Pro^j = \begin{cases}
\Pro_0^j ,\quad \text{ if } j \notin A \\
\Pro_1^j , \quad \text{ if } j \in A.
\end{cases} 
\end{equation}
Moreover,  we denote by  $\MF_n^j$ the  $\sigma$-field generated by the first $n$ observations in the  $j$-{th} stream, i.e.,  $\sigma (X^j(1), \ldots, X^j(n))$, and by  $\MF_n$ the $\sigma$-field  generated by the first $n$ observations in  all streams, i.e.,  $\sigma( \cF_{n}^{j},\, j \in [J] )$, where $n \in \bN$. 

Assuming that the data in all streams  become available \textit{sequentially},  the goal is to stop sampling \textit{as soon as possible}, and upon stopping to solve the   $J$ hypothesis  testing  problems   subject to certain   error control guarantees.  Formally,  a \textit{sequential multiple  testing procedure}  is a   pair  $\delta = (T,D)$ where $T$ is  an $\{\MF_n\}$-stopping time at which  sampling is terminated in all streams,  and $D$  an $\MF_T$-measurable, $J$-dimensional  vector of Bernoullis, $(D^1, \ldots, D^J)$, so that the alternative hypothesis is selected  in the $j$-{th} stream  if and only if $D^j=1$. With an abuse of notation, we also identify  $D$ with the subset of streams in which the  alternative hypothesis is selected upon stopping, i.e.,  $\{j \in [J] : \; D^j = 1\}$.

We consider two  kinds of  error control, which lead to two different problems. Their main difference is that the first one does not differentiate between \textit{false positives}, i.e., rejecting the null when it is correct, and \textit{false negatives}, i.e., accepting the null when it is false.  Specifically, in the first one  we control the generalized mis-classification rate, i.e., the   probability of  committing \textit{at least $k$ mistakes, of any kind,} where $k$ is a  user-specified integer such that $1 \leq k < J$.   When $A$ is the true subset of signals, a decision rule $D$ makes at least $k$ mistakes, of any kind, if $D$ and $A$ differ in at least $k$ components, i.e., $ |A \;\triangle\; D| \geq k$, where for any two sets $A$ and $D$, $A \;\triangle\; D$ is their symmetric difference, i.e. $(A \setminus  D) \cup (D \setminus A)$, and $|\cdot|$  denotes set-cardinality. Thus, 
given tolerance level $\alpha \in (0,1)$, the class of  multiple testing  procedures of interest in this case is 
\begin{align*} 
 \Delta_{k}(\alpha)&:= \left\{ (T,D):  \; \max_{A \subset [J]} \Pro_A( |A \;\triangle\; D| \geq k) \leq \alpha\right\}.
\end{align*}
Then, the  first problem is formulated as follows:

\begin{problem} \label{prob:GM}
Given a user-specified integer $k$  in $[1,J)$,  find a  sequential multiple testing  procedure  that (i) controls the generalized mis-classification rate, i.e., it can be designed to belong to  $\Delta_{k}(\alpha)$
for any given $\alpha$,  and  (ii) achieves
 the smallest possible expected sample size,  
\begin{align*} 
N^*_A(k,\alpha) &:= \inf_{(T,D) \in \Delta_{k}(\alpha)} \Exp_A[T], 
\end{align*}
for every $A \subset [J]$, to a first-order asymptotic approximation as $\alpha \to 0$. 
\end{problem}   

In the second problem of interest in this work, we  control generalized familywise error rates of both types, i.e., the probabilities of  \textit{at least $k_1$ false positives  and   at least $k_2$ false negatives}, where $k_1, k_2 \geq 1$ are  integers  such that $k_1 + k_2 \leq J$.  When the true subset of signals is $A$,  a decision rule $D$ makes at least $k_1$ false positives when $|D\setminus A| \geq k_1$ and  at least $k_2$ false negatives when $|A\setminus D| \geq k_2$. Thus, 
given tolerance levels $\alpha, \beta \in (0,1)$, the class of  procedures of interest  in this case is
\begin{align} \label{gfam}
\begin{split}
\Delta_{k_1, k_2}(\alpha,\beta) :=\{(T,D) : & \; \max_{A \subset [J]} \Pro_{A}(|D\setminus A| \geq k_1) \leq \alpha \quad \text{and} \quad\\    & \; \max_{A \subset [J]}\Pro_{A} (|A \setminus D| \geq k_2) \leq \beta \}.
\end{split}
\end{align}
Then, the second  problem is formulated as follows:
\begin{problem} \label{prob:GFAM}
Given user-specified integers $k_1, k_2 \geq 1$  such that $k_1 + k_2 \leq J$,  find a sequential multiple testing procedure that (i) controls generalized familywise error rates of both types, i.e., it can be designed to belong to $\Delta_{k_1,k_2}(\alpha, \beta)$  for any given $\alpha, \beta \in (0,1)$, and (ii)  achieves the  smallest possible expected sample size,
\begin{align*} 
\begin{split}
N^*_A(k_1,k_2,\alpha,\beta) &:=\inf_{(T,D) \in \Delta_{k_1,k_2}(\alpha, \beta)} \Exp_A[T] ,
\end{split}
\end{align*}
for every $A \subset [J]$,  to a first-order asymptotic approximation as $\alpha$ and  $\beta$ go to 0, at arbitrary rates.
\end{problem}

\subsection{Assumptions}   \label{subsec:distributional}

We now state the assumptions that we will make  in the next two sections  in order to solve these two problems.  First of all, for each $j \in [J]$ we assume that the probability measures $\Pro_0^j$ and $\Pro_1^j$  in \eqref{hypo} are mutually absolutely continuous when restricted to $\MF_n^j$, and  we denote the corresponding log-likelihood ratio (LLR) statistic as follows:
\begin{equation*}
\lambda^{j}(n) := \log \frac{d\Pro_{1}^j}{d\Pro_{0}^j}  (\MF_n^j ), \; \text{ for } n \in \bN. 
\end{equation*}
For   $A,C \subset [J]$ and $n \in \bN$ we denote by   $\lambda^{A,C}(n)$ the LLR   of $\Pro_{A}$ versus $\Pro_{C}$ when both measures are restricted to $\mathcal{F}_n$, and from \eqref{product} it follows that
\begin{align}\label{ll_diff}
\lambda^{A,C}(n) &:=  \log \frac{d\Pro_{A}}{d\Pro_{C}} (\MF_n ) 
= \sum_{j \in A \setminus C} \lambda^{j}(n) -
\sum_{j \in C \setminus A} \lambda^{j}(n).
\end{align}
In order to guarantee that the proposed multiple testing procedures terminate almost surely and satisfy the desired error control,  it will suffice to assume  that 
\begin{align}\label{sep_lr}
\Pro_1^j \left(\lim_{n \to \infty} \lambda^j(n) = \infty \right) 
= \Pro_0^{j} \left(\lim_{n \to \infty} \lambda^j(n) = -\infty \right) = 1 \quad \forall \; j \in [J].
\end{align}
In order to establish an asymptotic lower bound on the optimal ESS for each problem, we will need the stronger assumption that for each $j \in [J]$ there are positive numbers, $\mathcal{I}_1^j, \mathcal{I}_0^j$, such that  the following Strong Laws of Large Numbers (SLLN) hold: 
\begin{align} \label{LLN}
\Pro_1^j \left( \lim_{n \rightarrow \infty} \frac{\lambda^{j}(n) }{n} = \mathcal{I}_1^j\right)\; = \;   
\Pro_0^j \left( \lim_{n \rightarrow \infty} \frac{\lambda^{j}(n) }{n}  = -\mathcal{I}_0^j\right) \; = \; 1.
\end{align} 
When the LLR statistic in each stream  has \textit{independent and identically distributed (i.i.d.)  increments}, the SLLN \eqref{LLN}  will  also be sufficient for establishing the asymptotic optimality of the proposed procedures. When this is not the case,  we will need an assumption on the rate of convergence in  \eqref{LLN}.  Specifically, we will  assume that 
 for every $\epsilon>0$ and $j \in [J]$,
\begin{align} \label{CLLN}
\sum_{n=1}^{\infty}  \Pro_1^j \left( \Big| \frac{\lambda^{j}(n)}{n}  - \mathcal{I}_1^{j} \Big| > \epsilon \right)<\infty, \quad 
\sum_{n=1}^{\infty}  \Pro_0^j \left( \Big| \frac{\lambda^{j}(n)}{n}  + \mathcal{I}_0^{j} \Big| > \epsilon \right)<\infty.
\end{align}
Condition \eqref{CLLN} is known as \textit{complete} convergence \citep{hsu1947complete}, and is  a stronger assumption  than \eqref{LLN}, due to the Borel-Cantelli lemma. This condition is satisfied in various testing problems where the observations in each data stream are dependent,  such as autoregressive time-series models and state-space models. For more details, we refer to \cite[Chapter 3.4]{tartakovsky2014sequential}. 

To sum up, the only distributional assumption for our asymptotic optimality theory  is that the LLR statistic in  each stream 
\begin{equation}\label{cond_for_asym_opt}
\begin{split}
&\text{ either has i.i.d. increments and satisfies the SLLN \eqref{LLN},} \\
&\text{ or satisfies the SLLN with complete convergence \eqref{CLLN}.}
\end{split}
\end{equation}

\begin{remark}
If \eqref{LLN} (resp. \eqref{CLLN}) holds, 
 the normalized LLR,  $\lambda^{A,C}(n)/ n$, defined in \eqref{ll_diff},  converges almost surely (resp. completely) under $\Pro_A$ to 
\begin{align} \label{KLs}
 \mathcal{I}^{A,C} :=  \sum_{i \in A \setminus C}  \mathcal{I}_1^i + \sum_{j \in C \setminus A} \mathcal{I}_0^j. 
\end{align}
The numbers  $\mathcal{I}^{A,C}$ and    
 $\mathcal{I}^{C,A}$ will turn out to determine the inherent difficulty in distinguishing between $\Pro_A$ and $\Pro_C$ and will play an important role in characterizing the optimal performance under $\Pro_A$ and $\Pro_C$, respectively. 
\end{remark}

\subsection{The Intersection rule}\label{subsec:inter_rule_def}
To the best of our knowledge, Problem \ref{prob:GFAM}  has been solved only under the assumption of i.i.d. data streams and \textit{only in the case of classical error control, that is when  $k_1 = k_2 = 1$} \cite{song2017asymptotically}. An asymptotically optimal procedure  in this setup is the so-called ``\textit{Intersection}" rule, $\delta_I:=(T_I, D_I)$, proposed  in \cite{de2012sequential,de2012step}, where 
\begin{align}\label{intersection_rule}
\begin{split}
T_I &:= \inf \left\{n \geq 1: \lambda^j(n) \not \in (-a,b) \;\; \text{for every} \; j \in [J] \right\}, \\
D_I &:=\left\{ j \in [J]:\;  \lambda^j(T_I) > 0 \right\},
\end{split}
\end{align}
and $a,b$ are positive thresholds. This procedure  requires the local test statistic in \textit{every} stream to provide sufficiently strong evidence  for the sampling to be terminated.   The Intersection rule was also  shown in  \cite{De201588} to control \textit{generalized} familywise error rates, however its efficiency in this setup remains an open problem,  even in the case of i.i.d. data streams. Our asymptotic optimality theory in the next sections will reveal that  the Intersection rule  is asymptotically optimal with respect to Problems \ref{prob:GM} and \ref{prob:GFAM} only when the multiple testing problem satisfies  \textit{a very special structure}.  

\begin{definition}
We  say that the multiple testing problem \eqref{hypo} is 
\begin{enumerate}
\item[(i)] \textit{symmetric}, if for every $j \in [J]$ the distribution of $\lambda^j$ under $\Pro_0^j$ is the same as the distribution of $-\lambda^j$ under $\Pro_1^j$,
\item[(ii)] \textit{homogeneous}, if for every $j \in [J]$ the distribution of $\lambda^j$ under $\Pro_i^j$ does not depend on $j$, where $i \in \{0,1\}$.
\end{enumerate}
\end{definition}


It is clear that when the multiple testing problem is both \textit{symmetric and homogeneous}, we have 
\begin{align}\label{sym_hom_KM}
\MI_{0}^j=\MI_1^j=\MI \quad \text{for every} \;  j \in [J].
\end{align}
In the next sections we will  show that the Intersection rule is asymptotically optimal for Problem  \ref{prob:GM} when  \eqref{sym_hom_KM} holds,  whereas its asymptotic optimality with respect to  Problem \ref{prob:GFAM}  will \textit{additionally} require that  the user-specified parameters satisfy the following conditions:
\begin{align}\label{sym_hom_GF}
k_1 = k_2 \quad \text{and} \quad  \alpha = \beta.
\end{align}



\subsection{Fixed-sample size schemes}\label{subsec:fix_rules}
Let $\Delta_{fix}(n)$  denote the class of  procedures for which the decision rule depends on the data collected up to a \textit{deterministic} time $n$, i.e., 
\begin{align*}
\Delta_{fix}(n) := \{(n, D):\; D  \subset [J] \;\text{ is }  \MF_n\text{-measurable} \}.
\end{align*}
For any given integers $k, k_1, k_2 \geq 1$ with $k, k_1+k_2 <J$ and $\alpha, \beta \in (0,1)$, let 
\begin{align}\label{def:fix_ns}
\begin{split}
n^{*}(k,\alpha) &:= \inf\left\{n \in \bN:\; \Delta_{fix}(n) \bigcap \Delta_{k}(\alpha) \neq \emptyset\right\},\\
n^{*}(k_1,k_2,\alpha,\beta) &:= \inf\left\{n \in \bN:\; \Delta_{fix}(n) \bigcap \Delta_{k_1,k_2}(\alpha, \beta) \neq \emptyset\right\},
\end{split}
\end{align}
denote the minimum sample sizes required by \textit{any fixed-sample size scheme} under the two error metrics of interest.
In the case of i.i.d. observations in the data streams, we establish \textit{asymptotic lower bounds} for the above two quantities as the error probabilities go to 0.  To the best of our knowledge, there is no  fixed-sample size procedure that attains these  bounds. For this reason, we also study a specific procedure that runs a \textit{Neyman-Pearson test at each stream}. 
Formally, this procedure is defined as follows:
\begin{align}\label{NP_fixed_sample}
\delta_{NP}(n,h):= (n,D_{NP}(n,h)), \; \; D_{NP}(n,h):= \{j \in [J]:  \lambda^{j}(n)> n h_j \},
\end{align}
where $h=(h_1, \ldots, h_J) \in \bR^{J}$, $n \in \bN$, and we refer to it as  \textit{multiple Neyman-Pearson (MNP) rule}. In the case of Problem \ref{prob:GM} , we characterize the minimum sample size  required by this procedure,
\begin{align*}
n_{NP}(k,\alpha) := \inf\{n \in \bN: \; \exists \;  h \in \bR^{J}, \;\; \delta_{NP}(n,h) \in \Delta_{k}(\alpha)\},
\end{align*}
to a first-order approximation as $\alpha \rightarrow 0$.  In the case of Problem \ref{prob:GFAM}, for simplicity of presentation  we further restrict ourselves to \textit{homogeneous}, but not necessarily symmetric, multiple testing problems, and  characterize the asymptotic minimum sample size  required  by the MNP rule that utilizes the same threshold in each stream, i.e., 
\begin{align*}
\hat{n}_{NP}(k_1,k_2,\alpha,\beta) := \inf\{n \in \bN: \exists \,  h \in \bR,\;\; \delta_{NP}(n,h\mathbf{1}_{J}) \in \Delta_{k_1,k_2}(\alpha,\beta)\},
\end{align*}
where $\mathbf{1}_{J} \in  \bR^{J}$ is a $J$-dimensional vector of  ones.



\subsection{The i.i.d. case}\label{subsec:iid_streams}
As mentioned earlier, our asymptotic optimality theory will apply whenever condition  \eqref{cond_for_asym_opt}  holds, thus, beyond the case of i.i.d. data streams. However, our analysis of 
\textit{fixed-sample size} schemes will rely on  large deviation theory \citep{dembo1998large} and will be focused on  the i.i.d. case. Thus, it is useful to introduce some relevant notations for this setup. 

Specifically, when for each $j \in [J]$  the observations in the $j$-{th} stream  are independent with common density $f^j$ relative to a $\sigma$-finite measure $\nu^j$,   the hypothesis testing problem \eqref{hypo}  takes the form 
\begin{equation}\label{iid_case}
{\Hyp}_0^j\;: f^j = f_0^j \; \text{ versus }\; {\Hyp}_1^j\;: f^j = f_1^j,
\end{equation}
and $\MI_1^{j}, \MI_0^{j}$ correspond to the  \textit{Kullback-Leibler divergences} between  $f_1^j$ and $f_0^j$, i.e., 
\begin{equation}\label{KL_numbers}
\MI_1^{j} = \int \log\left( {f_1^j}/{f_0^j} \right) f_1^j \; d\nu^j,  \quad 
\MI_0^{j} = \int \log\left({f_0^j}/{f_1^j} \right) f_0^j \;  d\nu^j.
\end{equation}
In this case, each LLR statistic $\lambda^j$ has i.i.d. increments, and   \eqref{cond_for_asym_opt} is satisfied as long as  $\MI_1^{j}$ and $\MI_0^{j}$ are both positive and finite.  For each $j \in [J]$, we further introduce the convex conjugate of the cumulant generating function of $\lambda^{j}(1)$ 
\begin{equation}
\label{convex_conjugate}
z \in \bR \mapsto  \Phi^j(z) := \sup_{\theta \in \bR} \left\{z\theta   -  \Psi^j(\theta) \right\},  \; \text{where}   \; \Psi^j(\theta):=  \log \Exp_{0}^{j}\left[ e^{\theta \lambda^j(1)}  \right].
\end{equation}
The value of  $\Phi^j$  at zero is   the \textit{Chernoff information} \citep{dembo1998large} for the testing problem \eqref{iid_case}, and we will denote it as $\cC^j$, i.e., $\cC^j:=\Phi^{j}(0)$.

Finally, we will  illustrate our general results in the 
case of  testing normal means. Hereafter, $\mathcal{N}$ denotes the density of the normal distribution. 

\begin{example}  \label{Gaussian_mean_test}

If $f_0^j= \mathcal{N}(0, \sigma_j^2)$ and
$f_1^j = \mathcal{N}(\mu_j, \sigma_j^2)$ for all $j \in [J]$, then $$\lambda^j(1)= \theta_j^2\left( X^j(1)/\mu_j- 1/2 \right), \quad \text{where} \quad  \theta_j:=\mu_j/\sigma_j.
$$ 
Consequently the multiple testing problem is symmetric and
\begin{equation} \label{normal_related_quantity}
\cI^j := \cI_0^j = \cI_1^j = {\theta_j^2}/2,\quad
\Phi^j(z) = {(z+\cI^j)^2}/{(4\cI^j)} \text{ for any } z \in \bR.
\end{equation}
\end{example}

\subsection{Notation}

We collect here some notations that will be used extensively  throughout the rest of the paper:   $C_{k}^{J}$  denotes the binomial coefficient $\binom{J}{k}$, i.e., the number of subsets of size $k$ from a set of size  $J$;  
$a \vee b$ represents $\max\{a,b\}$; $x \sim y$ means that $\lim_{y} x/y = 1 $ and  $x(b) = o(1)$ that $\lim_{b} x(b)  = 0$, with $y,b \to 0$ or $\infty$.
Moreover,  we recall that $|\cdot|$  denotes set-cardinality, $\bN :=\{1,2, \ldots\}$, $[J] := \{1, \ldots, J\}$, and that    $A \;\triangle\; B$ is the symmetric difference, $(A \setminus  B) \cup (B \setminus A)$, of two sets $A$ and $B$.

%% file: a_Misclassification_rate.tex
\section{Generalized mis-classification rate} \label{sec:KM}
In this section we consider Problem \ref{prob:GM}
and  carry out  the following program: first, we propose a novel  procedure that  controls  the generalized mis-classification rate. Then, we establish an asymptotic lower bound on the optimal ESS and  show that it is attained  by the  proposed scheme. As a corollary, we show that  the Intersection rule is asymptotically optimal when condition \eqref{sym_hom_KM}  holds. 
Finally, we make a comparison with fixed-sample size procedures in the i.i.d. case \eqref{iid_case}.

\subsection{Sum-Intersection rule} \label{sum_int_def}
In order to implement the proposed   procedure, which we will denote $\delta_{S}(b) := (T_{S}(b), D_S(b))$, we need at each time $n \in \bN$ prior to stopping to  order the  \textit{absolute values} of the local test statistics,  $|\lambda^j(n)|,  j \in [J]$. If we denote the corresponding ordered values by  
 $$
 \widetilde{\lambda}^{1}(n) \leq    \ldots \leq \widetilde{\lambda}^{J}(n),
$$
 we can think of  $\widetilde{\lambda}^{1}(n)$ (resp.  $\widetilde{\lambda}^{J}(n)$)  as the least (resp. most) ``significant''  local test statistic at time $n$, in the sense that it provides the weakest (resp. strongest)  evidence in favor of either the null or the alternative. 
 Then, sampling is terminated  at the first time    the \textit{sum of the $k$ least  significant local LLRs} exceeds some  positive  threshold $b$, and  the null hypothesis is rejected  in every stream that has a positive LLR upon stopping, i.e., 
\begin{align*} 
T_S(b)&:= \inf\left\{n \geq 1 : \sum_{j=1}^{k} \widetilde{\lambda}^{j}(n) \geq b \right\}, \; 
D_S(b) := \left\{ j \in [J]:\;  \lambda^j(T_S(b)) > 0 \right\}.
\end{align*}
The threshold $b$ is selected to guarantee the desired error control.   When $k=1$, $\delta_{S}(b)$ coincides with the \procInt, $\delta_{I}(b,b)$, defined in \eqref{intersection_rule}. When $k>1$, the two  rules are different but  share a similar flavor, since  $\delta_{S}(b)$  stops the first time $n$  that \textit{all sums $\sum_{j \in B} |\lambda^j(n)|$ with $B \subset [J]$ and $|B|=k$} are simultaneously above $b$.  For this reason, we  refer to $\delta_S(b)$ as  \textit{\procKM}.  Hereafter,  we typically  suppress the  dependence of $\delta_S(b)$  on  threshold $b$  in order to lighten the notation. 

\subsection{Error control of the Sum-Intersection rule}\label{subsec:sum_intersection_error_control}
For any choice of  threshold $b$,  the \procKM~clearly terminates  almost surely,  under every signal configuration, as long as   condition \eqref{sep_lr} holds.  In the next theorem we show  how to select  $b$   to guarantee the desired error control. We stress that  no additional distributional assumptions are needed for this purpose. 

\begin{theorem} \label{KM_err_control}
Assume \eqref{sep_lr} holds. For any $\alpha \in (0,1)$ we have   $\delta_S(b_\alpha) \in \Delta_{k}(\alpha)$ when 
\begin{equation} \label{OS_threshold}
b_\alpha = |\log(\alpha)| + \log(C_{k}^{J}).
\end{equation}
\end{theorem}
\begin{proof}
The proof can be found in Appendix \ref{proof_of_KM_err_control}.
\end{proof}

The choice of $b$ suggested by the previous theorem will be  sufficient for establishing the asymptotic optimality  of the \procKM, but  may be conservative for practical purposes. In the absence of more accurate approximations for the error probabilities, we recommend  finding the value of $b$ for which the target level is attained using  Monte Carlo simulation. This means simulating off-line, i.e., before the sampling process begins,  for every   $A \subset [J]$  the error probability $\Pro_A(|A \;\triangle \; D_S(b)| \geq k)$ for various values of $b$, and then selecting the value for which   the maximum of these probabilities over $A \subset [J]$ matches the nominal level $\alpha$. 
 
This simulation task is significantly facilitated when the multiple testing problem 
has a special structure. If the problem is \textit{symmetric}, for any given threshold $b$  the error probabilities of the \procKM~coincide for all $A \subset [J]$; thus, it suffices to simulate the error probability under a single measure, e.g.,  $\Pro_\emptyset$.  
If the problem is \textit{homogeneous}, the error probabilities depend only on the size of $A$, not the actual subset; thus, it suffices to simulate the above probabilities  for at most $J+1$ configurations.  Similar ideas apply  in the presence of block-wise homogeneity.

Moreover, it is worth pointing out that when $b$ is large,  importance sampling techniques can be applied to simulate the corresponding ``small'' error probabilities, similarly to \cite{song2016logarithmically}.



\subsection{Asymptotic lower bound  on the optimal performance} 

We now obtain an  asymptotic (as $\alpha \rightarrow 0$) lower bound on $N^*_A(k,\alpha)$, the optimal ESS  for Problem \ref{prob:GM} when the true subset of signals is $A$, for any given $k \ge 1$. When $k=1$, from \cite[Theorem 2.2]{tartakovsky1998asymptotic} it follows that when \eqref{LLN} holds, such a lower bound is  given  by $|\log(\alpha)|/  \min_{C \neq A}\MI^{A,C}$, where  $\cI^{A,C}$ is defined in \eqref{KLs}.  Thus, the asymptotic lower bound when $k=1$ is determined by the  ``wrong'' subset that is the most difficult to be distinguished from $A$, where the difficulty level is quantified by the information numbers defined in \eqref{KLs}. 

The techniques in  \cite{tartakovsky1998asymptotic}  require  that the probability of selecting the wrong subset  goes to 0; thus, they do not apply to the case of generalized error control ($k>1$).  Nevertheless, it is reasonable to conjecture that the corresponding asymptotic  lower bound when $k>1$ will still  be determined by the  wrong subset that is the most difficult to be distinguished from $A$,   with the difference that  a subset will now be ``wrong'' under $\Pro_A$ \textit{if it  differs from $A$ in at least $k$ components}, i.e., if it  does \textit{not} belong to
\begin{equation*}
\MU_k(A) := \{C \subset [J]: |A \; \triangle \; C| < k \}.
\end{equation*}
This conjecture is verified by the following theorem. 
\begin{theorem}\label{KM_lower_bound}
Fix $k\geq 1$. If  \eqref{LLN} holds,  then  for any $A \subset [J]$,   as  $\alpha \to 0$,
\begin{align} \label{ALB}
N_A^*(k,\alpha)  \geq  \frac{|\log(\alpha)|}{\mathcal{D}_A(k)} (1 - o(1)),\; \text{where} \quad 
\mathcal{D}_A(k) := \min_{C \not \in \MU_k(A)} \MI^{A,C}.
\end{align}
\end{theorem}

The proof  in the case of the \textit{classical} mis-classification rate  ($k=1$) 
 is based on a change of measure from $\Pro_A$ to $\Pro_{A^*}$, where $A^*$ is chosen such that (i) $A$ is a ``wrong'' subset under $\Pro_{A^*}$, i.e., $A \neq A^*$ 
  and (ii) $A^*$ is ``close'' to $A$, in the sense that 
$\MI^{A, A^*} \leq  \MI^{A,C}$ for every  $C \neq A$ (see, e.g.,  \cite[Theorem 2.2]{tartakovsky1998asymptotic}).

When $k \geq 2$, there are more than one ``correct'' subsets under $\Pro_A$. The key  idea in our proof is that    for \textit{each}  ``correct" subset  $B \in \MU_{k}(A)$ 
we apply a different change of measure  $\Pro_A \to \Pro_{B^{*}}$, 
where $B^*$ is chosen such that (i) $B$ is a ``wrong'' subset under $\Pro_{B^*}$, i.e., $B \notin  \MU_{k}(B^*)$,  and (ii) $B^*$ is ``close'' to $A$, in the sense that  $I^{A,B^*} \leq  \MI^{A,C}$ for every 
$C \notin \MU_k(A)$.
The existence of such $B^*$ is established in Appendix~\ref{app:exist_B_star},
and the proof of Theorem~\ref{KM_lower_bound} is carried out in Appendix~\ref{app:proof_KM_lower_bound}.

\subsection{Asymptotic optimality}
We are now ready to establish the asymptotic optimality of the 
\procKM~by showing that it attains the asymptotic lower bound of Theorem \ref{KM_lower_bound} under every  signal configuration. 

\begin{theorem}\label{order_sum_first_opt}
Assume \eqref{cond_for_asym_opt} holds. Then, for any $A \subset [J]$ we have   as $b \to \infty$ that 
\begin{equation} \label{KM_upper_bound_equation}
\Exp_A[T_S(b)] \leq \frac{b}{\mathcal{D}_A(k)}  \, (1 + o(1)).
\end{equation}
When in particular $b$ is selected such that $\delta_S \in   \Delta_{k}(\alpha)$ and  $b\sim |\log(\alpha)|$, e.g. as in \eqref{OS_threshold},  then for every  $A \subset [J]$ we have as  $\alpha \to 0$ 
$$\Exp_A \left[ T_S \right]
\, \sim\,  \frac{|\log \alpha|}{\mathcal{D}_A(k)}\, \sim \,N_A^*(k,\alpha).
$$
\end{theorem}

\begin{proof}
If \eqref{KM_upper_bound_equation} holds and $b$ is such that  $\delta_S \in   \Delta_{k}(\alpha)$ and  $b\sim |\log(\alpha)|$, then $\delta_S$ attains the asymptotic lower bound in Theorem \ref{KM_lower_bound}.   Thus, it suffices to prove \eqref{KM_upper_bound_equation}, which is done in the Appendix \ref{proof_of_KM_upper_bound}. 
\end{proof}

The asymptotic characterization of the optimal ESS, $N_A^*(k,\alpha)$,   illustrates the trade-off among the ESS, the number of mistakes to be tolerated, and the error tolerance level  $\alpha$.
Specifically, it suggests that, for ``small'' values of $\alpha$,    tolerating $k-1$ mistakes reduces the ESS by a factor of $\mathcal{D}_A(k)/\mathcal{D}_A(1)$, which is \textit{at least}  $k$ for every $A \subset [J]$. To justify the latter claim, note that if 
we denote the ordered  information numbers $\{\mathcal{I}_1^{j}, j \in A\} \cup \{\mathcal{I}_0^{j},  j \notin A\}$ by
$\widetilde{\mathcal{I}}^{(1)}(A) \leq \ldots  \leq \widetilde{\mathcal{I}}^{(J)}(A)$,
then 
$$
\mathcal{D}_A(k) \;=\; \sum_{j=1}^k \widetilde{\mathcal{I}}^{(j)}(A). 
$$


In the following corollary we show that
the Intersection rule is asymptotically optimal when
\eqref{sym_hom_KM} holds, which is the case for example when the multiple testing problem is \textit{both symmetric and homogeneous}.




\begin{corollary}\label{inter_sym}
(i) Assume \eqref{sep_lr} holds. For any $\alpha \in (0,1)$ 
we have  $\delta_I(b,b) \in \Delta_{k}(\alpha)$ when $b$ is equal to $b_\alpha/k$, where $b_\alpha$ is defined in \eqref{OS_threshold}. 

\noindent (ii) Suppose $b$ is selected such that  $\delta_I(b,b) \in \Delta_{k}(\alpha)$ and  $b \sim |\log \alpha|/k$, e.g., as in  (i). If \eqref{cond_for_asym_opt} holds, then 
$$
\Exp_A \left[ T_I \right] \leq   \frac{|\log \alpha|}{k  \cD_A(1)  } \; (1+o(1)).
$$
If also \eqref{sym_hom_KM} holds, then  for any $A \subset [J]$ we have as $\alpha \to 0$  that
$$
\Exp_A \left[ T_I \right] \sim  \frac{|\log \alpha|}{k \MI} \sim N_A^*(k,\alpha).
$$
\end{corollary}
\begin{proof}
The proof can be found in Appendix \ref{proof_of_inter_sym}.
\end{proof}

\begin{remark}
When \eqref{sym_hom_KM} is violated, the \procInt~fails to be asymptotically optimal. This will be  illustrated with a simulation study  in Appendix \ref{sec:asym_case_KM}.
\end{remark}

\subsection{Fixed-sample size rules}\label{subsec: KM_fix_comp}
Finally, we  focus on the i.i.d. case \eqref{iid_case} and consider procedures that stop at a deterministic time, selected to control the generalized mis-classification rate. 
We recall that  $\cC^j$ is the Chernoff information in the $j^{th}$ testing problem, and we denote by $\cB(k)$ the sum of the smallest $k$ local 
Chernoff informations, i.e., 
\begin{equation*} 
\cB(k) := \sum_{j=1}^k  \cC^{(j)},
\end{equation*}
where $\cC^{(1)}\leq  \cC^{(2)} \leq \ldots \leq \cC^{(J)}$ are  the ordered values of the local Chernoff information numbers $\cC^{j}, j \in [J]$.

\begin{theorem}
\label{KM_comp_fix_sample_rule}
Consider the multiple testing problem with i.i.d. streams defined in \eqref{iid_case} and suppose that the Kullback-Leibler  numbers in  \eqref{KL_numbers} are positive and finite.  For any user-specified integer $1 \leq k \leq (J+1)/2$ and  $A \subset [J]$, we have as $\alpha \to 0$ 
\begin{align*}
\frac{\cD_A(k)}{\cB(2k-1)} \; (1-o(1)) \; \leq \;
\frac{n^{*}(k,\alpha)}{ N_A^*(k,\alpha)} \; \leq \; 
\frac{n_{NP}(k,\alpha)}{ N_A^*(k,\alpha)}  \; \sim \; 
 \frac{\cD_A(k)}{ \cB(k)}.
\end{align*}

\end{theorem}
\begin{proof}
The proof can be found in Appendix \ref{proof_of_KM_comp_fix_sample_rule}.
\end{proof}

\begin{remark}
Since any fixed time is also a stopping time, the lower bound is  relevant only when 
$\cD_A(k) > \cB(2k-1)$ for some $A \subset [J]$.
\end{remark}

We now specialize the results of the  previous theorem to the \textit{testing of normal means}, introduced  in Example \ref{Gaussian_mean_test} (a Bernoulli example is presented  in Appendix \ref{app:binom_example}). In this case, $\cC^j = \cI^j/4$ for every $j \in [J]$,  which implies $\cD_A(k)= 4 \cB(k)$  for every $A \subset [J]$, and  by Theorem  \ref{KM_comp_fix_sample_rule} it follows that  
$$n_{NP}(k,\alpha) \sim 4\,  N_A^*(k,\alpha) \quad   \forall \;  A \subset [J].
$$ 
That is,  for any $k \in [1,(J+1)/2]$, 
 when  utilizing the MNP rule instead of the proposed asymptotically optimal Sum-Intersection rule,
the ESS increases  by roughly a  factor of $4$, for small values of $\alpha$, under every configuration. 
From Theorem  \ref{KM_comp_fix_sample_rule}  it also follows that 
       for any $A \subset [J]$ we have
\begin{align*}
\liminf_{\alpha \to 0} \;
\frac{n^{*}(k,\alpha)}{ N_A^*(k,\alpha)}  &\geq 
\frac{ 4\cB(k)}{\cB(2k-1)}.
\end{align*}
If in addition the hypotheses have identical information numbers, i.e., \eqref{sym_hom_KM} holds, this lower bound is  always larger than $2$, which means that 
 \textit{any} fixed-sample size scheme will require  at least twice as many observations as the Sum-Intersection rule, for small error probabilities. 




%% file: a_Familywise_rates.tex
\section{Generalized familywise error rates \textit{of both kinds}}\label{sec:GF}
In this section we study Problem \ref{prob:GFAM}.
While we follow similar ideas and the results are of similar nature as in the previous section, the proposed procedure and the proof of its asymptotic optimality turn out to be much more complicated.

To describe the proposed multiple testing procedure, we first need to introduce some additional notations. Specifically, we denote by   $$ 0 <  \widehat{\lambda}^{1}(n) \leq  \ldots \leq \widehat{\lambda}^{p(n)}(n)
$$
the  order statistics of the \textit{positive} LLRs at  time $n$,  $\{\lambda^j(n): \lambda^{j}(n) > 0, \, j \in [J]\}$,
where $p(n)$ is the number of the strictly positive LLRs  at time $n$. Similarly,   we  denote by $$
 0 \leq  \widecheck{\lambda}^{1}(n) \leq  \ldots \leq \widecheck{\lambda}^{q(n)}(n) 
 $$ 
the  order statistics of the absolute values of the \textit{non-positive} LLRs at time $n$, i.e.,
$\{-\lambda^j(n): \lambda^j(n) \leq 0, \, j \in [J]\}$, 
where  $q(n):=J-p(n)$. 
We also  adopt the  following convention:
\begin{align} \label{exceed_convention}
\begin{split}
\widehat{\lambda}^j (n) = \infty \;  \text{ if } \;  j > p(n), \quad \text{ and } \quad
\widecheck{\lambda}^{j}(n) = \infty \;  \text{ if } \; j > q(n).
\end{split}
\end{align}
 Moreover, we  use the following notation 
\begin{align*}
\lambda^{ \widehat{i}_j(n)}(n) &:= \widehat{\lambda}^{j}(n),\quad \forall\; j \in \{1, \ldots, p(n)\},\\
\lambda^{\widecheck{i}_j(n)}(n) &:=- \widecheck{\lambda}^{j}(n), \quad\forall\; j \in \{1, \ldots, q(n)\},
\end{align*}
for the  indices of  streams with \textit{positive}  and  \textit{non-positive} LLRs at time $n$, respectively.
Thus,   stream $\widehat{i}_1(n)$ (resp.  $\widecheck{i}_1(n)$)  has the least  significant positive (resp. negative) LLR at time $n$.  

\subsection{Asymmetric \procKM}\label{subsec:asym_sum_int_def}
We start with a procedure that has the same decision rule as the  Sum-Intersection procedure (Subsection \ref{sum_int_def}), but a different  stopping rule that accounts for the asymmetry in the error metric that we consider in this section. 
Specifically, we consider  a procedure $\delta_0(a,b) \equiv  (\tau_0, D_0)$ that stops  as soon as  the following two conditions are satisfied simultaneously: (i)  the sum of the $k_1$ least   significant positive LLRs is larger than $b>0$, and (ii)  the  sum of the $k_2$ least  significant negative LLRs is smaller than  $-a<0$. 
%
 Formally, 
\begin{align}\label{tau_0_def}
\begin{split}
\tau_0 &:=\inf\left\{n \geq 1: \sum\limits_{j=1}^{k_1} \widehat{\lambda}^{j}(n) \geq b \; \text{ and } \;  \sum\limits_{j=1}^{k_2} \widecheck{\lambda}^{j}(n) \geq a \right\},\; \\
D_0 &:=\left\{ j \in [J]:\;  \lambda^j(\tau_0) > 0 \right\} =\left\{ \widehat{i}_{1}(\tau_0), \ldots, \widehat{i}_{p(\tau_0)}(\tau_0)\right\}.
\end{split}
\end{align}
We  refer to this procedure as \textit{asymmetric Sum-Intersection rule}. Note that similarly to the \procKM, this procedure 
does not require strong evidence from every individual stream in order to terminate sampling. Indeed, upon stopping  
there may be insufficient evidence for the hypotheses that correspond to the  $k_1-1$ least  significant positive statistics and the  $k_2-1$  least   significant negative statistics, turning them into the anticipated false positives and false negatives, respectively, which we are allowed to make.

We will see that while the asymmetric Sum-Intersection rule can control  generalized familywise error rates of both types, it is not in general  asymptotically optimal.  To understand  why this is the case,  let $A$ denote true  subset of streams with signals and suppose that there is a subset $B$  of   $\ell$ streams with \textit{noise}, i.e.,   $B \subset A^c$ with $|B|=\ell$, such that   $\ell < k_1$ and 
$$
\MI_1^{j} \; \gg \; \MI_0^{i_1}\; \gg\; \MI_0^{i_2}, \quad \forall\; j \in A,\quad  i_1 \in A^c\setminus B,\quad i_2 \in B, 
$$
i.e.,  the hypotheses in  streams with signal are much easier than in streams with noise, and the  hypotheses in  $B$ are much harder than in the other streams with noise.  In this case, the first stopping requirement in $\tau_0$ will be easily satisfied, but not the second one, since  the streams in $B$ will slow down the growth of the sum of the $k_2$ least significant negative LLRs. 



These observations suggest that  the performance of $\delta_{0}$ can be improved in the above scenario   if we essentially ``give up'' the testing problems in $B$, presuming that we will make $\ell$ of the $k_1-1$ false positives in these streams. This can be achieved by (i)  ignoring the $\ell$ least significant negative statistics in the second stopping requirement of $\tau_0$, and asking  the  sum of the \textit{next} $k_2$ least significant negative statistics to be small upon stopping, and (ii)
modifying the decision rule to  reject  the nulls not only in streams with positive LLR,   but also in the $\ell$ streams with the least significant \textit{negative} LLRs upon stopping.  However, if we modify the decision rule in this way,  we have spent from the beginning $\ell$ of the $k_1-1$   false positives we are allowed to make. This  implies that we need to also modify the first stopping requirement in $\tau_0$ and ask the sum of the $k_1-\ell$ least significant positive LLRs to  be  large  upon stopping. 
If we denote by $\widehat{\delta}_{\ell} :=(\widehat{\tau}_{\ell},\widehat{D}_{\ell})$ the procedure that incorporates the above modifications, then 
\begin{align*}
\begin{split}
\widehat{\tau}_\ell &:=\inf\left\{n \geq 1: \sum\limits_{j=1}^{k_1-\ell} \widehat{\lambda}^{j}(n) \geq b \;\text{ and }\; \sum\limits_{j=\ell+1}^{\ell+k_2} \widecheck{\lambda}^{j}(n) \geq a \right\}, \\
\widehat{D}_\ell &:=\{ \widehat{i}_{1}(\widehat{\tau}_\ell), \ldots, \widehat{i}_{p(\widehat{\tau}_\ell)}(\widehat{\tau}_\ell)\}\; \bigcup\; \{\widecheck{i}_1(\widehat{\tau}_\ell), \ldots, \widecheck{i}_{\ell}(\widehat{\tau}_\ell) \},
\end{split}
\end{align*}
where  we omit the dependence on $a,b$ in order to lighten the notation.

By the same token, if there are $\ell<k_2$ streams \textit{with signal}   in which the testing problems are much harder than in  other streams,  it is reasonable to expect that $\delta_0$ may be outperformed by  a procedure
$\widecheck{\delta}_{\ell} := (\widecheck{\tau}_{\ell}, \widecheck{D}_{\ell})$, where
\begin{align*}
\begin{split}
\widecheck{\tau}_\ell &:=\inf\left\{n \geq 1: \sum\limits_{i=\ell+1}^{\ell+k_1} \widehat{\lambda}^{i}(n) \geq b \;\text{ and }\; \sum\limits_{j=1}^{k_2-\ell} \widecheck{\lambda}^{j}(n) \geq a \right\} \\
\widecheck{D}_\ell &:=\{ \widehat{i}_{\ell+1}(\widecheck{\tau}_\ell), \;\ldots\;, \widehat{i}_{p(\widecheck{\tau}_\ell)}(\widecheck{\tau}_\ell)\}.
\end{split}
\end{align*}
Fig. \ref{fig:leap_proc_vis} provides a visualization of these stopping rules.

\begin{figure}[htbp]
\centering
\includegraphics[width=0.9\textwidth]{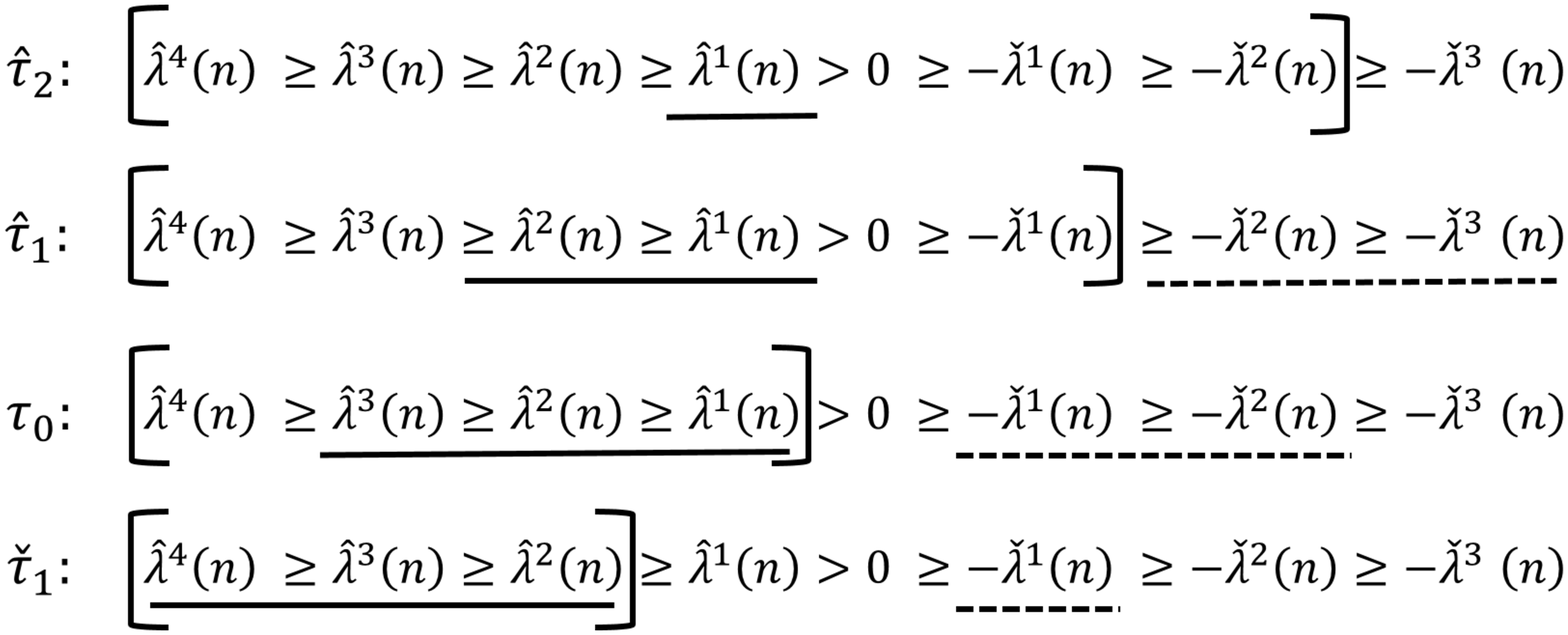}
\caption{ \small{Set $J=7$, $k_1 = 3$, $k_2 = 2$.  Suppose at time $n$, $p(n) = 4, q(n) = 3$. Each rule stops  when the sum of the terms with solid underline exceeds $b$, and at the same time the sum of the terms with dashed underline is below $-a$. Upon stopping, the null hypothesis for the streams in the bracket are rejected. Note that by convention~\eqref{exceed_convention}, $\check{\lambda}^{4}(n) = \infty$, which makes the stopping rule $\hat{\tau}_2$  have only one condition to satisfy.}}
\label{fig:leap_proc_vis}
\end{figure}


\subsection{The Leap rule}\label{subsec:leap_rule}
The previous discussion suggests that the asymmetric Sum-Intersection rule, defined in \eqref{tau_0_def},  may be significantly outperformed by some of the  procedures, $\{\widehat{\delta}_{\ell}, 0 \leq \ell <k_1\}$ and  $\{\widecheck{\delta}_{\ell}, 1 \leq \ell < k_2\}$, under some signal configurations, when the multiple testing problem is  \textit{asymmetric and/or inhomogeneous}. In this case, we 
propose combining the above procedures, i.e.,  stop as soon as any of them does so, and use the corresponding decision rule upon stopping. If multiple stopping criteria are satisfied at the same time, we then use the decision rule that rejects the most null hypotheses.  

 Formally,  the proposed procedure $\delta_L := (T_L, D_L)$ is defined as follows:
\begin{align}\label{leap_def}
\begin{split}
T_L &:= \min\left\{\min_{0 \leq \ell < k_1} \widehat{\tau}_{\ell},\;  \min_{1 \leq \ell < k_2} \widecheck{\tau}_{\ell} \right\}, \\
D_L &:= \left( \bigcup_{0 \leq \ell < k_1, \widehat{\tau}_{\ell} = T_{L}} \widehat{D}_{\ell} \right)\; \bigcup \;
\left( \bigcup_{1 \leq \ell < k_2, \widecheck{\tau}_{\ell} = T_{L}}  \widecheck{D}_{\ell} \right),
\end{split}
\end{align}
and we  refer to it  as  ``\procGF", because $\widehat{\delta}_{\ell}$ (resp. $\widecheck{\delta}_{\ell}$)   
``leaps" across the $\ell$ least significant  negative (resp. positive) LLRs.


\subsection{Error control of the Leap rule} \label{subsec:leap_error_control}
We now show that the  \procGF~can   control  generalized familywise error rates of both types.

\begin{theorem}\label{leap_error_control}
Assume \eqref{sep_lr} holds. For any $\alpha, \beta \in (0,1)$ we have that  $\delta_L \in \Delta_{k_1,k_2}(\alpha,\beta)$ when the thresholds are selected as follows:
 \begin{equation}\label{leap_thresholds}
  a = |\log(\beta)| + \log(2^{k_2} C^{J}_{k_2}), \quad b = |\log(\alpha)| + \log(2^{k_1} C^{J}_{k_1}).
 \end{equation}
\end{theorem}
\begin{proof}
The proof can be found in Appendix \ref{proof:leap_error_control}.
\end{proof}

The above threshold values are sufficient for establishing the asymptotic optimality of the Leap rule, but may be conservative in practice. Thus, as in the previous section, we recommend  using simulation  to find the thresholds that attain the target error probabilities. This means simulating for every  $A \subset [J]$ the error probabilities of the Leap rule,
$\Pro_{A}(|D_L(a,b) \setminus A| \geq k_1)$ and
$\Pro_{A}(|A \setminus D_L(a,b)| \geq k_2),
$
for various pairs of thresholds, $a$ and $b$, and  selecting the values for which the maxima (with respect to $A$) of the above error probabilities  match the nominal levels, $\alpha$ and $\beta$, respectively. 

As in the previous section, this task is facilitated when the multiple testing problem has a special structure. Specifically, when  it is \textit{symmetric} 
and  the user-specified parameters are selected so that  $\alpha = \beta$ and $k_1 = k_2$, i.e., when condition \eqref{sym_hom_GF} holds,   we can select without  any loss of generality the thresholds to be equal ($a =b$). Moreover, if the  multiple testing problem is \textit{homogeneous}, the discussion following  Theorem \ref{KM_err_control} also applies here.



\subsection{Asymptotic optimality}  \label{subsec:asymp_opt_leap}

For any  $B \subset [J]$ and 
 $1\leq \ell \leq u \leq J$,
  we denote by
$$
\mathcal{I}_1^{(1)}(B) \leq  \ldots \leq \mathcal{I}_1^{(|B|)}(B)
$$
the  increasingly ordered sequence of $\MI_1^{j}, j \in B$,  and by 
$$
\mathcal{I}_0^{(1)}(B) \leq  \ldots \leq \mathcal{I}_0^{(|B|)}(B)
$$
the increasingly ordered sequence of $\MI_0^{j}, j \in B$, and  we set
\begin{align*}
\mathcal{D}_1 (B; \ell, u) &:= \sum_{j=\ell}^{u} \mathcal{I}_1^{(j)}(B), \quad   \text{where} \quad 
\mathcal{I}_1^{(j)}(B) = \infty\quad \text{for} \quad j > |B| ,\\
\mathcal{D}_0 (B; \ell, u) &:= \sum_{j=\ell}^{u} \mathcal{I}_0^{(j)}(B),  \quad   \text{where} \quad 
\mathcal{I}_0^{(j)}(B) = \infty\quad \text{for} \quad j > |B|.
\end{align*}

The following lemma provides an asymptotic upper bound on the expected sample size of the stopping times that compose the stopping time of the \procGF. 

\begin{lemma} \label{taus_upper_bound_leap}
Assume \eqref{cond_for_asym_opt} holds.
For any $A \subset [J]$ we have as $a, b \to \infty$ 
\begin{align*}
\Exp_A[\widehat{\tau}_\ell] &\leq \max\left\{ \frac{b(1 + o(1)) }{\mathcal{D}_1(A; 1, k_1 -\ell)} \;  , \;  \frac{a(1 + o(1))}{\mathcal{D}_0(A^c; \ell+1 ,\ell+ k_2)} \right\} ,
\;  0 \leq \ell <k_{1},  \\
\Exp_A[\widecheck{\tau}_\ell] &\leq  \max\left\{ \frac{b(1 + o(1))}{\mathcal{D}_1(A; \ell+1,\ell+k_1)} \,  , \,  \frac{a(1 + o(1))}{\mathcal{D}_0(A^c; 1 , k_2 - \ell ) } \right\}, \;  0 \leq \ell <k_{2}. 
\end{align*}
\end{lemma}

\begin{proof}
The  proof can be found in Appendix \ref{proof:taus_upper_bound}.
\end{proof}

If thresholds are selected  according to \eqref{leap_thresholds}, then the upper bounds in the previous lemma are equal (to a first-order asymptotic approximation) to
\begin{align*} 
\widehat{L}_A (\ell; \alpha, \beta) &:=  \max \left\{ 
\frac{|\log \alpha| }{\mathcal{D}_1(A; 1, k_1 -\ell)} \; , \;  \frac{|\log \beta|}{\mathcal{D}_0(A^c; \ell+1,  \ell+ k_2)} \right\} \; \text{ for }  \ell < k_1   , \\
\widecheck{L}_A (\ell; \alpha, \beta) &:=\max \left\{ 
\frac{|\log \alpha|}{\mathcal{D}_1(A;\ell+1, \ell+k_1)}  \; , \; \frac{|\log \beta|}{\mathcal{D}_0(A^c;1 ,k_2 - \ell )}\right\} \; \text{ for }  \ell < k_2,
\end{align*}
and from the definition of Leap rule in \eqref{leap_def}  it follows that  as $\alpha, \beta \rightarrow 0$ we have 
$
\Exp_A[T_L]  \leq  L_A(k_1,k_2,\alpha, \beta) \, (1 + o(1)), 
$ 
where 
\begin{align}\label{fund_GF_limit}
L_A(k_1,k_2,\alpha, \beta):= \min \left\{ \min_{0 \leq \ell <k_1}  \widehat{L}_A (\ell; \alpha, \beta)  \; ,  \; \min_{0 \leq \ell <k_2} \widecheck{L}_A (\ell; \alpha, \beta) \right\}.
\end{align}
In the next theorem we show that it is not possible to achieve a smaller ESS, to a first-order asymptotic approximation as $\alpha, \beta \rightarrow 0$, proving in this way the  asymptotic optimality of the Leap rule.

\begin{theorem} \label{main}
Assume \eqref{cond_for_asym_opt} holds and that
the thresholds in the \procGF~are selected such that $\delta_L \in \Delta_{k_1,k_2}(\alpha,\beta)$ and $a \sim |\log(\beta)|, b \sim |\log(\alpha)|$, e.g. according to \eqref{leap_thresholds}.  Then, for any $A \subset [J]$  we have as  $\alpha, \beta \to 0$,
\begin{align*}
\Exp_A \left[ T_L \right]
   \;\sim \; L_A(k_1,k_2,\alpha, \beta) \;\sim\; N_A^*(k_1,k_2,\alpha,\beta).
\end{align*}

\end{theorem}

\begin{proof}
In view of the discussion prior to the theorem, it suffices to show that for any $A \subset [J]$  we have 
as $\alpha, \beta \to 0$ that 
$$
N_A^*(k_1,k_2,\alpha,\beta) \geq  L_A(k_1,k_2,\alpha, \beta) \, (1 - o(1)).
$$ 
For the proof of this asymptotic lower bound 
we employ  similar ideas as in the proof of Theorem \ref{KM_lower_bound} in the previous section. The  change-of-measure argument is more complicated now, due to the interplay of the two kinds of error. We carry out the proof in Appendix \ref{proof_of_GF_lower_bound}.
\end{proof}


\begin{remark}
 When $k_1=k_2=1$, the  asymptotic optimality of the  \procInt~was established in \cite{song2017asymptotically} only in the i.i.d. case. Since the Leap rule coincides with the Intersection rule when $k_1=k_2=1$, 
 Theorem \ref{main} generalizes this result in \cite{song2017asymptotically} beyond the i.i.d. case.
\end{remark}

We motivated the  Leap rule  by the inadequacy of the asymmetric Sum-Intersection rule, $\delta_0$, in the case of \textit{asymmetric and/or inhomogeneous} testing problems. In the following corollary we show that $\delta_0$ is  asymptotically optimal  when
(i)  condition \eqref{sym_hom_KM} holds, which is the case when the multiple testing problem is symmetric and homogeneous,  and also (ii)  the user-specified parameters are   selected in  a symmetric way,  i.e., when  \eqref{sym_hom_GF} holds.  In the same setup we establish the asymptotic optimality of the \procInt, $\delta_I$, defined in \eqref{intersection_rule}.

\begin{corollary} \label{delta_0_cor}
Suppose \eqref{cond_for_asym_opt} and  \eqref{sym_hom_KM}--\eqref{sym_hom_GF} hold and consider the  asymmetric Sum-Intersection rule  $\delta_0(b,b)$ with $b=b_\alpha$ and the Intersection rule $\delta_I(b,b)$ with $b=b_\alpha/k_1$, where $b_a$ is defined in \eqref{OS_threshold} with $k=k_1$. 
Then  $\delta_0, \delta_I  \in  \Delta_{k_1,k_1}(\alpha, \alpha)$, and 
for any $A \subset [J]$ we have  as $\alpha  \to 0$ that 
\begin{align*}
\Exp_A \left[ \tau_0 \right] \sim 
\Exp_A \left[ T_I \right] \sim
\frac{|\log(\alpha)| }{k_1 \MI}    \sim N_A^*(k_1,k_1,\alpha,\alpha).
\end{align*}
\end{corollary}
\begin{proof}
The proof can be found in Appendix \ref{proof_of_delta_0_cor}.
\end{proof}

\begin{remark}
In Section \ref{sec:asym_case} we will  illustrate numerically that when condition \eqref{sym_hom_KM} is violated,  both $\delta_0$ and $\delta_I$ fail to be asymptotically optimal.
\end{remark}


\subsection{Fixed-sample size rules}\label{subsec:GF_fix_comp}
We now  focus on the i.i.d. case \eqref{iid_case} and consider procedures that stop at a \textit{deterministic} time, which is selected to control the generalized familywise error rates. 

For simplicity of presentation, we restrict ourselves to  \textit{homogeneous} testing problems, i.e.,  there are densities $f_0$ and $f_1$ such that 
\begin{equation}\label{homo_iid_case}
f_0^j = f_0, \quad f_1^j = f_1
\; \text{ for every } j \in [J].
\end{equation}
This assumption allows us to omit the dependence on the stream index $j$ and write $\cI_0 := \cI_0^j$, $\cI_1 := \cI_1^j$ and $\Phi:= \Phi^j$, where $\Phi^j$ is defined in \eqref{convex_conjugate}. Moreover,   without loss of generality,
we  apply the  MNP rule \eqref{NP_fixed_sample} with the same  threshold for each stream. 

We further assume that  user-specified parameters  are selected as follows
\begin{equation} \label{partial_sym}
k_1 = k_2, \;\;\; \alpha = \beta^d\; \text{ for some } d > 0,
\end{equation}
and  that  for each $d>0$  there exists some $h_d \in (-\cI_0, \cI_1)$ such that
\begin{equation} \label{h_hat_d}
\Phi(h_d)/ d = \Phi(h_d) - h_d  .
\end{equation}
When $d=1$, condition \eqref{partial_sym} reduces to \eqref{sym_hom_GF} and $h_d$ is equal to $0$.
However, when $d\neq 1$, we allow for an asymmetric treatment of the two kinds of error. 

\begin{theorem}\label{GF_comp_fix_sample_rule}
Consider the multiple testing problem \eqref{homo_iid_case} and assume that the Kullback-Leibler  numbers in \eqref{KL_numbers} are positive and finite. 
Further, assume that \eqref{partial_sym} and \eqref{h_hat_d} hold. Then as $\beta \to 0$,
\begin{align*}
\frac{d\, (1-o(1)) }{(2k_1-1) \Phi({h}_d)}&\leq 
\frac{n^{*}(k_1,k_1,\beta^d,\beta)}{|\log(\beta)|} \leq
\frac{\widehat{n}_{NP}(k_1,k_1,\beta^d,\beta)}{|\log(\beta)|} \sim 
\frac{d}{k_1 \Phi({h}_d)}.
\end{align*}
\end{theorem}
\begin{proof}
The proof is similar to that of Theorem \ref{KM_comp_fix_sample_rule}, but   requires a \textit{generalization} of Chernoff's lemma \cite[Corollary 3.4.6]{dembo1998large} to account for the asymmetry of the two kinds of error.  This generalization is presented in  Lemma \ref{generalized_chernoff_lemma} and more  details can be found in Appendix \ref{proof_of_GF_comp_fix_sample_rule}. 
\end{proof}

Theorem \ref{GF_comp_fix_sample_rule}, in conjunction with Theorem \ref{main}, allows us to quantify  the  performance loss that is induced by stopping at a deterministic time.   Specifically,  in the case of testing  normal means (Example \ref{Gaussian_mean_test}),
by \eqref{normal_related_quantity}   we have 
 $\cI=\cI_1=\cI_0$ and  for any $d \geq 1$ 
$$
{h}_d = \frac{\sqrt{d}-1}{\sqrt{d}+1} \, \cI, \quad
\Phi({h}_d) = \frac{d}{(1+\sqrt{d})^2} \, \cI.
$$
Thus,  by Theorem \ref{main} it follows  that as $\beta \to 0$,
$$
N_A^*(k_1,k_1, \beta^d,\beta)  \; \leq \; \widehat{L}_A(0;
\beta^d,\beta)
\; = \;  
\begin{cases}
 \frac{|\log(\beta)|}{k_1 \cI} ,  \text{ if } |A| < k_1 \\[5pt]
 \frac{d|\log(\beta)|}{k_1 \cI} ,  \text{ if } |A| \geq k_1
\end{cases}.
$$
When in particular  $d = 1$, i.e., $\alpha=\beta$, for any $A \subset [J]$ we have
\begin{align*}
 2 \, N_A^*(k_1,k_1,\beta,\beta) (1-o(1))  &\leq 
n^{*}(k_1,k_1,\beta,\beta)  \\
&\leq  \widehat{n}_{NP}(k_1,k_1,\beta,\beta) \sim  4 \, N_A^*(k_1,k_1,\beta,\beta),
\end{align*}
which agrees with the corresponding findings in Subsection \ref{subsec: KM_fix_comp}.


%% file: a_FWER_Simulation.tex
\section{Simulations for generalized familywise error rates} \label{sec:simulation_GF}
In this section we present two simulation studies that complement our asymptotic optimality theory
in Section \ref{sec:GF} for procedures that control generalized familywise error rates.  In the first study we compare the  Leap rule \eqref{leap_def},   the Intersection rule \eqref{intersection_rule} and the asymmetric Sum-Intersection rule  \eqref{tau_0_def},   in a \textit{symmetric and homogeneous} setup where conditions \eqref{sym_hom_KM} and \eqref{sym_hom_GF} hold and  all three procedures are asymptotically optimal.   
In the second  study we compare the same procedures when condition \eqref{sym_hom_KM} is  slightly violated, and only the Leap rule enjoys the asymptotic optimality property. 

In both studies  we  consider the  testing of normal means  (Example \ref{Gaussian_mean_test}), with $\sigma_j=1$ for every $j \in [J]$. This is a  \textit{symmetric} multiple testing problem,  where the Kullback-Leibler information in the $j$-th testing problem is  $\cI^j= \mu_j^2/2$. Moreover, we assume that condition \eqref{sym_hom_GF} holds, i.e.,  $\alpha = \beta$ and $k_1 = k_2$. This implies that  we can set the thresholds  in each \textit{sequential} procedure to be equal, i.e.,  $a = b$, and as a result the two types of generalized familywise  error rates will be the same.
Finally, in both studies we include the performance of the fixed-sample size  \textit{multiple Neyman-Pearson} (MNP) rule \eqref{NP_fixed_sample}, for which the  choice of thresholds depends crucially on whether the problem is homogeneous or not. 

In what follows, the ``error probability (Err)'' is the  generalized familywise  error rate of false positives~\eqref{gfam}, i.e., the \textit{maximum} probability of  $k_1$ false positives, with the maximum  taken over all signal configurations. Thus, Err \textit{does not} depend on the true subset of signals $A \subset [J]$.

\subsection{Homogeneous case}
In the first simulation study we set $\mu_j = 0.25$ for each  $j \in [J]$. In this homogeneous setup,   the expected sample size (ESS)  of all procedures under consideration depend only on the \textit{number} of signals, and we can set the thresholds in the
 MNP rule, defined in \eqref{NP_fixed_sample}, to be equal to 0. Moreover, it suffices to study the performance when the number of signals is no more than $J/2$.
We consider $J = 100$ in Fig. \ref{fig:complete_sym_iid} and $J = 20$ in Fig. \ref{fig:complete_sym_iid_J20}.



In Fig. \ref{fig:sym_a}, we fix $k_1 = 4$ and evaluate the  ESS of  the \procGF~for four different cases regarding the number of signals. We see that, for any given Err, the smallest possible  ESS  is achieved in the boundary case of no signals ($|A| = 0$). This is because some components in the \procGF~only have one condition to be satisfied in the boundary cases (e.g. $\widehat{\tau}_2$ in Fig. \ref{fig:leap_proc_vis}).

In Fig. \ref{fig:sym_b}, we fix the number of signals to be $|A|= 50$ and evaluate the \procGF~for different values of $k_1$. We observe that there are significant savings in the ESS as $k_1$ increases and more mistakes are tolerated.

In Fig. \ref{fig:sym_c} and \ref{fig:sym_d}, we fix $k_1 = 4$ and compare the four rules for  $|A| = 0$ and $50$, respectively. In this \textit{symmetric and homogeneous} setup, where \eqref{sym_hom_KM} and \eqref{sym_hom_GF} both hold, we have shown that all three sequential procedures are asymptotically optimal. Our simulations suggest that in practice the \procGF~works better when the number of signals,  $|A|$, is close to $0$ or $J$, but may perform slightly worse than the asymmetric Sum-Intersection rule, $\delta_0$,  when $|A|$ is close to $J/2$.

In Fig. \ref{fig:sym_c}, \ref{fig:sym_d} and \ref{fig:sym_a_J20}, we also compare the performance of
the Leap rule with the MNP rule. Further, in Fig. \ref{fig:sym_e}, \ref{fig:sym_f}, \ref{fig:sym_b_J20} and \ref{fig:sym_c_J20}, we show the histogram of the stopping time of the Leap rule at particular error levels. From these figures we can see that the best-case scenario for  the MNP is when both the number of hypotheses, $J$, and the error probabilities,  $\text{Err}$, are large.
Note that this does not contradict our asymptotic analysis,  where  $J$ is fixed and we let $\text{Err}$ go to 0. 

\begin{figure}[htbp!]
\centering
    \makebox[\linewidth]{
\subfloat[\procGF: $k_1=4$]{\label{fig:sym_a}
\includegraphics[width=0.35\linewidth]{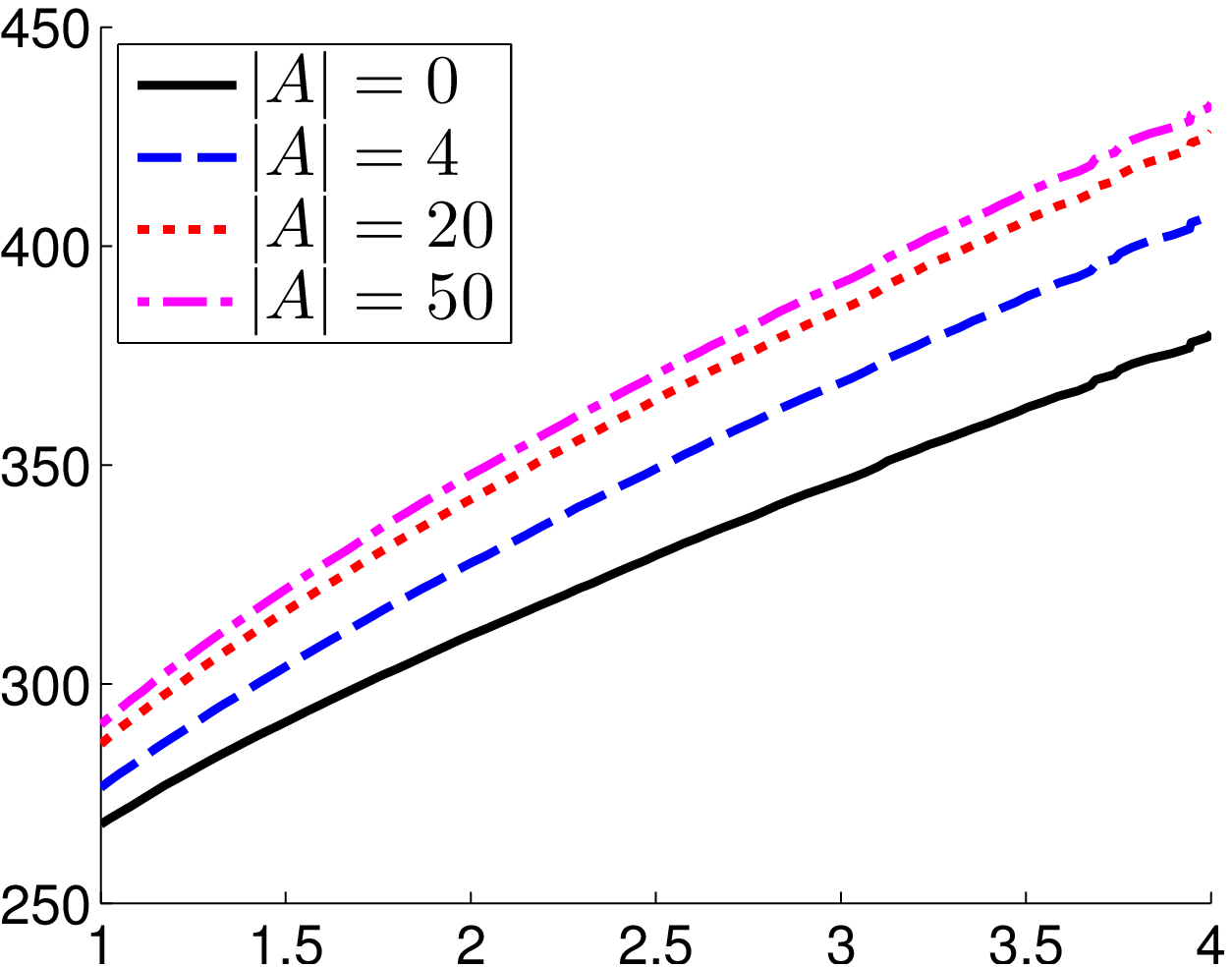}  
}
\subfloat[\procGF: $|A|=50$]{\label{fig:sym_b}
\includegraphics[width=0.35\linewidth]{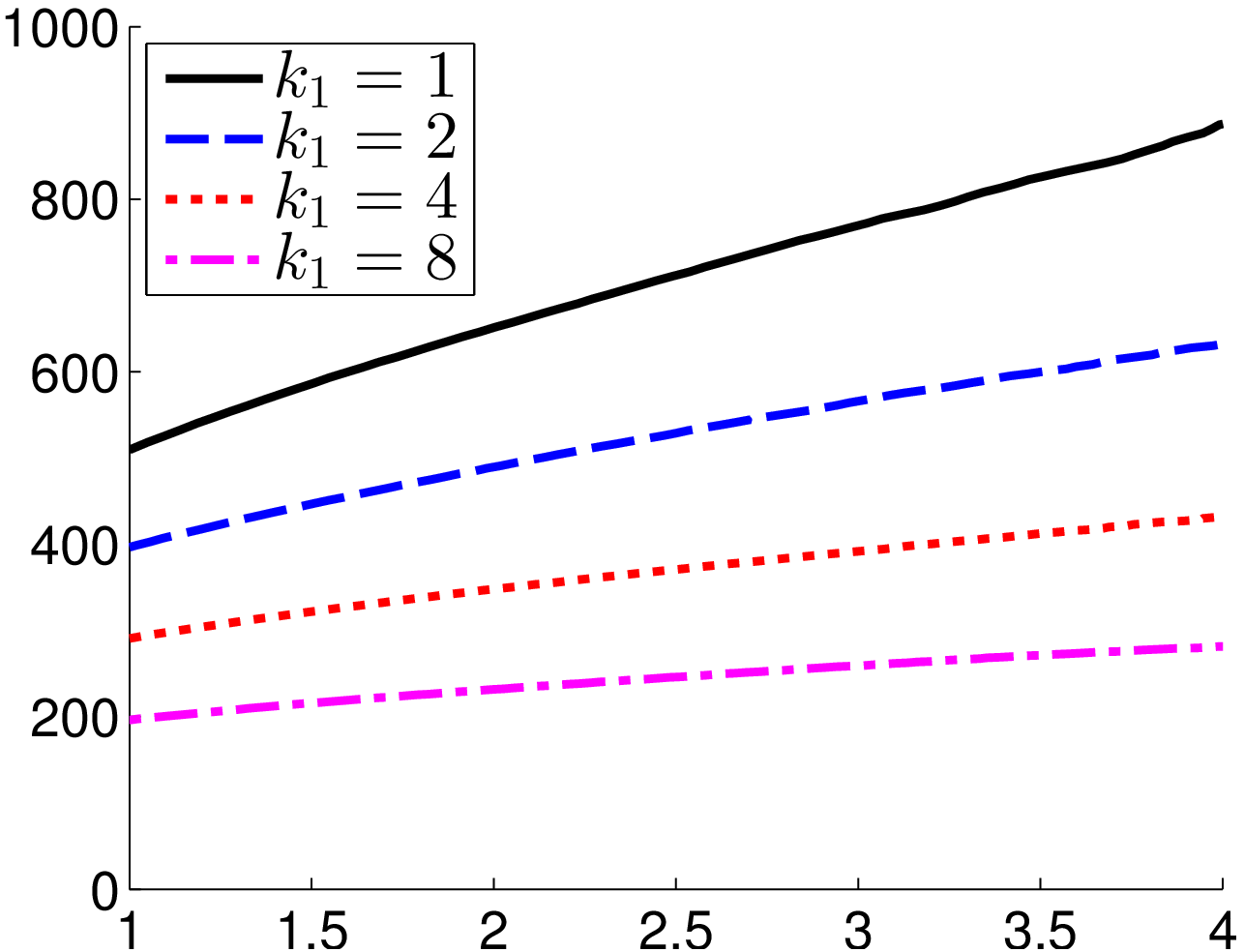}  
}
\subfloat[$k_1 = 4, |A| = 0$]{\label{fig:sym_c}
\includegraphics[width=0.35\linewidth]{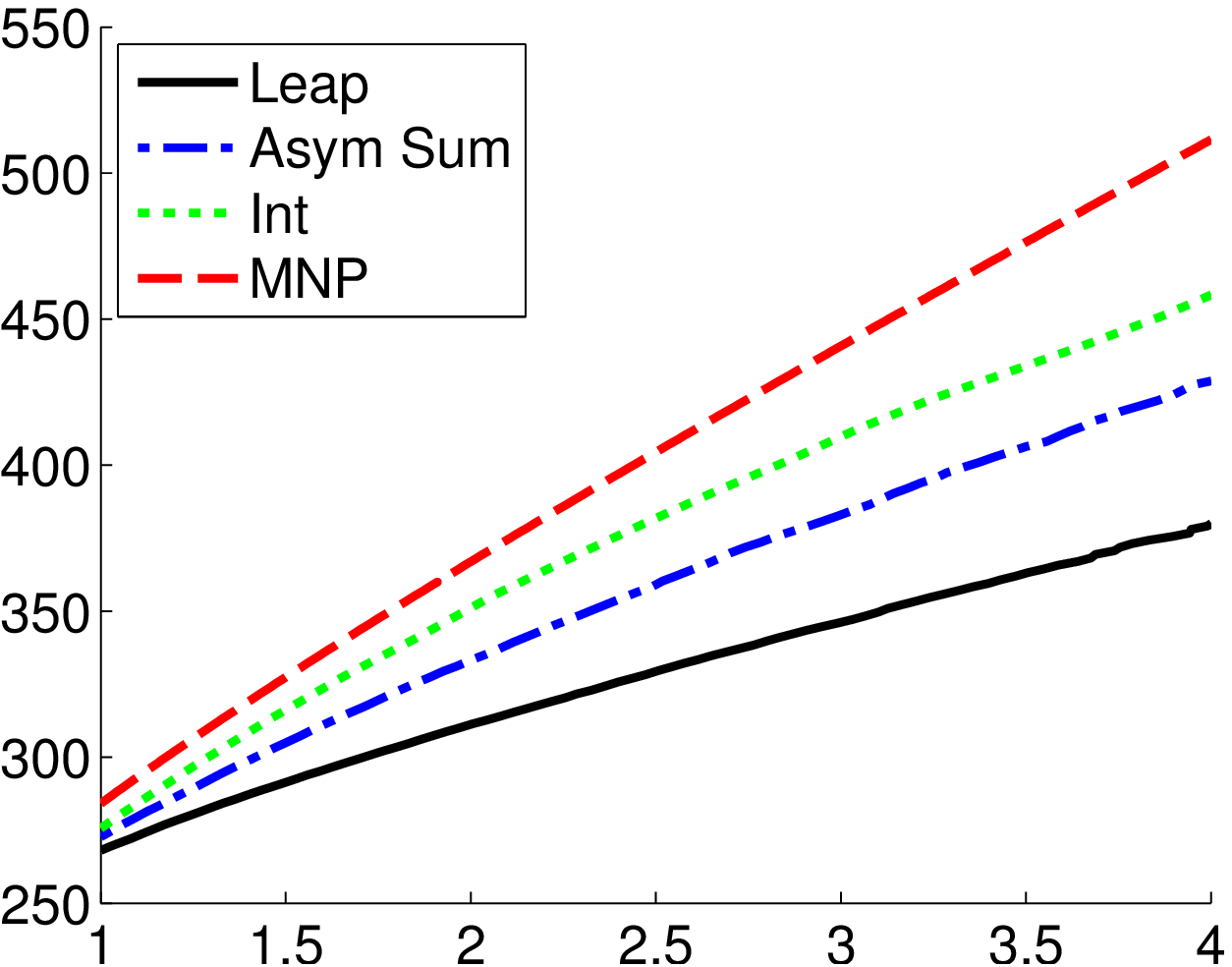}  
}
}\\
    \makebox[\linewidth]{
\subfloat[$k_1 = 4, |A| = 50$]{\label{fig:sym_d}
\includegraphics[width=0.35\linewidth]{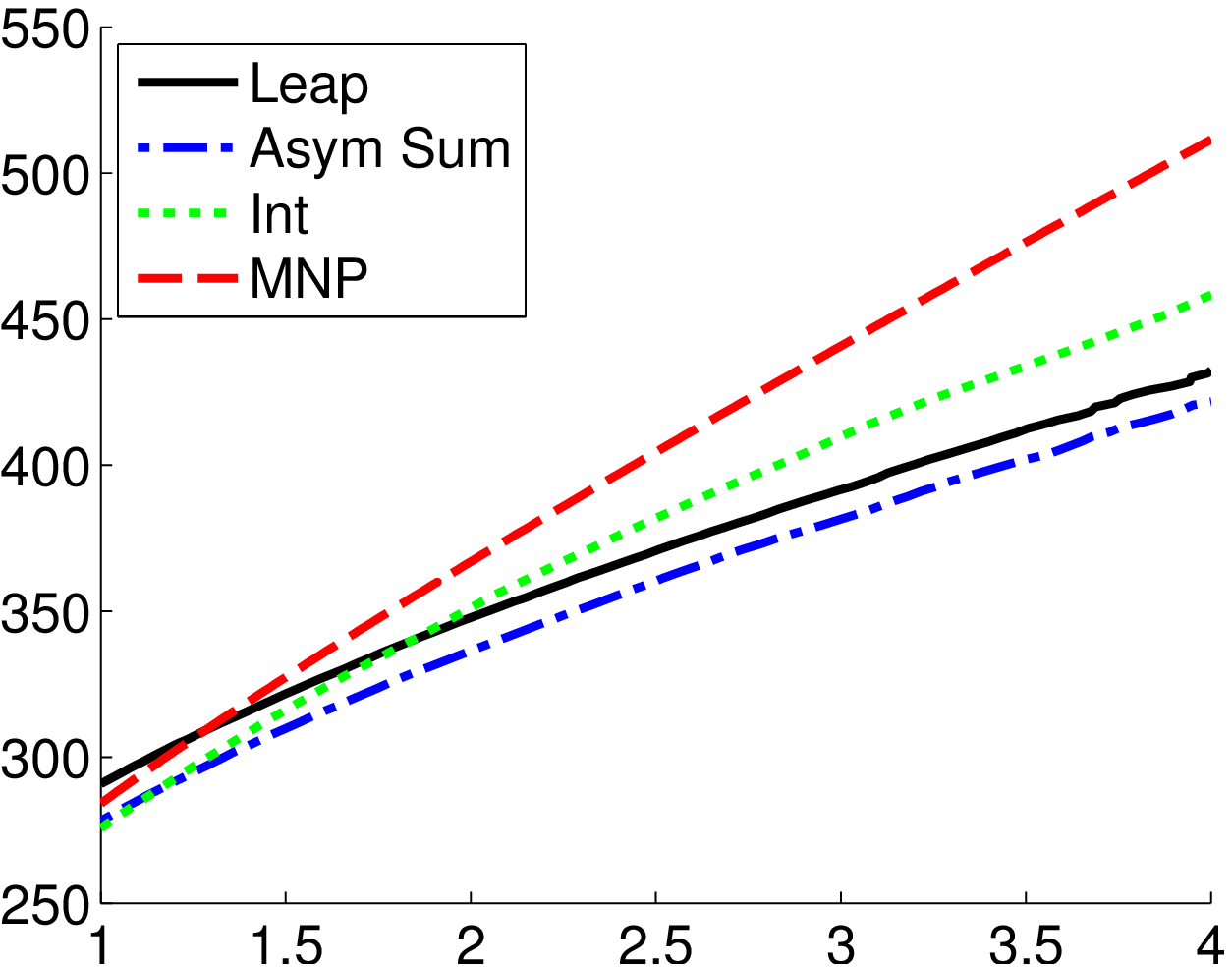}  
}
\subfloat[$k_1 = 4, |A| = 0, \text{Err} = 5\%$]{\label{fig:sym_e}
\includegraphics[width=0.35\linewidth]{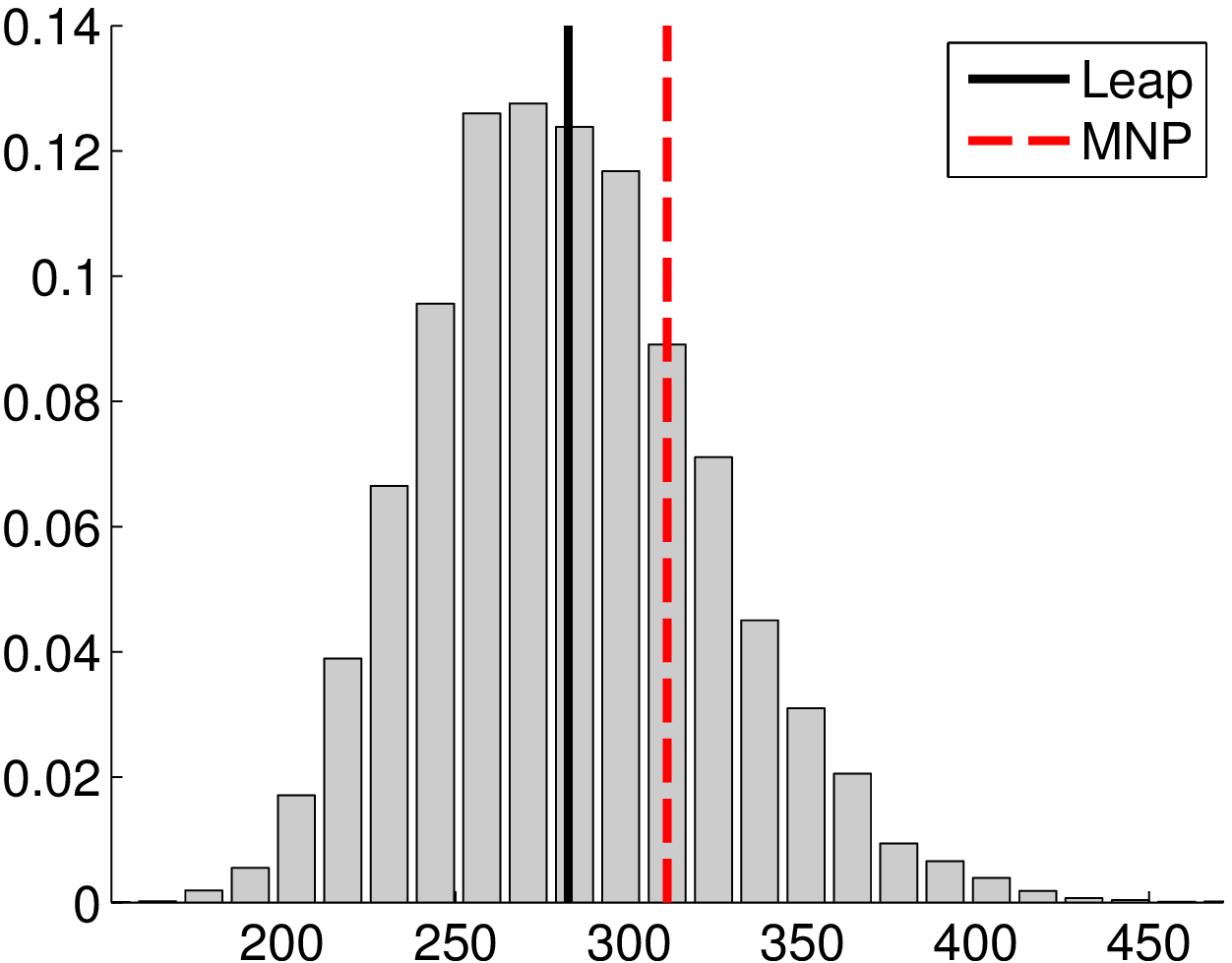}  
}
\subfloat[$k_1 = 4, |A| = 50, \text{Err} = 5\%$]{\label{fig:sym_f}
\includegraphics[width=0.35\linewidth]{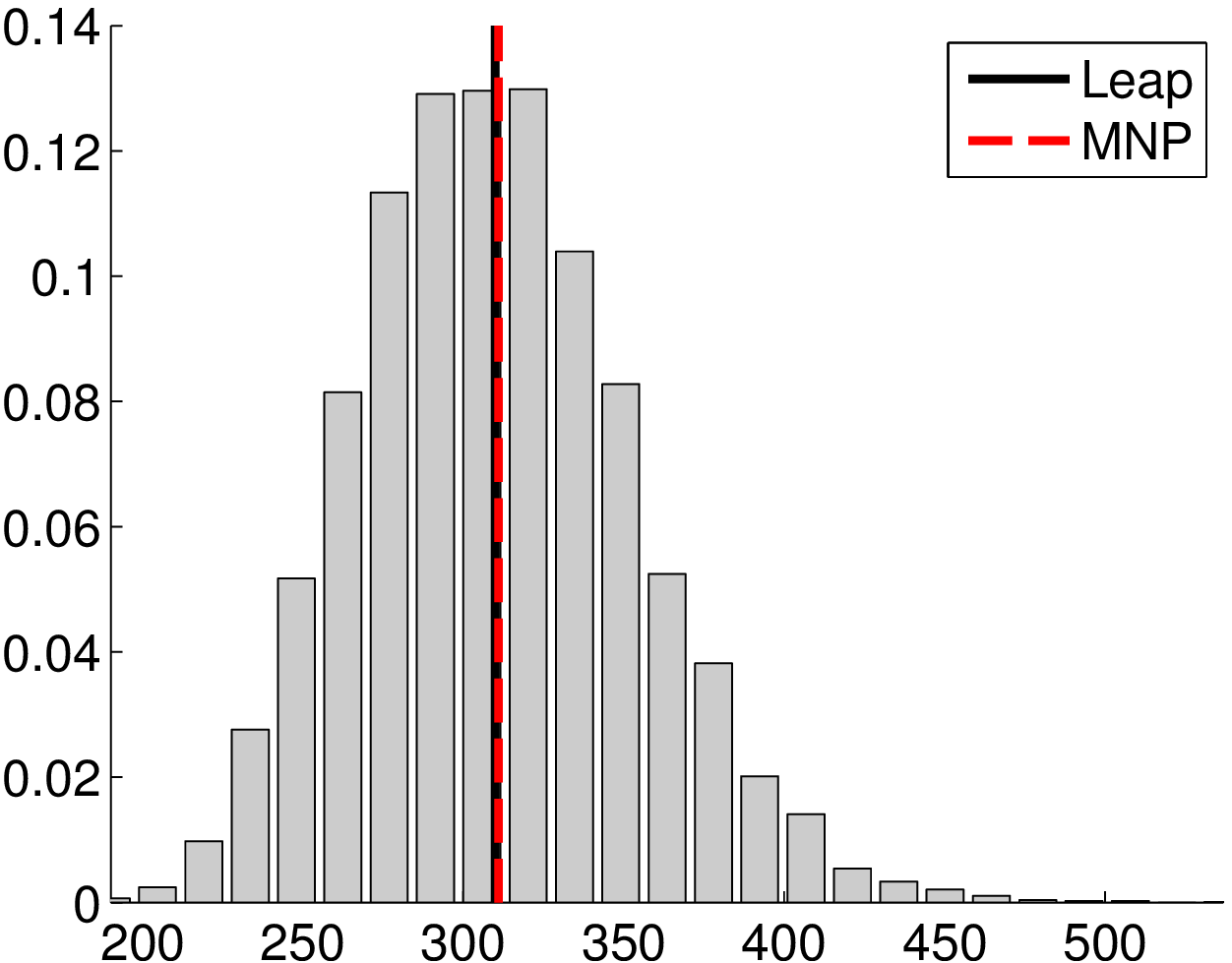}  
}
}
\caption{Homogeneous case: $J = 100, k_1 = k_2$. 
In (a)-(d), the x-axis  is  $|\log_{10}(\text{Err})|$ and the y-axis is the ESS under $\Pro_A$. In (e) and (f) are the histogram of the 
stopping time of the Leap rule with $\text{Err} = 5\%$.
}
\label{fig:complete_sym_iid}
\vspace{-0.3cm}
\end{figure}

\begin{figure}[htbp!]
\centering
    \makebox[\linewidth]{
\subfloat[$k_1 = 2$]{\label{fig:sym_a_J20}
\includegraphics[width=0.35\linewidth]{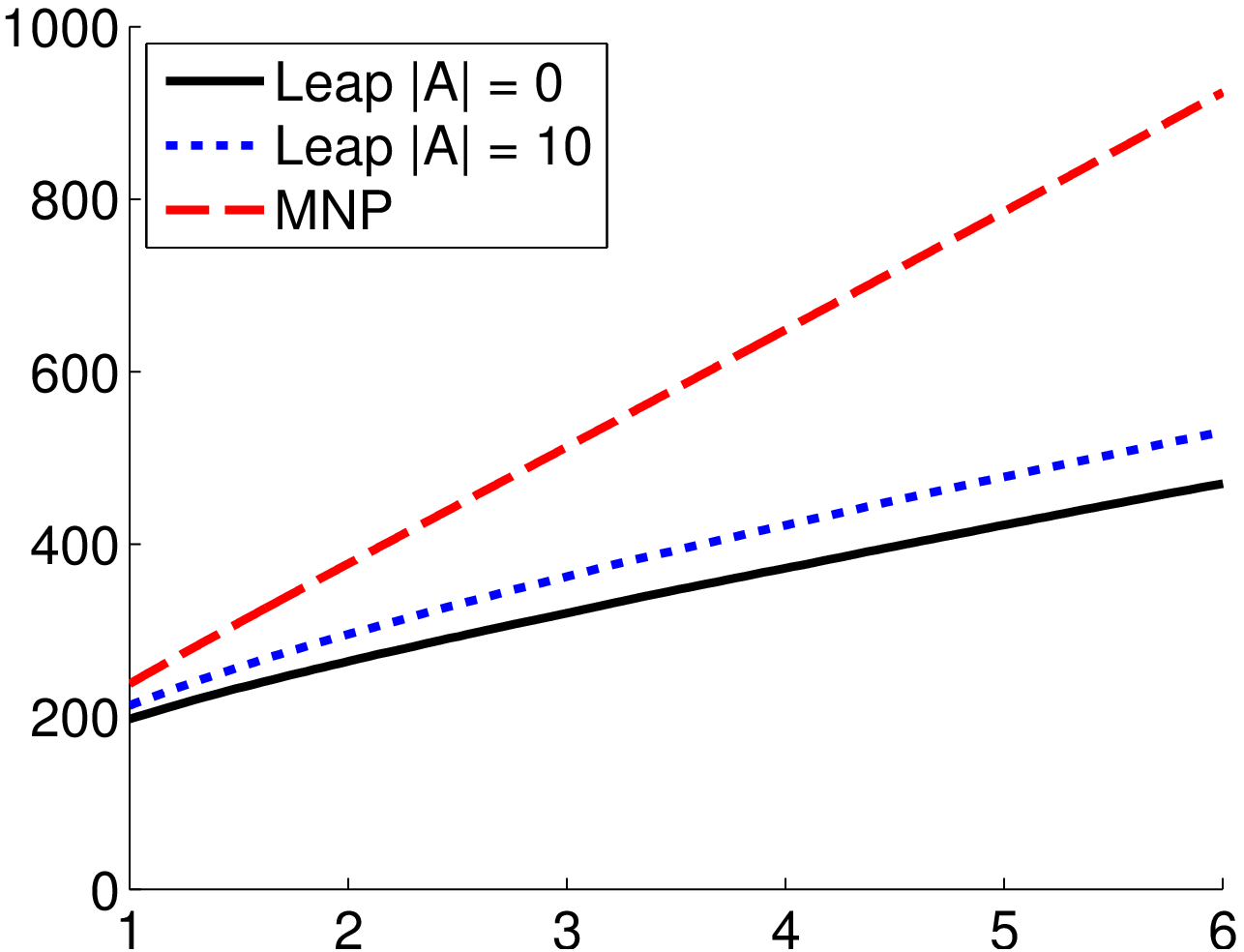}  
}
\subfloat[$k_1 = 2, |A| = 10, \text{Err} = 5\%$]{\label{fig:sym_b_J20}
\includegraphics[width=0.35\linewidth]{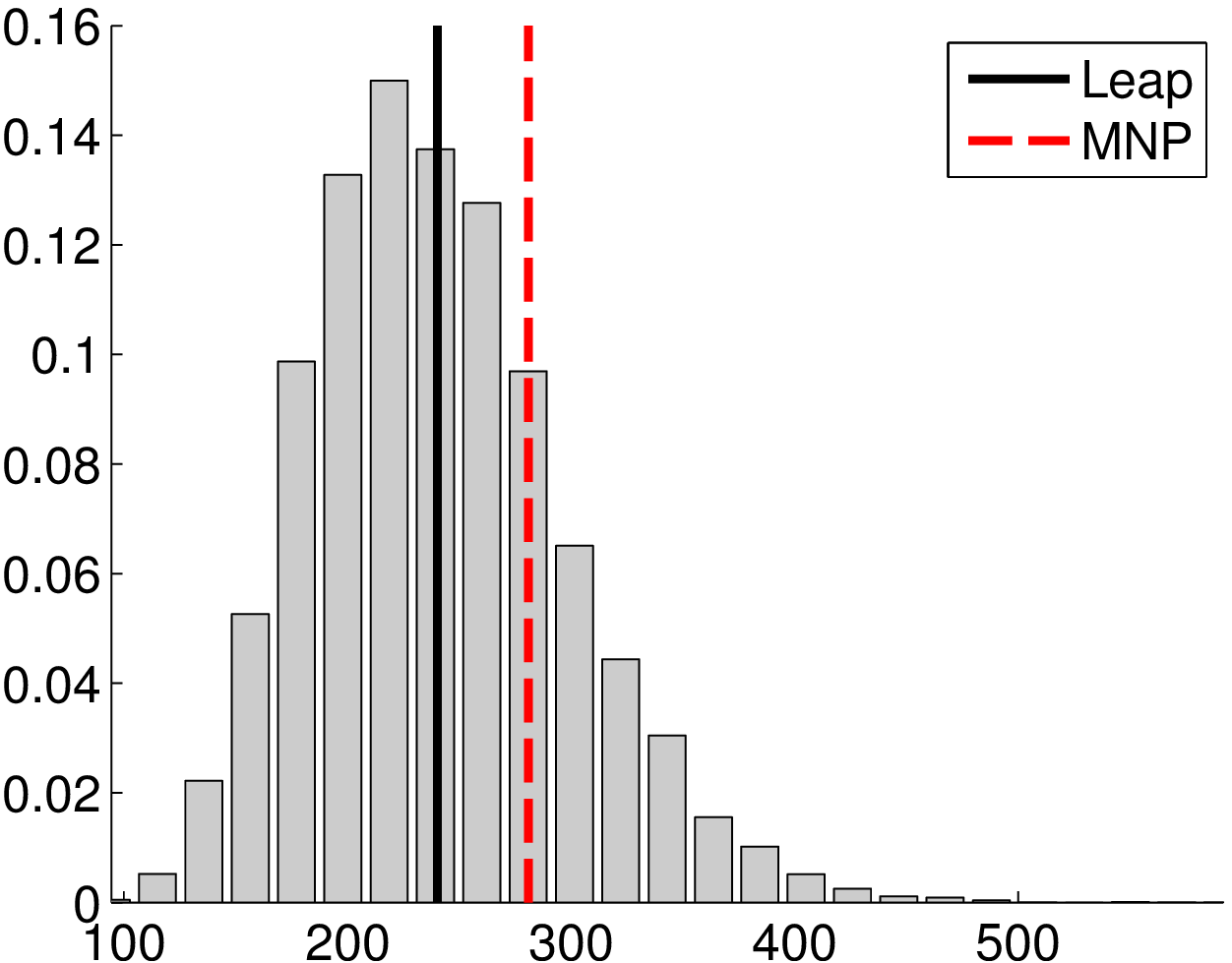}  
}
\subfloat[$k_1 = 2, |A| = 10, \text{Err} = 1\%$]{\label{fig:sym_c_J20}
\includegraphics[width=0.35\linewidth]{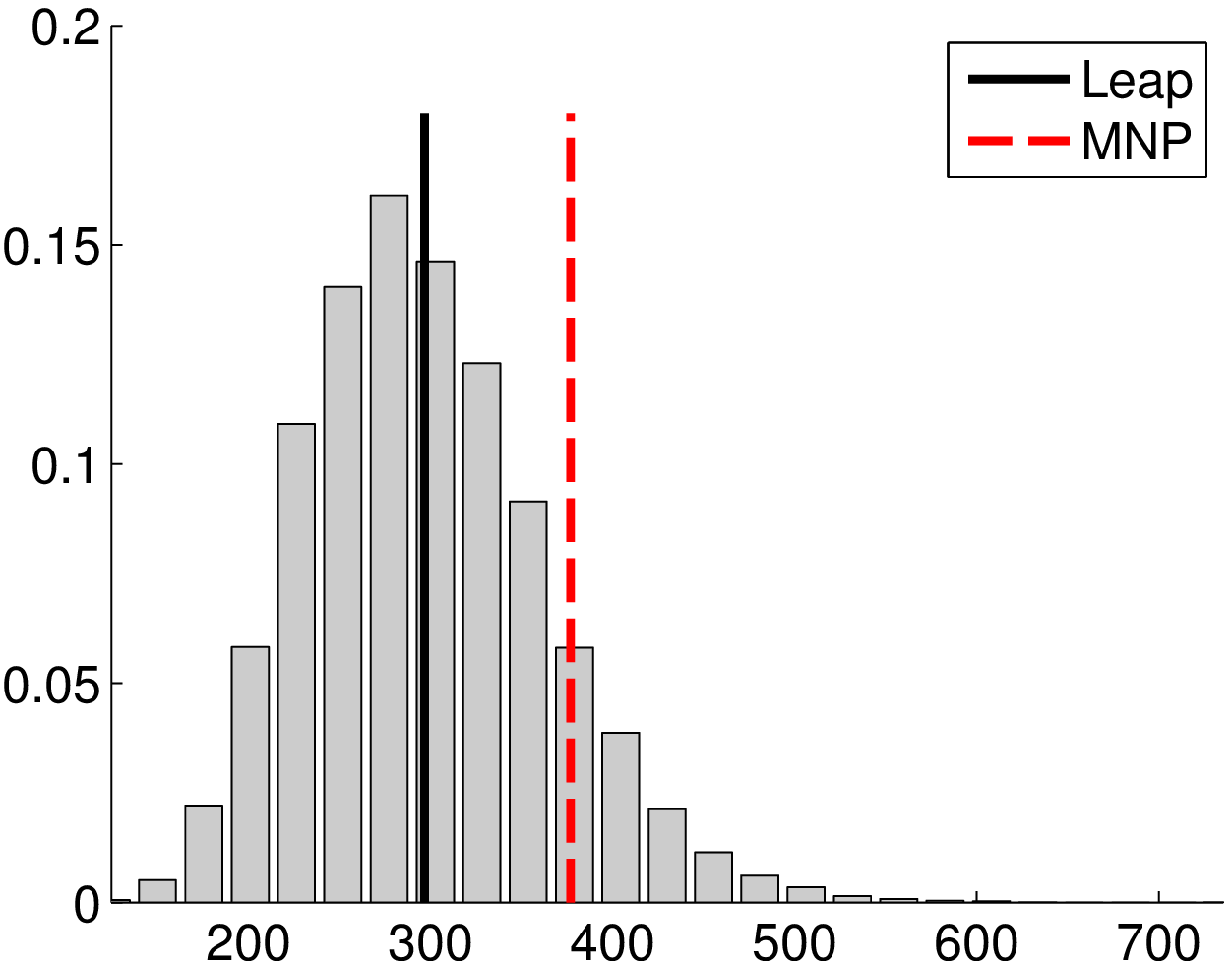}  
}
}
\caption{Homogeneous case: $J = 20, k_1 = 2$. 
In (a), the x-axis  is  $|\log_{10}(\text{Err})|$ and the y-axis is the ESS under $\Pro_A$. In (b) and (c) are the histogram  of the 
stopping time of the Leap rule with $\text{Err} = 5\%$ and $1\%$.
}
\label{fig:complete_sym_iid_J20}
\vspace{-0.3cm}
\end{figure}

\subsection{Non-homogenous case}\label{sec:asym_case}
In the second simulation study we set $J =10$, 
 $\mu_j = 1/6$, $j=1,2$, $\mu_j = 1/2$, $j\geq 3$, so that  the first two hypotheses are much harder than others. Specifically, 
$\MI^j = 1/72$ for $j=1,2$, and $\MI^j = 1/8$ for $j \geq 3$. 


When the true subset of signals is $A^* = \{6,\cdots,10\}$, the optimal asymptotic performance, \eqref{fund_GF_limit},  
is equal to  $8|\log(\text{Err})|$.
In Fig. \ref{fig:asym_a}, we plot the ESS against $|\log_{10}(\text{Err})|$, and  the ratio of ESS over $8|\log(\text{Err})|$ in Fig. \ref{fig:asym_b}. For the (asymptotically optimal) \procGF, this  ratio tends to 1 as $\alpha \to 0$. In contrast, the other rules have a different ``slope'' from the \procGF~in Fig. \ref{fig:asym_a}, which indicates that they fail to be asymptotically optimal in this context. 

Finally, we note that in such a non-homogeneous setup, the choice of  thresholds for the MNP rule \eqref{NP_fixed_sample} is not obvious. We found that instead of setting  $h_j = 0$ for every $j \in [J]$, it is much more efficient to take advantage of the flexibility of generalized familywise error rates, as we did in the  construction of the  Leap rule in Subsection \ref{subsec:leap_rule}, and set 
$h_1 = -\infty$, $h_2 = \infty$ and $h_j = 0$ for $j \geq 3$.  
This choice ``gives up'' the first two ``difficult'' streams by always rejecting the null in the first one and accepting it in the second.  The error constraints can then still be met as long as we do not make any mistakes in the remaining ``easy" streams. In fact, we see that while the MNP rule behaves significantly worse than the asymptotically optimal Leap rule, it performs better than the Intersection rule,  which  requires strong  evidence from each individual stream in order to stop. 

\begin{figure}[htbp!]
\subfloat[ESS under $\Pro_{A^*}$]{\label{fig:asym_a}
\includegraphics[width=0.4\linewidth]{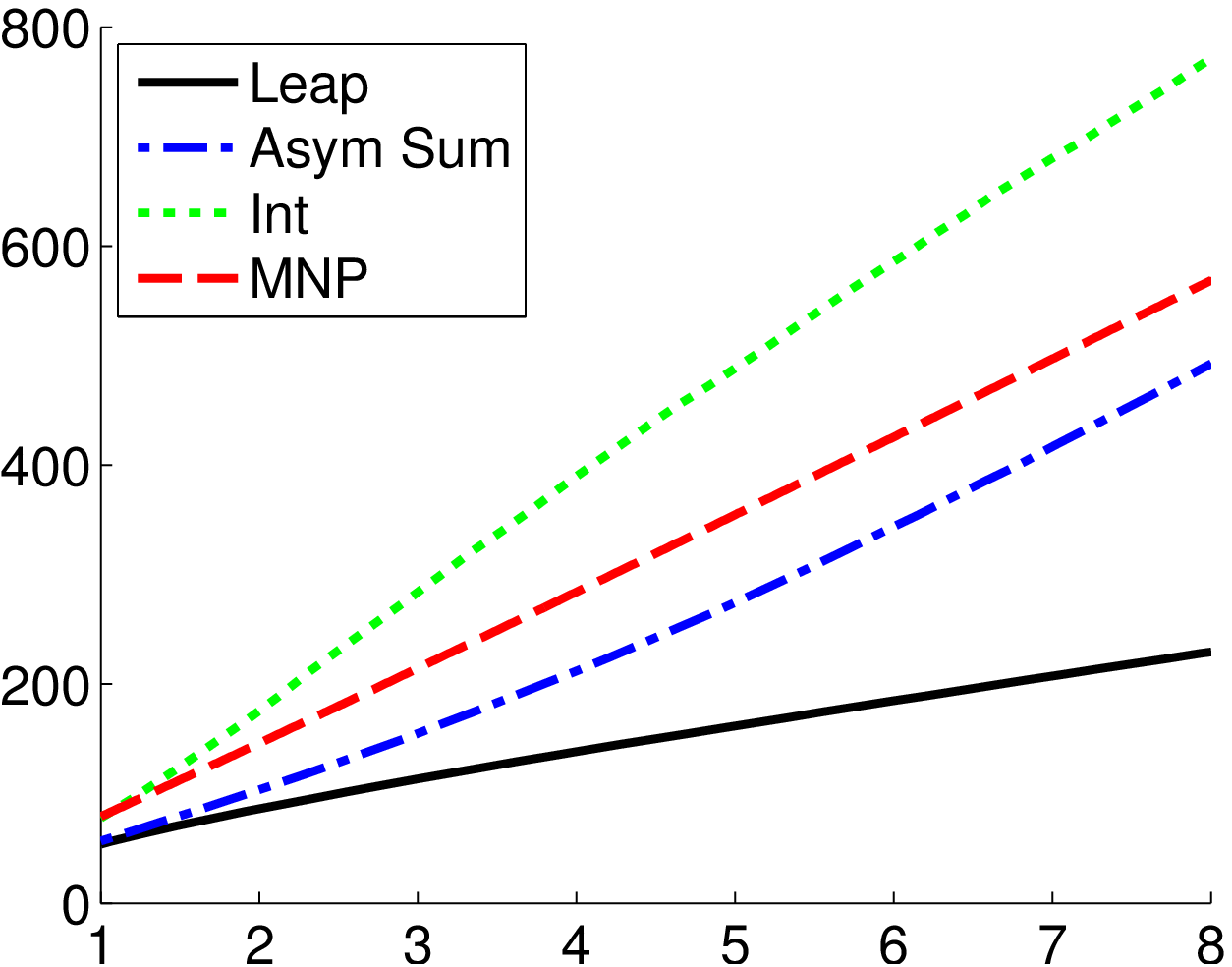} 
}\hspace{1cm}
\subfloat[Normalized by $8|\log(\text{Err})|$]{\label{fig:asym_b}
\includegraphics[width=0.4\linewidth]{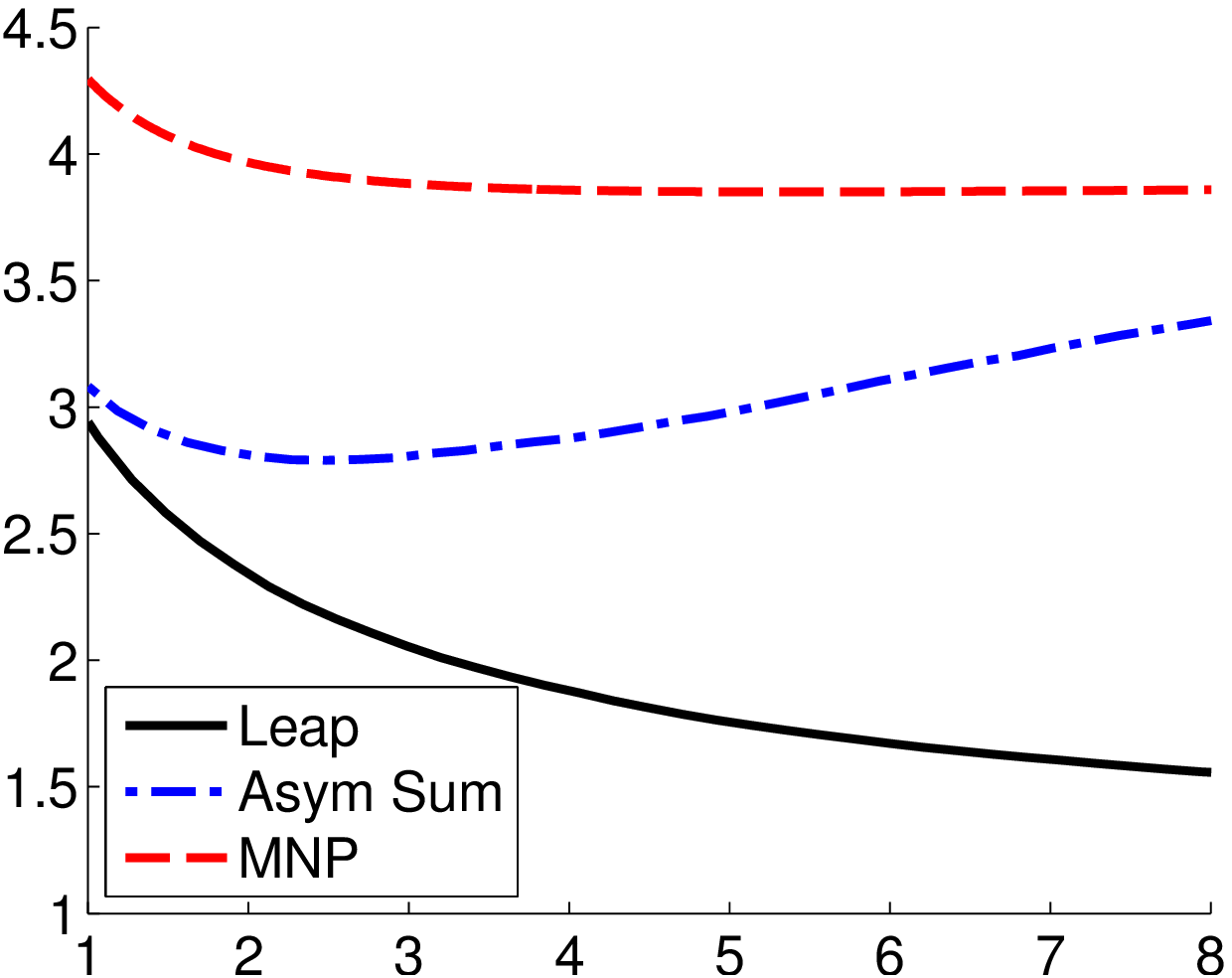} 
} 
\caption{Non-homogeneous case: $J = 10, k_1 = k_2 =2 ,A^* = \{6,\cdots,10\}$. 
The x-axis in both graphs  is  $|\log_{10}(\text{Err})|$. 
The y-axis in (a) is the ESS under $\Pro_{A^*}$, and in (b)
is the ratio of the ESS over $8|\log(\text{Err})|$.
}
\label{fig:asym_iid}
\vspace{-0.3cm}
\end{figure}

%% file: a_Composite.tex
\section{Extension to composite hypotheses} \label{sec:composite}
We now extend the setup  introduced in Section \ref{sec:formulation},  allowing both the null and the alternative hypothesis  in each local testing problem  to  be composite.  Thus,  for each $j \in [J]$, the distribution of $X^j$,  the sequence of  observations in the $j$-{th} stream,  is now parametrized by   $\theta^j \in \Theta^j$, where  $\Theta^j$  is a subset of some Euclidean space, and the hypothesis testing problem in the $j$-th stream becomes
\begin{align*}
{\Hyp}_0^j\;: \theta^j \in \Theta_0^{j} \;\; \; \text{ versus } \; \;\; {\Hyp}_1^j\;: \theta^j \in \Theta_1^{j},
\end{align*}
where  $\Theta^j_0$ and $\Theta^j_1$ are two disjoint subsets of $\Theta^j$.  
When  $A \subset [J]$ is  the subset of streams in which the alternative is correct,   we  denote by $\boldsymbol{\Theta}_A$ the subset of  the parameter space $\boldsymbol{\Theta}:= \Theta^1 \times \ldots \times \Theta^J$ that is  compatible with $A$, i.e., 
$$ \boldsymbol{\Theta}_A := \{  (\theta^1, \ldots, \theta^J) \in \boldsymbol{\Theta}:  \; 
\theta^i \in \Theta_0^i, \;\; \theta^j \in \Theta_1^j \;\;\; \forall \;i \not\in A,  j \in A
\}.
$$
We denote by  $\Pro^{j}_{\theta^{j}}$ the distribution of the $j$-th stream when the value of its local parameter is $\theta^{j}$. Moreover,  we denote  by  $\Pro_{A, \butheta}$  the  underlying  probability measure  when the subset of signals is $A$ and the parameter is $\butheta= (\theta^1, \ldots, \theta^J) \in \boldsymbol{\Theta}_A$, and by $\Exp_{A,\butheta}$ the corresponding expectation.  Due  to the independence across streams, we have 
$\Pro_{A,\butheta}=  \Pro^{1}_{\theta^{1}} \otimes \ldots \otimes \Pro^{J}_{\theta^{J}}$.

Our presentation in the case of composite hypotheses will focus on the control of generalized \textit{familywise error rates}; the corresponding treatment of the generalized mis-classification rate will be  similar. Thus, given $k_1,k_2 \geq 1$ and $\alpha, \beta \in (0,1)$, the class of 
procedures of interest now is:
\begin{align*} 
\begin{split}
\Delta_{k_1, k_2}^{comp}(\alpha,\beta) :=\{(T,D) : & \; \max_{A,\; \butheta} \Pro_{A, \butheta}(|D\setminus A| \geq k_1) \leq \alpha \quad \text{and} \quad\\    & \; \max_{A,\; \butheta}\Pro_{A, \butheta} (|A \setminus D| \geq k_2) \leq \beta \},
\end{split}
\end{align*}
and  the goal is the same as the one in Problem~\ref{prob:GFAM} with $N^*_{A}(k_1,k_2,\alpha,\beta)$ being replaced by
\begin{align*} 
\begin{split}
N^*_{A,\butheta}(k_1,k_2,\alpha,\beta) &:=\inf_{(T,D) \in \Delta_{k_1,k_2}^{comp}(\alpha, \beta)} \Exp_{A,\butheta}[T],
\end{split}
\end{align*}
and the asymptotic optimality being achieved for every $A\subset [J]$ and  $\butheta \in \boldsymbol{\Theta}_{A}$.

\subsection{Leap rule with adaptive log-likelihood ratios}
The proposed procedure in this setup is  the Leap rule~\eqref{leap_def}, with the only difference that the local LLR statistics are  replaced by  statistics that account for the composite nature of the two hypotheses.  To be more specific, for every  $j \in [J]$ and $n \in \bN$ we denote  by  
$\ell^j(n, \theta^j)$  the  log-likelihood function  (with respect to some $\sigma$-finite measure $\nu^j_n$)  
in the $j$-th stream based on the first $n$ observations, i.e.,
\begin{align*} 
\ell^j(n, \theta^j) &:= \ell^j(n-1, \theta^j)+ 
\log\left( p^{j}_{\theta^j}(X^j(n) \, \vert \cF^j_{n-1}) \right); \quad  \ell^j(0, \theta^j):=0,
\end{align*}
where $p^{j}_{\theta^j}(X^j(n) \, \vert \cF^j_{n-1})$ is the conditional density of $X^j(n)$ given the previous $n-1$ observations in the j-th stream. 
Moreover, for  every  stream $j \in [J]$ and time $n \in \bN$ we denote  by $\ell^j_{i}(n)$ the corresponding \textit{generalized} log-likelihood  under $\Hyp_i^j$, i.e.,  
\begin{align*}
{\ell}^j_{i}(n) := \; \sup\left\{ \ell^j(n, \theta^j)\;:\;\theta^j \in \Theta^j_i \right\}, \quad i=0,1.
\end{align*}
Further, at each $n \in \bN$, we select  an  $\cF_n$-measurable estimator of $\butheta$, 
$\widehat{\butheta}_n=(\widehat{\theta}_n^1, \ldots, \widehat{\theta}_n^J) \in \boldsymbol{\Theta}$, 
and  define the \textit{adaptive  log-likelihood} statistic for the $j$-th stream as follows:
\begin{equation}
\label{adaptive_LL}
{\ell}^j_{*}(n) := {\ell}^j_{*}(n-1) + 
\log\left(
p^{j}_{\widehat{\theta}_{n-1}^j}(X^j(n) \, \vert  \cF^j_{n-1}) \right); \quad \ell^{j}_{*}(0) = 0,
\end{equation}
where  $\widehat{\butheta}_0:=(\widehat{\theta}_0^1, \ldots, \widehat{\theta}_0^J) \in \boldsymbol{\Theta}$ is some deterministic initialization. 
The proposed procedure in this context is the Leap rule~\eqref{leap_def}, where  each LLR statistic $\lambda^j(n)$ is replaced by  the following  \textit{adaptive} log-likelihood ratio:
\begin{align}\label{adaptive_stat}
\lambda_*^j(n) := 
\begin{cases}
\quad {\ell}_{*}^j(n)  - {\ell}_0^j(n),  \;
\;\;\; \text{ if } {\ell}_0^j(n) < {\ell}_1^j(n)  \;\; \text{and} \; \; {\ell}_0^j(n) < {\ell}_{*}^j(n) \\
-({\ell}_{*}^j(n)  - {\ell}_1^j(n)), \; 
\; \text{ if } {\ell}_1^j(n) < {\ell}_0^j(n) \;  \; \text{and} \;\;  {\ell}_1^j(n) < {\ell}_{*}^j(n) \\
\quad \text{undefined}, \;\;\quad\quad \text{ otherwise },
\end{cases}
\end{align}
with the understanding that there is no stopping at time $n$  if  $\lambda_*^j(n)$ is   undefined  for some $j$.  Clearly,  large positive values of $\lambda^j_{*}$ support $\Hyp_1^j$, whereas  large negative values of  $\lambda^j_{*}$ support $\Hyp_0^j$.   We denote this modified version of  the Leap rule by  $\delta_L^*(a,b) = (T_L^*, D_L^*)$. 

In the next  subsection we  establish the asymptotic optimality of $\delta_L^*$  under general conditions. 
In Appendix~\ref{comp_discussion} we discuss in more detail the above adaptive statistics, as well as other choices for the local statistics. In Appendix~\ref{app:simulation_composite} we demonstrate with a simulation  study that  if 
 we replace the LLR $\lambda^{j}$ by the  adaptive statistic $\lambda^{j}_{*}$~\eqref{adaptive_stat} in the 
 \textit{Intersection rule} \eqref{intersection_rule} and the \textit{asymmetric Sum-Intersection rule} \eqref{tau_0_def}, then these  procedures fail to be  asymptotically optimal \textit{even in the presence of special structures}.  
 Finally, we should point out that the gains over  fixed-sample size procedures are also   larger compared to the case of simple hypotheses,  as sequential methods  are more adaptive to the unknown parameter.



%

\subsection{Asymptotic optimality}
 First of all,  for each $j \in [J]$  we generalize  condition~\eqref{CLLN} and assume that for any  distinct $\theta^j,\tilde{\theta}^j \in \Theta^j$  there exists a positive number $I^j(\theta^j, \tilde{\theta}^j)$ such that
\begin{align}\label{composite_complete_convergence}
\frac{1}{n}\left(\ell^j(n,\theta^j) - \ell^j(n, \tilde{\theta}^j) \right) \;
\xrightarrow[n \to \infty]{\Pro^j_{\theta^j} \text{ completely } }
 \; I^j(\theta^j, \tilde{\theta}^j).
\end{align}

Second,  we  require that the null and alternative hypotheses in each stream are separated, in the sense that  if  for each $j \in [J]$  and $\theta^j \in \Theta^j$ we define
\begin{align}
\label{comp_info_discrete}
\cI_0^j(\theta^j) := \inf_{\tilde{\theta}^j \in \Theta_1^j} I^j(\theta^j, \tilde{\theta}^j) \qquad \text{and}  \qquad
\cI_1^j(\theta^j) := \inf_{\tilde{\theta}^j \in \Theta_0^j} I^j(\theta^j, \tilde{\theta}^j),
\end{align}
then 
\begin{align}
\label{comp_separable}
\cI_0^j(\theta^j) >0 \;\; \forall \; \theta^j \in \Theta_0^j \qquad \text{and}  \qquad
\cI_1^j(\theta^j) >0 \;\;\forall\; \theta^j \in \Theta_1^j.
\end{align}

Finally, we assume that for each $j \in [J]$ and $\epsilon > 0$,
\begin{align} \label{composite_uniform_convergence}
\begin{split}
\sum_{n=1}^{\infty}\Pro^j_{\theta^j} \left( \frac{{\ell}_{*}^j(n) - {\ell}_1^j(n)}{n}
- \cI_0^j(\theta^j) < -\epsilon\right) < \infty
 \text{ for every  } \theta^j \in \Theta_0^j, \\
\sum_{n=1}^{\infty}\Pro^j_{\theta^j} \left( \frac{{\ell}_{*}^j(n) - {\ell}_0^j(n)}{n}
- \cI_1^j(\theta^j) < -\epsilon\right) < \infty
 \text{ for every } \theta^j \in \Theta_1^j.
\end{split}
\end{align}

We  now state the  main result of this section,   the asymptotic optimality of $\delta_L^{*}$ under the above conditions. The proof is presented in   Appendix~\ref{sec:app_composte_case}.

\begin{theorem}\label{thm:comp_main}
Assume \eqref{composite_complete_convergence}, \eqref{comp_separable} and~\eqref{composite_uniform_convergence} hold. 
Further, assume the thresholds in the \procGF~are selected such that $\delta_L^{*}(a,b) \in \Delta^{comp}_{k_1,k_2}(\alpha,\beta)$ and $a \sim |\log(\beta)|, b \sim |\log(\alpha)|$, e.g. according to \eqref{leap_thresholds}.  Then, for any $A \subset [J]$ and $\butheta\in \boldsymbol{\Theta}_A$,  we have as  $\alpha, \beta \to 0$,
\begin{align*}
\Exp_{A,\butheta} \left[ T_L \right]
   \;\sim \; L_{A,\butheta}(k_1,k_2,\alpha, \beta) \;\sim\; N_{A,\butheta}^*(k_1,k_2,\alpha,\beta),
\end{align*}
where $L_{A,\butheta}(k_1,k_2,\alpha, \beta)$ is a quantity defined in Appendix~\ref{subsec:app_lower}
that characterizes the asymptotic optimal performance.
\end{theorem}

While  conditions~\eqref{composite_complete_convergence} and~\eqref{comp_separable} are easily satisfied and simple to check, the one-sided complete convergence condition~\eqref{composite_uniform_convergence} is not as apparent. It is known~\cite[p. 278-280]{tartakovsky2014sequential} that when  $\widehat{\theta}_n^j$ is selected to be the Maximum Likelihood estimator (MLE) of  $\theta^j$, condition~\eqref{composite_uniform_convergence} is  satisfied when  testing a normal mean with unknown variance, as well as when testing the coefficient of a first-order autoregressive model.  In Appendix~\ref{sub:exp_family} we further show that  condition~\eqref{composite_uniform_convergence} is satisfied when (i) the data in each stream are i.i.d.  with some \textit{multi-parameter exponential family} distribution, and  (ii) the null and the alternative parameter spaces are  compact.



%% file: a_Conclusion.tex
\section{Conclusion}\label{sec:conclusion}
In this paper we have considered the sequential multiple testing problem under two error metrics.
In the first one, the goal is to control the probability of  at least $k$ mistakes, of any kind. 
In the second one, the goal is to 
control simultaneously  the probabilities of  at least $k_1$ false positives and at least $k_2$ false negatives.
Assuming that the data for the various hypotheses are obtained sequentially in independent streams,    we  characterized the optimal performance to a first-order asymptotic approximation as the error probabilities vanish, and  proposed the first asymptotically optimal procedure for each of  the two problems.  Procedures that are asymptotically optimal under classical error control ($k=1$,  $k_1=k_2=1$) were found to be  suboptimal  under \textit{generalized} error metrics. Moreover, in the case of  i.i.d. data streams we quantified the asymptotic savings in the expected sample size relative to  fixed-sample size procedures.


There are certain questions that remain open.
First, we conducted a first-order asymptotic analysis, ignoring higher-order terms in the approximation to the optimal performance. The latter however appears to be non-negligible in practice(see Fig. \ref{fig:asym_b}). Thus, it is an open problem  to  obtain a more precise characterization of the optimal performance, 
as well as to examine whether the proposed rules enjoy a stronger optimality property. 
Second, the number of streams is treated as constant in our asymptotic analysis, but can be very large in practice. It is interesting to consider an enhanced asymptotic regime, where the number of streams also goes to infinity as the error probabilities vanish. Third, 
although simulation techniques can be used to determine  threshold values that guarantee the error control,  it is desirable to have closed-form expressions for less conservative threshold values. 
 

Finally, there are several interesting generalizations in various directions. One direction is to relax the assumption that the streams corresponding to the different testing problems are  independent.  Another direction is to allow for early stopping in some streams, in which case the goal may be to minimize the total number of observations in all streams. Finally, it is interesting to study the corresponding problems with FDR-type error control.



%% file: a_Supp_file.tex
\appendix

\section{Simulations for generalized mis-classification rate} \label{sec:simulation_KM}
In this section we present two simulation studies that complement our asymptotic optimality theory for procedures that control the generalized mis-classification rate (Section \ref{sec:KM}). Specifically, our goal is to compare the proposed Sum-Intersection rule (Subsection \ref{sum_int_def}) and the Intersection rule  \eqref{intersection_rule} in two setups. The first one is a \textit{symmetric and homogeneous} setup, in which  \eqref{sym_hom_KM} holds and both rules are asymptotically optimal. The second one is a non-homogeneous setup, where  condition \eqref{sym_hom_KM} is (slightly) violated and the \procInt~fails to be asymptotically optimal. 
In each setup, we also include the performance of the multiple Neyman-Pearson rule (MNP) \eqref{NP_fixed_sample}, which is a fixed-sample size procedure.

For these comparisons, we  consider the  testing of normal means, introduced in Example \ref{Gaussian_mean_test}. As discussed in Example \ref{Gaussian_mean_test}, this problem is \textit{symmetric}. As a result,
we set $h = 0$ in the MNP rule \eqref{NP_fixed_sample}, and further
the performance of each rule under consideration is the same for any subset of signals. 
Thus we do not need to specify the actual subset of signals.



\subsection{Homogeneous case} 
In the first study we set $\mu_j=0.25, \sigma_j = 1$ for every $j \in [J]$.  We consider $J = 100$ in Fig. \ref{fig:complete_sym_iid_KM} and $J =20$ in Fig. \ref{fig:complete_sym_iid_KM_J20}.

In Fig. \ref{fig:sym_a_KM}, we study the performance of the  Sum-Intersection rule for different values of $k$. We observe that there are significant savings in the ESS as $k$ increases and more mistakes are tolerated. In Fig. \ref{fig:sym_b_KM}, we compare the three rules  for $k =4$. Although both sequential  rules enjoy the asymptotic optimality property in this setup, we observe that the Sum-Intersection rule clearly outperforms the Intersection rule.

In Fig. \ref{fig:sym_b_KM} and \ref{fig:sym_a_KM_J20}, we also compare the Sum-Intersection rule with the MNP rule. Further, in Fig. \ref{fig:sym_c_KM}, \ref{fig:sym_b_KM_J20} and \ref{fig:sym_c_KM_J20},
we show the  histogram of the Sum-Intersection at particular error levels. 
From these figures we observe that the advantage of sequential procedures over the MNP rule increases as $\text{Err}$ decreases and decreases as $J$ increases. 


\begin{figure}[htbp]
\centering
    \makebox[\linewidth]{
\subfloat[Sum rule: vary $k$]{\label{fig:sym_a_KM}
\includegraphics[width=0.35\linewidth]{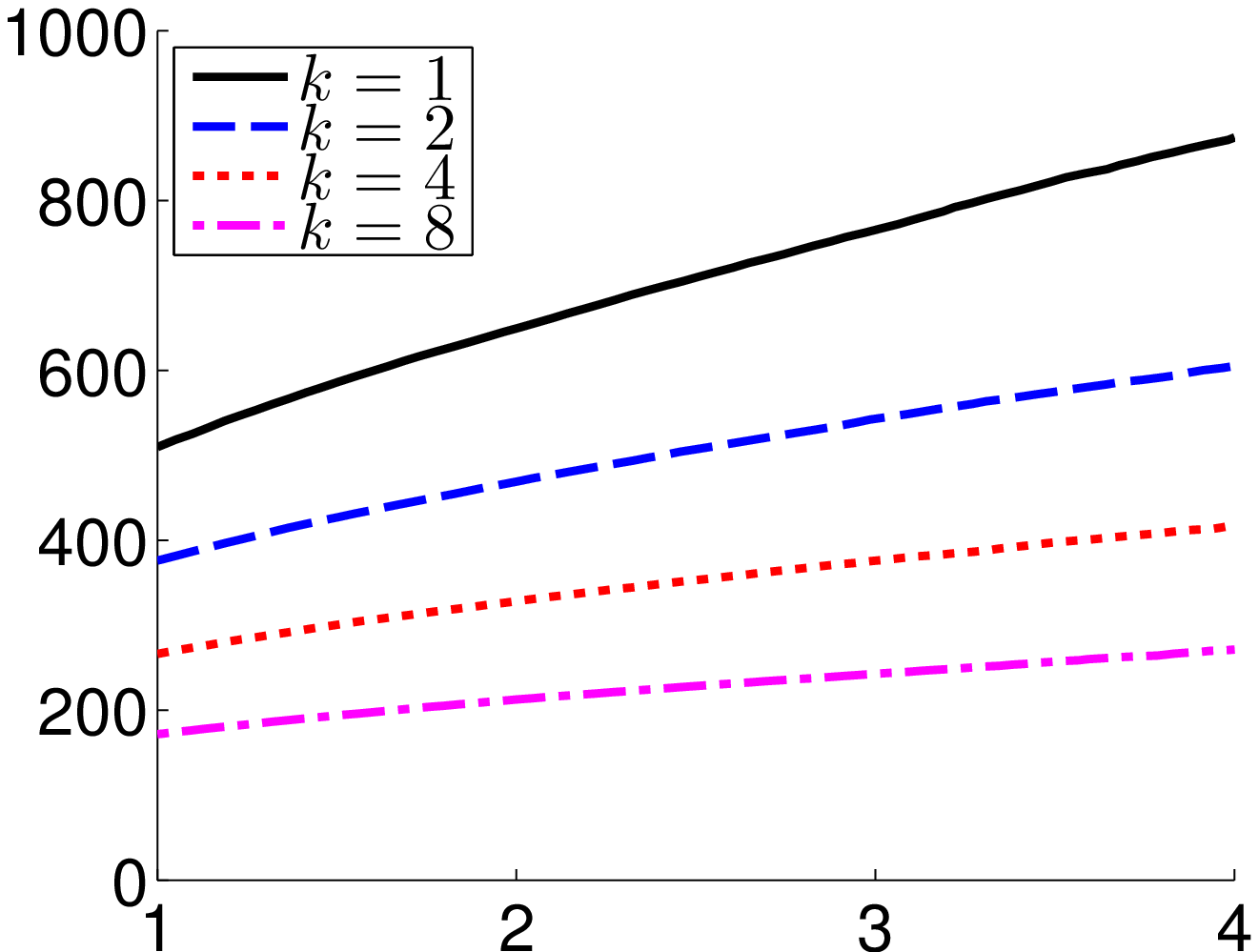}  
}
\subfloat[$k = 4$, vary rules]{\label{fig:sym_b_KM}
\includegraphics[width=0.35\linewidth]{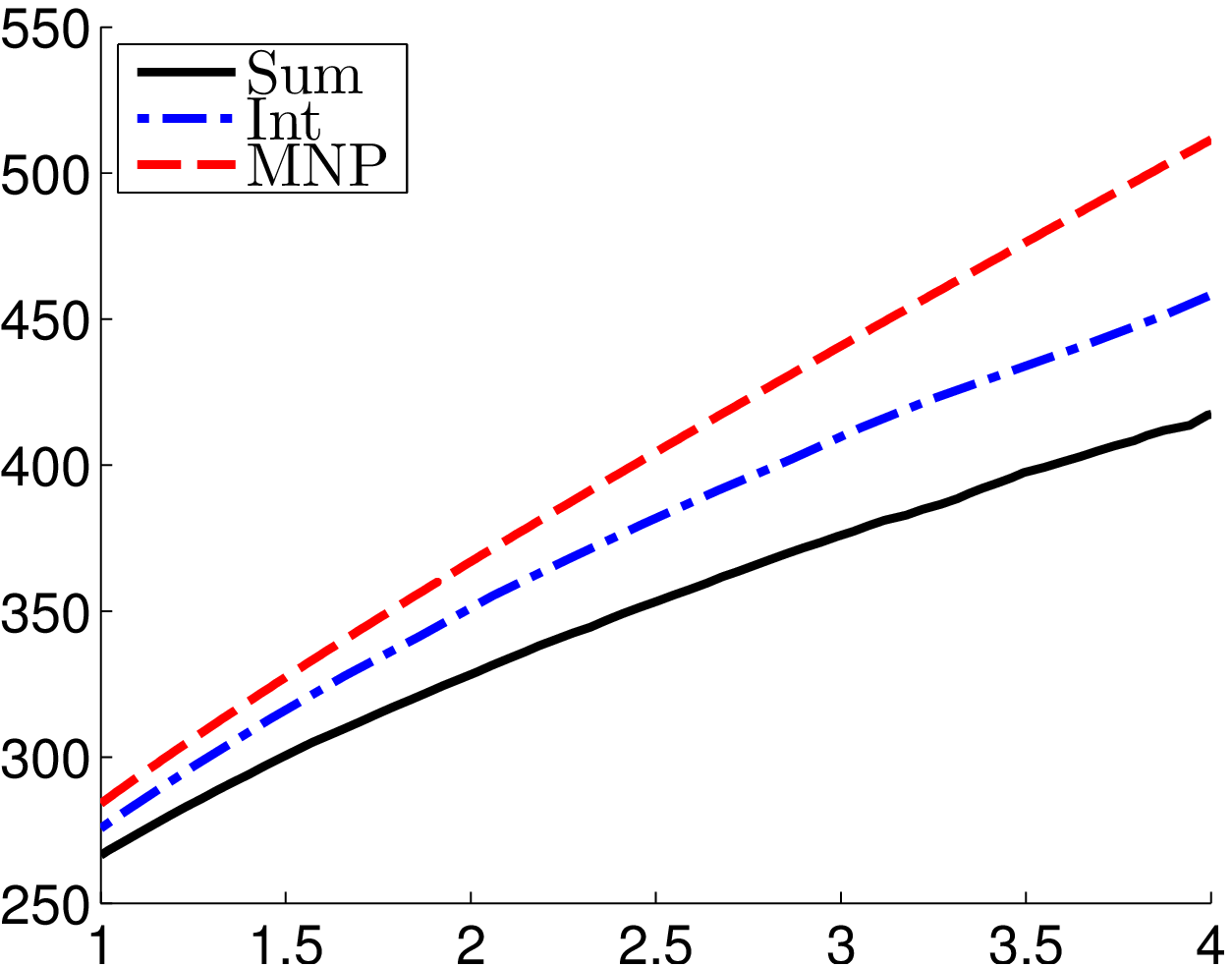}  
}

\subfloat[$k = 4, \text{Err} = 5\%$]{ \label{fig:sym_c_KM}
\includegraphics[width=0.35\linewidth]{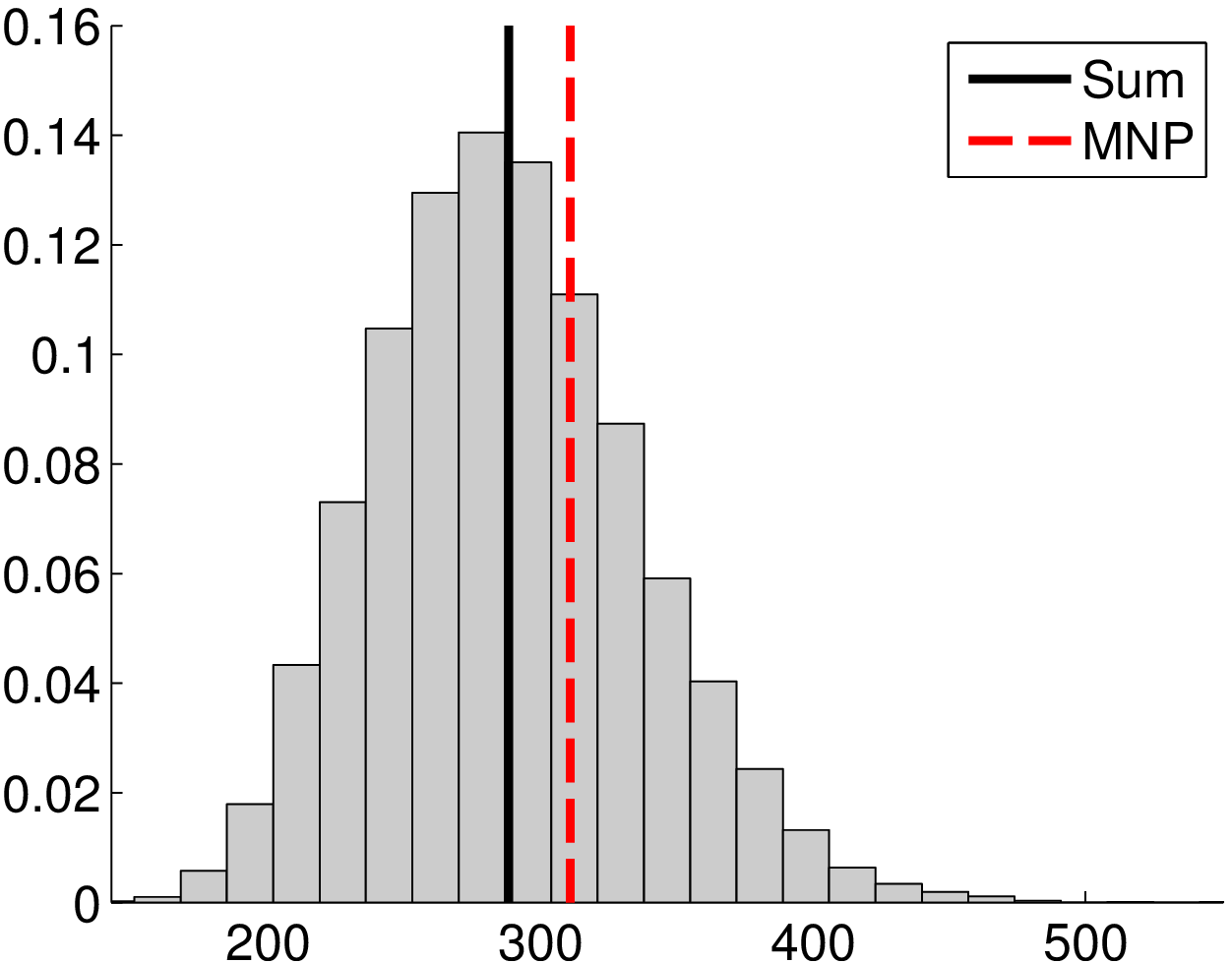}  
}
}
\caption{Homogeneous case: $J = 100$. 
In (a) and (b), the x-axis  is  $|\log_{10}(\text{Err})|$ and the y-axis represents the ESS.
In (c),  we plot the  histogram of the 
stopping time of the Sum-Intersection rule with $\text{Err} = 5\%$.
}
\label{fig:complete_sym_iid_KM}
\end{figure}

\begin{figure}[htbp]
\centering
    \makebox[\linewidth]{
\subfloat[$k = 2$]{\label{fig:sym_a_KM_J20}
\includegraphics[width=0.35\linewidth]{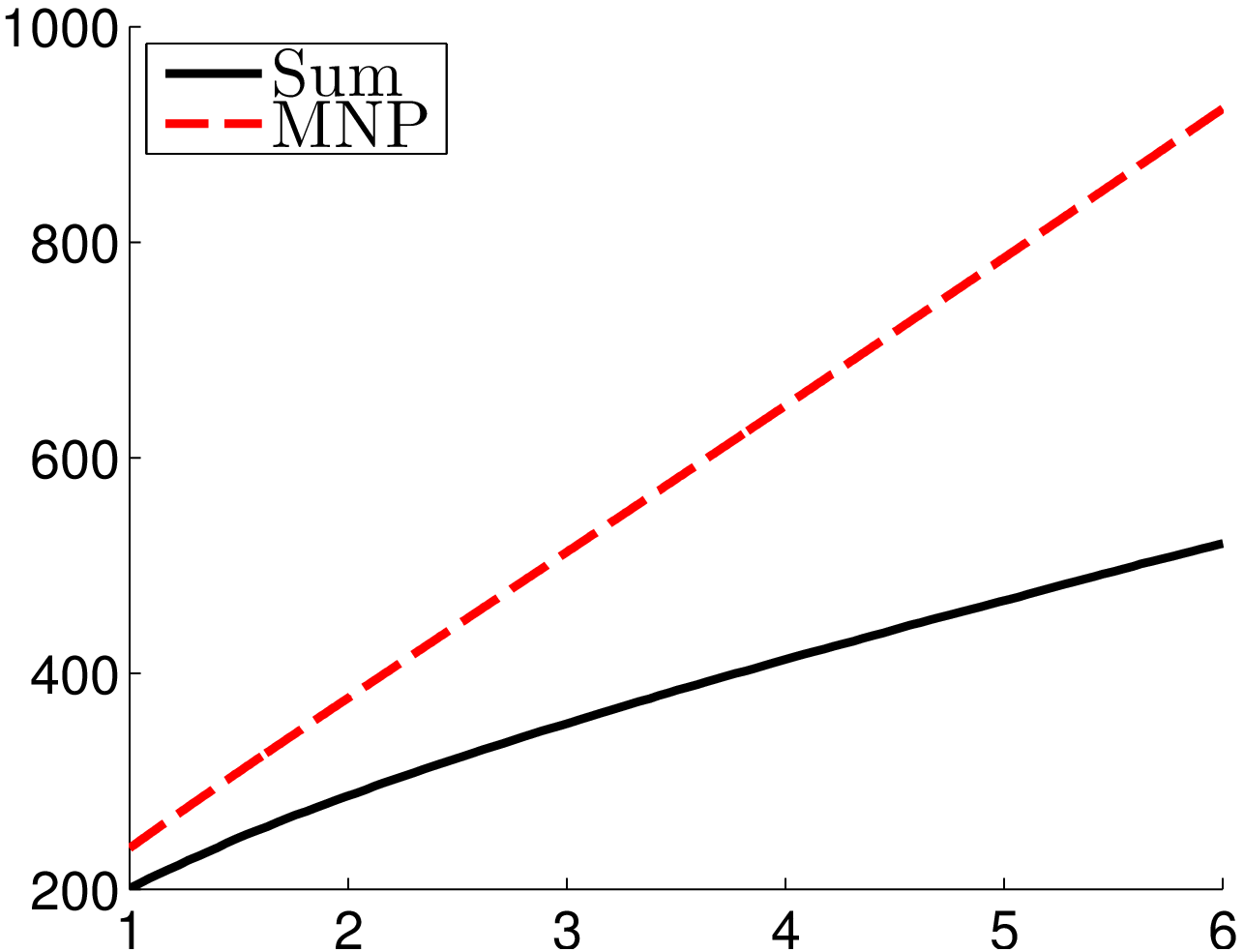}  
}
\subfloat[$k = 2, \text{Err} = 5\%$]{\label{fig:sym_b_KM_J20}
\includegraphics[width=0.35\linewidth]{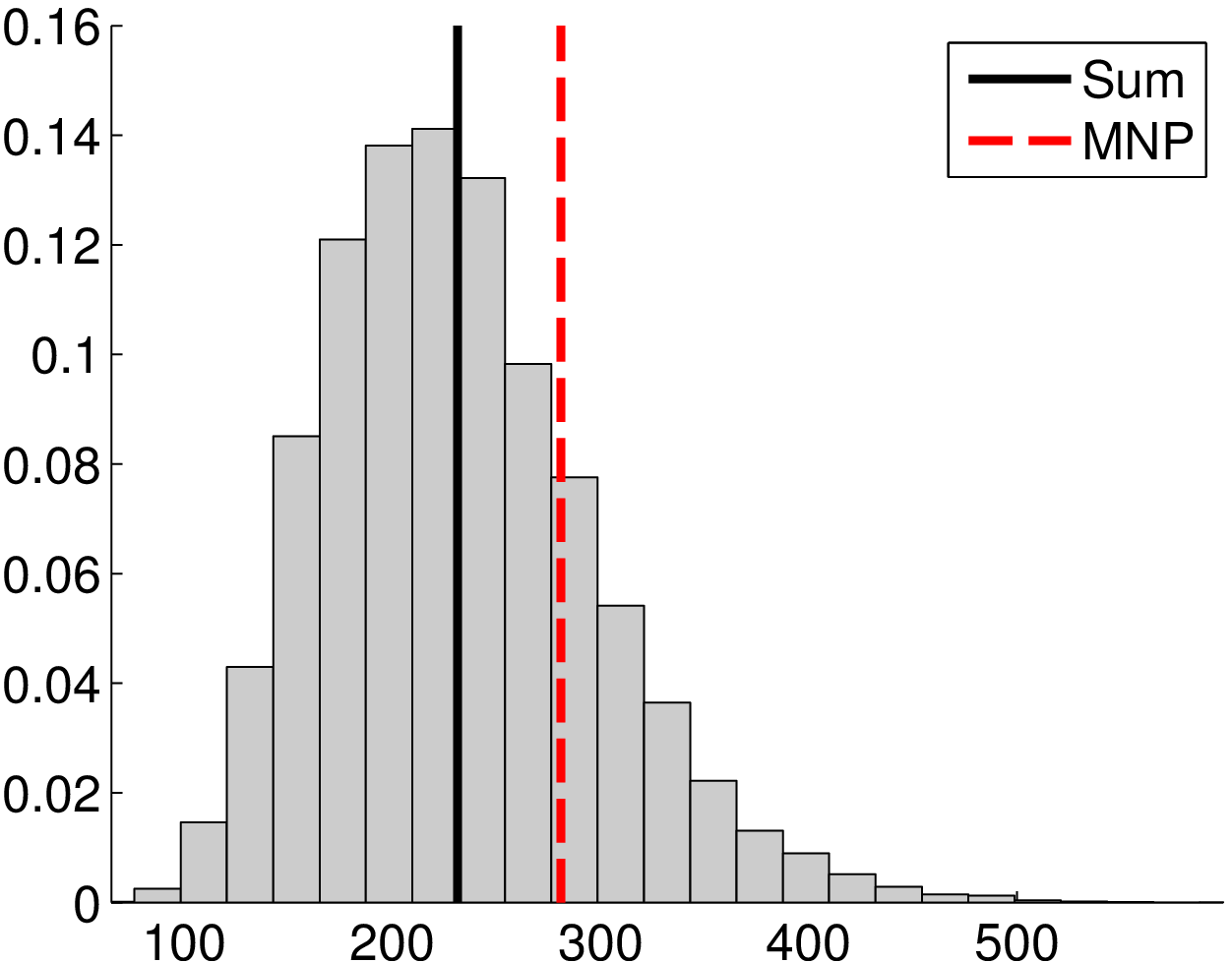}  
}

\subfloat[$k = 2, \text{Err} = 1\%$]{ \label{fig:sym_c_KM_J20}
\includegraphics[width=0.35\linewidth]{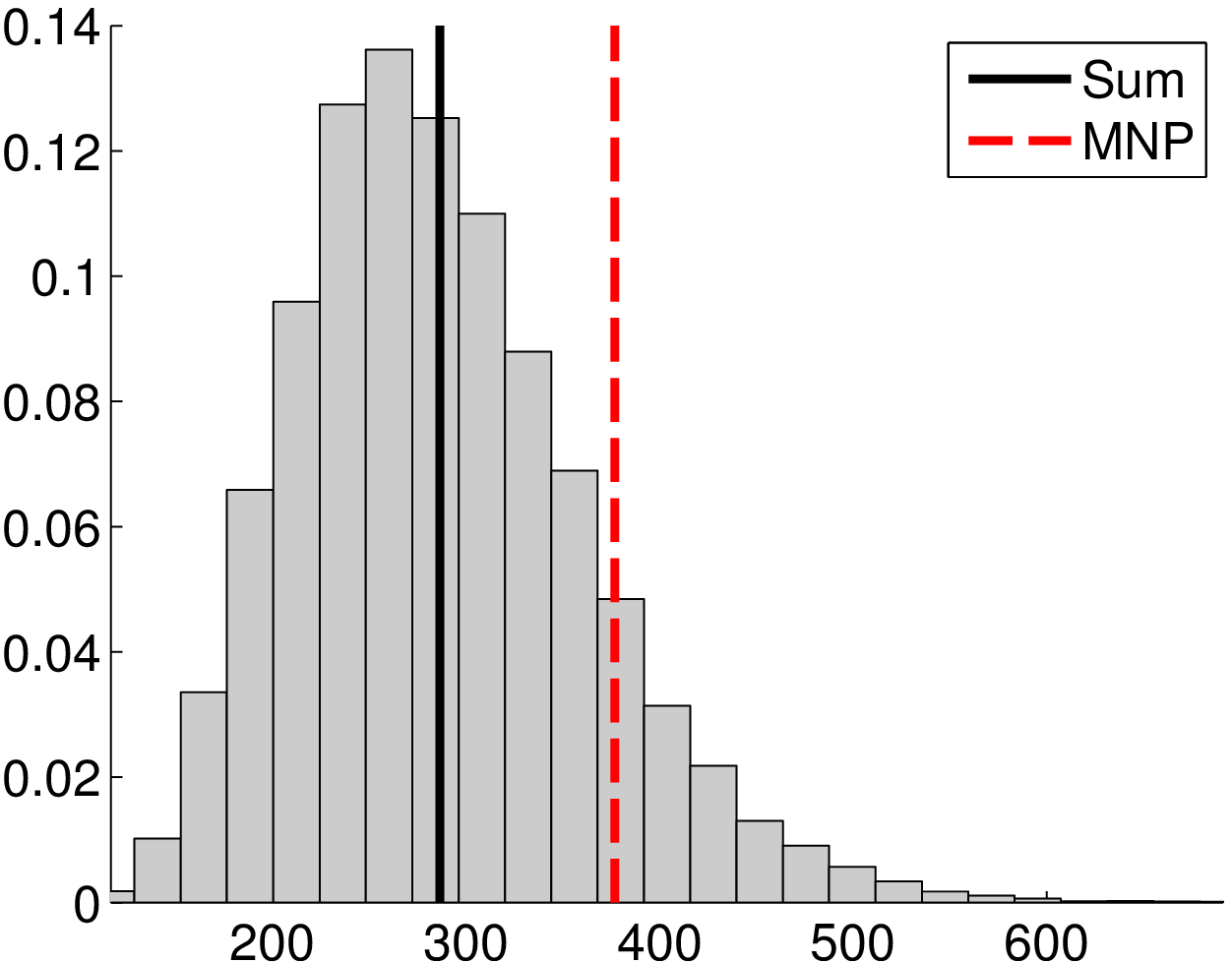}  
}
}
\caption{Homogeneous case: $J = 20$.
In (a), the x-axis  is  $|\log_{10}(\text{Err})|$ and the y-axis represents the ESS.
In (b) and (c),  we plot  the histogram of the 
stopping time of the Sum-Intersection rule with $\text{Err} = 5\%$ and $1\%$.
}
\label{fig:complete_sym_iid_KM_J20}
\end{figure}

\subsection{Non-homogeneous case}\label{sec:asym_case_KM}
In the second study  we have injected a slight violation of homogeneity. Specifically,  we set $J =10$, $k = 2$ and 
\begin{equation*}
f_0^j = \mathcal{N}(0,1) \;\;\forall\; j \in [J], \quad
f_1^j = \begin{cases}
\mathcal{N}(1/6,1) \;\text{ if } \; j = 1 \\
\mathcal{N}(1/2,1) \;\text{ if } \; j \geq 2 
\end{cases}.
\end{equation*}
Thus, all testing problems are identical apart from the first one, which is much harder.  Indeed,  $\MI_0^j = \MI_1^j = \MI^j$ and 
$\MI^j = 1/72$ for $j=1$ and $\MI^j = 1/8$ for $j \geq 2$. Since $k=2$, the optimal asymptotic performance in this problem is determined by the two most difficult hypotheses and is equal to 
$7.2|\log(\text{Err})|$.
In Fig. \ref{fig:asym_a_KM} we  plot the expected sample size (ESS)  against $|\log_{10}(\text{Err})|$ and in Fig. \ref{fig:asym_b_KM}  we plot  the ratio of ESS over $7.2|\log(\text{Err})|$.  We observe that this ratio tends to 1 for the asymptotically optimal  \procKM, whereas this is not the case for the  other two rules.
In particular, as predicted by Theorem \ref{KM_comp_fix_sample_rule},
the ratio for the MNP rule tends to $4$ as  $\text{Err} \to 0$.
\begin{figure}[htbp]
\subfloat[ESS vs $|\log_{10}(\text{Err})|$]{\label{fig:asym_a_KM}
\includegraphics[width=0.4\linewidth]{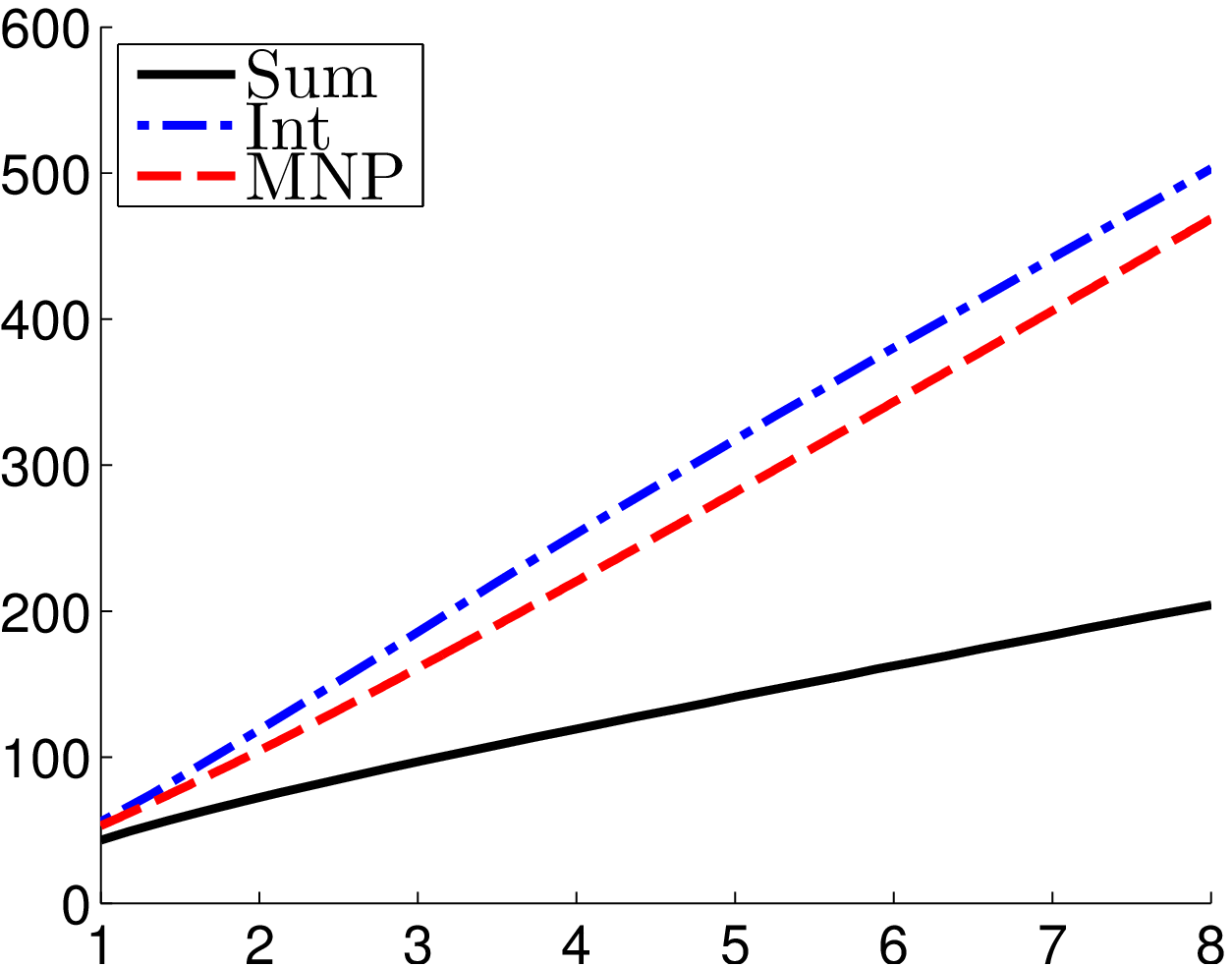} 
}\hspace{1cm}
\subfloat[Normalized version]{\label{fig:asym_b_KM}
\includegraphics[width=0.4\linewidth]{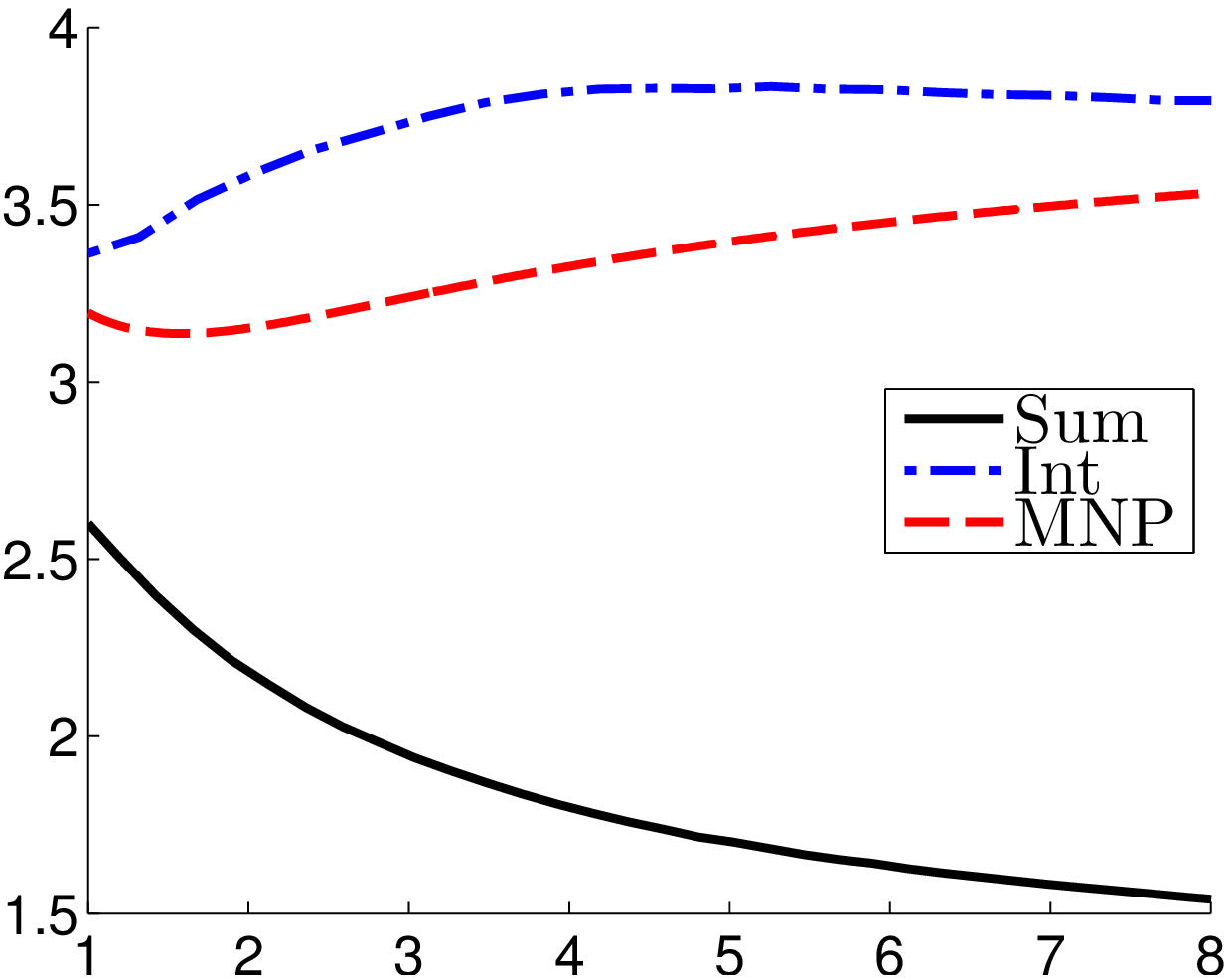} 
} 
\caption{Non-homogeneous case: $J = 10, k = 2$. 
The x-axis in both graphs  is  $|\log_{10}(\text{Err})|$. 
The y-axis  is the corresponding ESS in (a),  and is the ratio  of the ESS over $7.2|\log(\text{Err})|$ in (b).
}
\label{fig:asym_iid_KM}
\end{figure}

\section{Proofs regarding the generalized mis-classification rate}
\subsection{Proofs of Theorem \ref{KM_err_control}}\label{proof_of_KM_err_control}
\begin{proof} 
It suffices to show that  for any $b > 0$ and $A \subset [J]$ we have
\begin{align*}
\Pro_A(|A \;\triangle \; D_S(b)| \geq k) \leq C_{k}^{J} \, e^{-b}.
\end{align*}
Fix $A \subset [J]$ and   $b>0$. Observe that the event 
$\{|A \; \triangle \; D_S| \geq k\}$ occurs  if and only if there exist  $B_1 \subset A$ and $B_2 \subset A^c$ such that $|B_1| + |B_2| = k$ and the following event occurs:
\begin{equation*}
\Gamma(B_1,B_2) :=
\left\{D_S^{i} = 0, \, D_S^{j}= 1, \; \forall \; i \in B_1, j \in B_2 \right\}.
\end{equation*}
Since there are $C_{k}^{J}$ such pairs,  due to  Boole's inequality  it suffices to show that the probability of each of these events is bounded by $e^{-b}$. 
To this end,  fix  $B_1 \subset A, B_2 \subset A^c$ such that $|B_1| + |B_2| = k$ and consider the set   $C = (A \setminus B_1) \cup B_2$. Then, with the  change  of measure $\Pro_{A} \rightarrow \Pro_{C}$, we have
\begin{align}\label{eq:KM_change}
\Pro_{A} (\Gamma(B_1,B_2)) &= \Exp_{C}\left[ \exp\left\{ \lambda^{A,C}(T_S) \right\} ; \Gamma(B_1,B_2) \right].
\end{align}
For $i \in B_1$ we have $D_S^{i} = 0$, which implies $\lambda^{i}(T_S) \leq 0$, and  for $j \in B_2$ we have    $D_S^{j} = 1$, which implies $\lambda^{j}(T_S) >  0$. 
Thus, on the event $\Gamma(B_1,B_2)$, 
\begin{align}\label{KM_error_proof_eq}
\begin{split}
 \lambda^{A,C}(T_S) &= \sum_{i \in B_1} \lambda^{i}(T_S) - \sum_{j \in B_2} \lambda^j(T_S)  \\
 &=
- \sum_{i \in B_1} |\lambda^{i}(T_S)| - \sum_{j \in B_2} |\lambda^j(T_S)|  \leq   - \sum_{ i=1}^{k}  \widetilde{\lambda}^{i}(T_S) \leq  -b,
\end{split}
\end{align}
where the first equality is due to \eqref{ll_diff},  the first inequality follows from the definition of $\widetilde{\lambda}^i$'s, and the second from the definition of the stopping time $T_S$. Thus, the proof is complete in view of \eqref{eq:KM_change}.
\end{proof}

\subsection{An important Lemma}\label{app:exist_B_star}
The following lemma is crucial in establish Theorem~\ref{KM_lower_bound}.
\begin{lemma} \label{exist_B_star}
Let $A, B \subset [J]$. Then there exists $B^* \subset [J]$ such that
\begin{equation*}
(i)\quad  B \notin  \MU_{k}(B^*), \;\qquad (ii)\quad I^{A,B^*} \leq \mathcal{D}_A(k).
\end{equation*}
\end{lemma} 

To show Lemma~\ref{exist_B_star}, we start with a lemma about sets.
\begin{lemma}\label{set_lemma}
Let $A,B, \Gamma \subset [J]$. There exists $B^* \subset [J]$ such that
$$
A \;\triangle \; B^* \quad \subset\quad \Gamma  \quad\subset\; B \;\triangle \; B^*
$$
\end{lemma}
\begin{proof}
Define the following disjoint sets: 
\begin{align*}
B_1 = B \cap \Gamma, \quad B_2 = B^c \cap \Gamma, \quad
A_1 = A \cap \Gamma^c, \quad A_2 = A^c \cap \Gamma^c 
\end{align*}
Clearly, $\Gamma = B_1 \cup B_2$, and $\Gamma^c = A_1 \cup A_2$. Let $B^* = B_2 \cup A_1$.  

On one hand, if $j \in B_1$, then $j \in B$ and $j \not \in B^*$; if $j \in B_2$, then $j \not \in B$ and $j \in B^*$. It implies  
$\Gamma = B_1 \cup B_2 \subset B \; \triangle \; B^*$.

On the other, if $j \in A_1$, then $j \in A$ and $j \in B^*$; if $j \in A_2$, then $j \not \in A$ and $j \not \in B^*$. Thus $\Gamma^c = A_1 \cup A_2 \subset (A \; \triangle \; B^*)^c$, which implies
$A \; \triangle \; B^* \subset \Gamma$.
\end{proof}
Now we are ready to prove Lemma \ref{exist_B_star}.
\begin{proof}
Let $C^* \not \in \MU_k(A)$ such that $
\mathcal{D}_{A}( k) = \mathcal{I}^{A,C^*}$ and set  $\Gamma = A \;\triangle \; C^*$. Then, clearly 
$|\Gamma| \geq k$. By  Lemma \ref{set_lemma},  there exists a set $B^* \subset [J]$ such that
$$
A \;\triangle \; B^* \quad  \subset \quad  \Gamma = A \;\triangle \; C^* \quad \subset \quad \; B \;\triangle \; B^*.
$$
From the second inclusion it follows that $|B \;\triangle \; B^*| \geq |\Gamma| \geq k$, which proves (i).  From the first inclusion it follows that $A \setminus B^* \subset A\setminus C^*$ and $B^* \setminus A \subset C^* \setminus A$, therefore from \eqref{ll_diff} we conclude that
\begin{align*}
\MI^{A,B^*} = \sum_{i \in A\setminus B^*} \MI_1^{i} + \sum_{j \in B^* \setminus A} \MI_0^{j} \leq
\sum_{i \in A\setminus C^*} \MI_1^{i} + \sum_{j \in C^* \setminus A} \MI_0^{j} = \MI^{A,C^*}, 
\end{align*}
which proves (ii). 
\end{proof}

\subsection{Proof of Theorem \ref{KM_lower_bound}}\label{app:proof_KM_lower_bound}
\begin{proof}
Fix $A \subset [J]$, $k \in [J]$, and   set 
$$
\ell_{\alpha} := |\log(\alpha)|/\mathcal{D}_{A}(k) , \quad  \alpha \in (0,1).
$$
By Markov's inequality, for any stopping time $T$, $\alpha \in (0,1)$ and $q >0$,
$$
 \Exp_A[T] \geq q \ell_{\alpha} \, \Pro_A(T \geq q \ell_{\alpha}).
$$
Thus,  it suffices to show for every $q \in (0,1)$ 
we have
\begin{equation}
\label{KM_sufficient}
\liminf_{\alpha \to 0}  \inf_{(T,D) \in \Delta_{k}(\alpha)} \Pro_A(T \geq q \ell_{\alpha}) \geq 1,
\end{equation}
as this will imply
$\liminf_{\alpha \to 0} N^*_A(k,\alpha)/ \ell_{\alpha} \geq q,
$ 
and the desired result will follow  by letting $q \to 1$. 

In order to prove \eqref{KM_sufficient}, let us start by fixing arbitrary $\alpha, q \in (0,1)$ and  $(T,D) \in \Delta_{k}(\alpha)$.  Then,
\begin{align}  \label{exceed_prob_in_nb}
1 - \alpha &\leq  \Pro_A(D \in \MU_k(A)) = \sum_{B \in \MU_k(A)} \Pro_A(D = B). 
\end{align}
Now, consider an arbitrary  $B \in \MU_k(A)$, and  let $B^* \subset [J]$ be  a set that satisfies the two conditions in Lemma \ref{exist_B_star}. Then, $|B^* \; \triangle \; B| \geq k$, and  consequently 
\begin{equation} \label{bound}
\Pro_{B^*}(D = B) \leq \alpha.
\end{equation}
We can now decompose  the  probability  $\Pro_A(D = B)$ as follows:
\begin{align*}
 \Pro_A\left(\lambda^{A,B^*}(T) < \log\left(\frac{\eta}{\alpha}\right); D = B \right)  + \Pro_A\left(\lambda^{A,B^*}(T)  \geq \log\left(\frac{\eta}{\alpha}\right) ; D = B\right),
\end{align*}
where  $\eta$ is  an arbitrary constant in $(0,1)$.
We denote the first term by $\text{I}$ and second by $\text{II}$. For the first term, by a  change of measure $\Pro_A \rightarrow \Pro_{B^{*}}$ we have 
\begin{align*}
\text{I} &=\Exp_{B^*} \left[ \exp \{ \lambda^{A,B^*}(T)  \}  \;  ;  \; \lambda^{A,B^*}(T) <  \log\left(\frac{\eta}{\alpha}\right), D = B \right]   \\
 &\leq \frac{\eta}{\alpha} \Pro_{B^*}(D = B) \leq \eta,
\end{align*}  
where the second inequality follows from \eqref{bound}.
For the second term, we have
\begin{align*}
\text{II} &\leq \Pro_A \left( T \leq q \frac{|\log \alpha |}{\mathcal{D}_{A}(k)}, \,  \lambda^{A,B^*}(T)  \geq  \log\left(\frac{\eta}{\alpha}\right) \right) + \Pro_A(T \geq q \ell_\alpha, D = B).
\end{align*}  
By construction, $B^{*}$ satisfies $\MI^{A,B^*} \leq \mathcal{D}_{A}(k)$; thus the first term in the right-hand side is  bounded above by 
$$\epsilon_{\alpha, B^*}(T) :=  \Pro_A \left( T \leq q\frac{|\log \alpha |}{I^{A,B^*}}, \quad  \lambda^{A,B^*}(T)  \geq |\log \alpha | + \log(\eta) \right).$$
Due to the  SLLN \eqref{LLN},  we have 
\begin{align*} 
\Pro_A \left( \lim_{n \rightarrow \infty} \frac{\lambda^{A,B^*}(n)}{n} =\mathcal{I}^{A,B^*}\right)=1.
\end{align*} 
Therefore, by Lemma \ref{T_small_ST_larger_lemma}, it follows that  $\epsilon_{\alpha, B^*}(T) \to 0$ as $\alpha \to 0$  uniformly in $T$.


 Putting everything together we have 
\begin{align*}
\Pro_A(D = B)  \leq \eta +  \epsilon_{\alpha, B^*}(T)+ \Pro_A(T \geq q \ell_\alpha, D = B),
\end{align*}
and summing over $B \in \MU_k(A)$ we obtain
\begin{align*}
\Pro_A(D \in \MU_k(A))  &\leq |\MU_k(A)| \eta + \epsilon_{\alpha}(T) + \Pro_A(T \geq q \ell_\alpha, D \in \MU_k(A)) \\
&\leq |\MU_k(A)| \eta + \epsilon_{\alpha}(T) + \Pro_A(T \geq q \ell_\alpha),
\end{align*}
where $\epsilon_{\alpha}(T) := \sum_{B \in \MU_k(A)}  \epsilon_{\alpha, B^*}(T) \to 0$  as $\alpha \to 0$  uniformly in $T$.  
Due to \eqref{exceed_prob_in_nb}, we have
\begin{align*}
\Pro_A(T \geq q \ell_\alpha) \geq\; 1 - \alpha - \epsilon_{\alpha}(T) - |\MU_k(A)| \eta.
\end{align*}
Since  $(T,D) \in \Delta_{k}(\alpha)$  is  arbitrary  and $\alpha \in (0,1)$ also arbitrary,  taking the infimum over $(T,D)$ and letting $\alpha \to 0$  we obtain
\begin{equation*}
\liminf_{\alpha \to 0}\inf_{(T,D) \in \Delta_{k}(\alpha)}\Pro_A (T \geq q \ell_\alpha) \geq 1  -  |\MU_k(A)| \eta.
\end{equation*}
Finally, letting $\eta \to 0$ we obtain \eqref{KM_sufficient}, which completes the proof.
\end{proof}

\subsection{Proof of Theorem \ref{order_sum_first_opt}} \label{proof_of_KM_upper_bound}
The following fact about set operations will be needed:
\begin{equation}\label{set_symm}
\text{Let }\; A, B \subset [J] \;\text{ and }\; C =A\; \triangle\; B\;. \text{ Then }    A \; \triangle \; C = B.
\end{equation}

\begin{proof} 
Fix $A \subset [J]$ and consider the  stopping time
\begin{align*}
T^A(b) := \inf\left\{n \geq 1:  \lambda^{A,C}(n) \geq b  \quad  \forall \,  C \notin \MU_k(A) \right\} .
\end{align*}
Under  the conditions of the lemma, from Lemma \ref{multiple_cross} in the Appendix it follows that  as  $b \rightarrow \infty$ we have
$$
\Exp_A[T^A(b)] \leq \frac{b\, ( 1+ o(1)) }{\mathcal{D}_{A}(k)}.
$$
Thus,  it suffices to show that $T_S(b) \leq T^A(b)$ for any given $b>0$. In what follows, we fix $b>0$ and suppress the dependence on $b$. By the definition of the \procKM,  it suffices to show that
\begin{equation}\label{TP_dom_TS_suff}
 \sum_{i \in B} |\lambda^{i}(T^A)| \geq b, \quad \forall \;B \subset [J]: \;  |B|=k. 
\end{equation}
To this end,  fix  $B \subset [J]$ with $|B| = k$ and  set $C = A\; \triangle \; B$.  Then, from \eqref{set_symm} we have that $B= A \; \triangle \; C $.  Since $|B| \geq k$,  it follows that  $C \not\in \MU_k(A)$,   and  by the definition of $T^A$ we have $\lambda^{A,C}(T^A) \geq b$. As a result, 
\begin{align*}
b \leq \lambda^{A,C}(T^A) &=  \sum_{i \in A \setminus C} \lambda^{i}(T^A) - \sum_{j \in C \setminus A} \lambda^{j}(T^A) \\
&\leq \sum_{i \in A\; \triangle \; C} |\lambda^{i}(T^A)| = 
\sum_{i \in B} |\lambda^{i}(T^A)|.
\end{align*}
The proof is complete in view of  \eqref{TP_dom_TS_suff}. 
\end{proof}

\subsection{Proof of Corollary \ref{inter_sym}}\label{proof_of_inter_sym}
\begin{proof} 
 Fix $A \subset [J]$. For (i) it suffices to show that 
for any $b > 0$ 
\begin{align*}
\Pro_A(|A \;\triangle \; D_I(b,b)|) \leq C_{k}^{J} \; e^{-kb}.
\end{align*} 
The proof is  identical to that of Theorem \ref{KM_err_control} as long as we replace the inequalities in \eqref{KM_error_proof_eq} by
\begin{equation*}
- \sum_{i \in B_1} |\lambda^{i}(T_I)| - \sum_{j \in B_2} |\lambda^j(T_I)|  \leq - kb.
\end{equation*}

In order to prove (ii), setting  $k = 1$ in  Theorem \ref{order_sum_first_opt} we have  as $b \to \infty$ 
\begin{equation} \label{eq:int_KM_ESS}
 \Exp_A[T_I(b,b)] \leq \frac{b \, (1 + o(1))}{ \min_{C \neq A} \cI^{A,C}}.
\end{equation}
If condition \eqref{sym_hom_KM} is satisfied, then 
 $\min_{C \neq A} \cI^{A,C} = \MI$. Therefore, if 
$b \sim |\log \alpha|/k$, from
 \eqref{eq:int_KM_ESS} we have that   as $\alpha \to 0$
$$ \Exp_A \left[ T_I \right] \leq  \frac{|\log \alpha|}{k \MI} (1+o(1)).
$$
Further, this asymptotic upper bound agrees with the asymptotic lower bound in \eqref{ALB}, since    $\mathcal{D}_{A}(k)=k \MI$ when condition \eqref{sym_hom_KM} holds.
Thus, the proof is complete.
\end{proof}

\subsection{Proof of Theorem \ref{KM_comp_fix_sample_rule}}\label{proof_of_KM_comp_fix_sample_rule}
\begin{proof}
Since $k \leq (J+1)/2$ is fixed, we write $n^*(\alpha)$ (resp. $n_{NP}(\alpha)$) for $n^*(k,\alpha)$ (resp. $n_{NP}(k,\alpha)$) for simplicity. By Theorem \ref{order_sum_first_opt}, for any $A \subset [J]$ we have
\begin{align*}
N_A^*(k,\alpha)\, \sim\,  \frac{|\log \alpha|}{\mathcal{D}_A(k)} \;  \text{ as } \;  \alpha \to 0.
\end{align*}
Thus, it suffices to show that 
\begin{align} \label{aux_si_opt}
\liminf_{\alpha \to 0} \frac{n^*(\alpha)}{\vert \log(\alpha)\vert} \geq 
\frac{1}{\sum_{j=1}^{2k-1} \cC^{(j)}} \qquad \text{and} \qquad 
n_{NP}(\alpha) \sim \frac{|\log(\alpha)|}{\sum_{j=1}^{k} \cC^{(j)}}.
\end{align}

\noindent (i) Let us first focus on $n^*(\alpha)$. By its definition \eqref{def:fix_ns}, there exists 
$$D^*(\alpha) \in \Delta_{fix}(n^*(\alpha)) \cap \Delta_{k}(\alpha).$$

Denote $\Pro$ the probability measure for data in all streams.
For any $A \subset [J]$ with $|A| = 2k-1$, we consider the following simple versus simple problem:
\begin{align}\label{aux_simple_vs_simple}
{\Hyp_0'}: \Pro = \Pro_{\emptyset} \;\; \text{ vs. } \;\;
{\Hyp_1'}: \Pro = \Pro_{A},
\end{align}
where $\Pro_A$ is defined in \eqref{product}. Consider the following procedure for \eqref{aux_simple_vs_simple}:
\begin{align*}
\bar{D}^*(\alpha) = 
\begin{cases}
0 \;\; \text{ if } |D^*(\alpha)| < k \\ 
1 \;\; \text{ if } |D^*(\alpha)| \geq k 
\end{cases}.
\end{align*}
Then by definition of $D^*(\alpha)$, we have
\begin{align*}
&\Pro_{\emptyset}(\bar{D}^*(\alpha) = 1) = 
\Pro_{\emptyset}(|D^*(\alpha)| \geq k) \leq \alpha, \\
&\Pro_{A}(\bar{D}^*(\alpha) = 0) = 
\Pro_{A}(|D^*(\alpha)| < k) \leq \alpha,
\end{align*}
where the second inequality uses the fact that $|A| = 2k-1$. Thus 
\begin{align*}
\frac{1}{n^*(\alpha)} \log(\alpha) \geq
\frac{1}{n^*(\alpha)} \log\left(\frac{1}{2}\Pro_{\emptyset}(\bar{D}^*(\alpha) = 1)  + \frac{1}{2} \Pro_{A}(\bar{D}^*(\alpha) = 0) \right). 
\end{align*}
By Chernoff's lemma \ref{generalized_chernoff_lemma}, 
$$
\liminf_{\alpha \to 0}\frac{1}{n^*(\alpha)} \log\left(\frac{1}{2}\Pro_{\emptyset}(\bar{D}^*(\alpha) = 1)  + \frac{1}{2} \Pro_{A}(\bar{D}^*(\alpha) = 0) \right) 
 \geq - \Phi^{A}(0)
$$
where $\Phi^{A}(0) := \sup_{\theta \in \bR}\left\{-\log \left(\Exp_{\emptyset}\left[
e^{\theta \lambda^{A,\emptyset}(1)}
\right]
\right)\right\}$. Due to independence,
\begin{align*}
\Phi^{A}(0) =  \sup_{\theta \in \bR} \left\{\sum_{j \in A} -\log \left(\Exp_{0}^{j}\left[
e^{\theta \lambda^j(1)} 
\right]
\right)\right\}
\leq \sum_{j \in A} \Phi^{j}(0) \equiv \sum_{j \in A} \cC_{j},
\end{align*}
As a result, 
$$
\liminf_{\alpha \to 0}\frac{1}{n^*(\alpha)} \log(\alpha) \geq  -\sum_{j \in A} \cC_{j}.
$$
Maximizing the lower bound over $A \subset [J]$ with $|A| = 2k-1$, we obtain the inequality in  \eqref{aux_si_opt}.

\noindent (ii) We now focus on $n_{NP}(\alpha)$. By definition, there exists some $\tilde{h} \in \bR^{J}$ such that
$$
(n_{NP}(\alpha), \tilde{D}(\alpha)) \in \Delta_{k}(\alpha),
\text{ where }
\tilde{D}(\alpha) := D_{NP}(n_{NP}(\alpha), \tilde{h}).
$$
Denote
\begin{align*}
p_{j} := \Pro_0^{j}(\tilde{D}^{j}(\alpha) = 1) = 
\Pro_0^{j}\left(\lambda^j(n_{NP}(\alpha)) > \tilde{h}_{j}  \; n_{NP}(\alpha)\right) \\
q_{j} := \Pro_1^{j}(\tilde{D}^{j}(\alpha) = 0) = 
\Pro_1^{j}\left(\lambda^j(n_{NP}(\alpha)) \leq \tilde{h}_{j}  \; n_{NP}(\alpha)\right)
\end{align*}

For any $A_1, A_2 \subset [J]$ such that $A_1 \cap A_2 = \emptyset$ and $|A_1 \cup A_2| = k$,
\begin{align*}
\alpha \geq \Pro_{A_1}\left( \cap_{j \in A_1}\{\tilde{D}^{j}(\alpha) = 0\}
\bigcap \cap_{i \in A_2}\{\tilde{D}^{i}(\alpha) = 1\}
\right) = \prod_{j \in A_1} q_j \prod_{i \in A_2} p_i,\\
\alpha \geq \Pro_{A_2}\left( \cap_{j \in A_1}\{\tilde{D}^{j}(\alpha) = 1\}
\bigcap \cap_{i \in A_2}\{\tilde{D}^{i}(\alpha)= 0\}
\right) = \prod_{j \in A_1} p_j \prod_{i \in A_2} q_i.
\end{align*}
Since $A_1, A_2$ are arbitrary, we have for any $A \subset [J]$ with $|A| = k$
\begin{align*}
\alpha \geq \prod_{j \in A} \max\{p_j, q_j\},
\end{align*}
which implies that
\begin{align*}
\log(\alpha) \geq \sum_{j \in A} \max\{\log(p_j), \log(q_j)\} \geq 
\sum_{j \in A} \log( p_j/2 +  q_j/2).
\end{align*}
Thus, again by Chernoff's Lemma \ref{generalized_chernoff_lemma},
\begin{align*}
\liminf_{\alpha \to 0} \frac{1}{n_{NP}(\alpha)} \log(\alpha) \geq -\sum_{j \in A} \Phi^j(0).
\end{align*}
Maximizing the lower bound over $A \subset [J]$ with $|A| = k$, we have
\begin{align*}
\liminf_{\alpha \to 0} \frac{n_{NP}(\alpha)}{|\log(\alpha)|} \geq \frac{1}{\sum_{j=1}^{k} \cC^{(j)}}.
\end{align*}
The lower bound is achieved when $\tilde{h} = 0$, and this proves the equivalence in  \eqref{aux_si_opt}.
%
%
\end{proof}

\subsection{Bernoulli example under the generalized mis-classification rate}
\label{app:binom_example}
Suppose that for each $j \in [J]$, $\{X^{j}(n): n \in \bN\}$ are i.i.d. Bernoulli random variables and that there is a  constant  $p \in (0,1/2)$ such that 
$$
{\Hyp}_0^j\;: \Pro_0^{j}(X^j(1) = 1) = p \; \text{ versus }\; 
{\Hyp}_1^j\;: \Pro_1^{j}(X^j(1) = 1) = 1 - p := q.
$$ In this case, $\cI_0^{j} = \cI_1^{j} = H(p)$, where 
$$H(x) := x\log\left(\frac{x}{1-x} \right) + (1-x) \log\left(\frac{1-x}{x} \right).
$$
 Further, 
$$
\Phi(0) = \sup_{\theta \in \bR}
\left\{
-\log(p^{\theta}q^{1-\theta} + p^{1-\theta} q^{\theta})
\right\} = \log\frac{1}{2\sqrt{p(1-p)}}.
$$
By Theorem \ref{KM_comp_fix_sample_rule}, for any $A \subset [J]$ we have 
\begin{align*}
&\liminf_{\alpha \to 0} \;
\frac{n^{*}(k,\alpha)}{ N_A^*(k,\alpha)} \geq 
\frac{k H(p)}{(2k-1)\Phi(0)}, \qquad 
\lim_{\alpha \to 0}\;
\frac{n_{NP}(k,\alpha)}{N_A^*(k,\alpha)} = 
 \frac{H(p)}{\Phi(0)}.
\end{align*}
In Figure \ref{fig:binomial_example}, we plot ${H(p)}/{\Phi(0)}$ as a function of $p$.

\begin{figure}[htbp!]
\includegraphics[width = 0.5\textwidth]{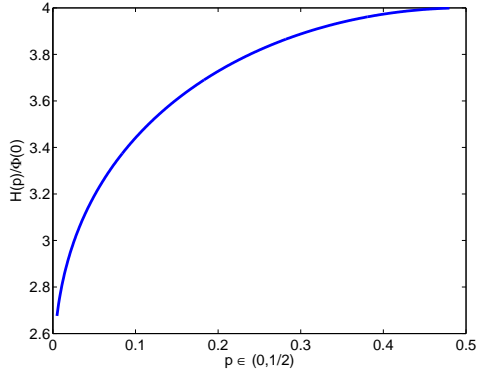}
\caption{The plot for ${H(p)}/{\Phi(0)}$ as a function of $p$ }
\label{fig:binomial_example}
\end{figure}

\section{Proofs regarding the generalized familywise error rates}

%

\subsection{Proof of Theorem \ref{leap_error_control}} \label{proof:leap_error_control}

The goal in this subsection is to show that 
for any $a,b > 0$ and $A \subset [J]$ we have  
\begin{equation*} 
\Pro_{A}(|D_L \setminus A| \geq k_1)  \leq Q(k_1) \,  e^{-b}, \quad \Pro_{A}(|A \setminus D_L| \geq k_2)  \leq Q(k_2)\,  e^{-a},
\end{equation*}
where $Q(k) := 2^{k} C^{J}_{k}$.
We start with a lemma that shows how to select the thresholds  for  procedures $\widehat{\delta}_\ell$, $0\leq  \ell <k_1$ and $\widecheck{\delta}_\ell$, $ 0 \leq  \ell <k_2$.

\begin{lemma} \label{taus_error_control}
Assume that \eqref{sep_lr} holds. Fix  $A \subset [J]$. Let $B_1 \subset A^c$ with $|B_1| = k_1$, and $B_2 \subset A$ with $|B_2| = k_2$.
\begin{enumerate}[label={(\roman*)}]
\item Fix any $0 \leq \ell < k_1$. For any event $\Gamma \in \MF_{\widehat{\tau}_\ell}$,  we have
\begin{equation*}
\Pro_A(B_1 \subset \widehat{D}_{\ell}) \leq  C_{\ell}^{k_1} e^{-b},\quad
\Pro_A(B_2 \subset \widehat{D}_{\ell}^c,\, \Gamma) \leq  e^{-a} \Pro_{A\setminus B_2} (\Gamma).
\end{equation*}
\item Fix any $0 \leq \ell < k_2$. For any event $\Gamma \in \MF_{\widecheck{\tau}_\ell}$,   we have
\begin{equation*}
\Pro_A(B_1 \subset \widecheck{D}_\ell,\, \Gamma) \leq  e^{-b}\Pro_{A\cup B_1} (\Gamma) ,\quad
\Pro_A(B_2 \subset \widecheck{D}_\ell^c) \leq  C_{\ell}^{k_2} e^{-a}.
\end{equation*}
\end{enumerate}
\end{lemma}

\begin{proof}
We will only prove (i), since (ii) can be shown in a similar way. Fix $0 \leq \ell < k_1$. 
By definition, $\widehat{D}_{\ell}$ rejects the nulls in the $\ell$ streams with the least significant non-positive LLR, in addition to the nulls in the streams with positive LLR. Thus, 
$$
\{B_1 \subset \widehat{D}_{\ell}\} \subset \bigcup_{M \subset B_1, |M| = k_1 - \ell} \Pi_{M} \;, \quad \text{ where }\;
\Pi_{M} := \{\lambda^j(\widehat{\tau}_\ell) > 0 \;\; \forall j \; \in M\}.
$$
With a change of measure from $\Pro_A \rightarrow \Pro_C$, where $C = A \cup M$, we have
$$
\Pro_A(\Pi_{M}) = \Exp_C\left[\exp\{\lambda^{A,C}(\widehat{\tau}_\ell)\}; \Pi_{M} \right] 
= \Exp_C\left[\exp\left\{-\sum_{j \in M}\lambda^{j}(\widehat{\tau}_\ell)\right\}; \Pi_{M} \right]. 
$$ 
By the definition of $\widehat{\tau}_\ell$, on the event $\Pi_{M}$ we have $\sum_{j \in M}\lambda^{j}(\widehat{\tau}_\ell) \geq b$. Thus $\Pro_A(\Pi_M) \leq e^{-b}$. Since the number of such $M$ is no more than $C_{\ell}^{k_1}$, the first inequality in (i) follows from Boole's inequality.

On the other hand, we observe that on the event $\{B_2 \subset \widehat{D}_{\ell}^c\}$ we have
$$
\sum_{j \in B_2} \lambda^{j}(\widehat{\tau}_\ell) \leq -a.
$$
Thus, with a change of measure from $\Pro_A \rightarrow \Pro_{A\setminus B_2}$  we have
$$
\Pro_A(B_2 \subset \widehat{D}_{\ell}^c,\, \Gamma) \leq \Exp_{A\setminus B_2} \left[ \exp\left\{\sum_{j \in B_2} \lambda^{j}(\widehat{\tau}_\ell)\right\};  \Gamma \right] \leq e^{-a} \Pro_{ A\setminus B_2} (\Gamma),
$$
which completes the proof.
\end{proof}

\begin{proof}[Proof of Theorem \ref{leap_error_control}]
We will only establish the upper bound for 
$\Pro_{A}(|A \setminus D_L| \geq k_2)$, since the other inequality can be established similarly.  Observe that
 \begin{align*}
\{ |A \setminus D_L| \geq k_2 \} \;\subset\; 
\bigcup_{B \subset A: |B| = k_2} \{B \subset D_L^c\}.
\end{align*}
Since the union consists of at most $C^{J}_{k_2}$ events, by Boole's inequality it suffices to show that the probability of each event is upper bounded by $2^{k_2} e^{-a}$. Fix an \textit{arbitrary}
$B \subset A$ with $|B| = k_2$. Further observe that
\begin{align*}
&\{B \subset D_L^c\} \subset  \cup_{\ell=0}^{k_1-1}\; \hat{\Gamma}_{B, \ell} \;\; \bigcup \;\; \cup_{\ell=1}^{k_2-1}\; \check{\Gamma}_{B, \ell},\; \text{ where }\;\\
&\widehat{\Gamma}_{B,\ell} :=  \{ B \subset \widehat{D}_\ell^c \} \cap  \{D_L =\widehat{D}_\ell\}, \quad \widecheck{\Gamma}_{B,\ell} := \{ B \subset \widecheck{D}_{\ell}^c \}.
\end{align*}
By Boole's inequality it follows that $\Pro_A(B \subset D_L^c)$ is upper bounded by
\begin{align*}
\sum_{\ell=0}^{k_1-1} \Pro_A(\hat{\Gamma}_{B,\ell}) + \sum_{\ell=1}^{k_2-1} \Pro_A(\check{\Gamma}_{B,\ell})
&\leq \sum_{\ell=0}^{k_1-1} e^{-a} \Pro_{A\setminus B} (D_L =\hat{D}_\ell) 
+  \sum_{\ell=1}^{k_2-1} C^{k_2}_{\ell} e^{-a} \\
&\leq  e^{-a} +  e^{-a} \left( \sum_{\ell=1}^{k_2-1} C^{k_2}_\ell \right) \leq 2^{k_2} e^{-a},
\end{align*} 
where the first inequality follows from Lemma \ref{taus_error_control}, and the second from the fact that 
$\{D_L =\hat{D}_\ell\}$ are disjoint events. Thus, the proof is complete.
\end{proof}

\subsection{Proof of Lemma \ref{taus_upper_bound_leap}}\label{proof:taus_upper_bound} 
\begin{proof}
We will only prove the inequality for $\widehat{\tau}_{\ell}$, as the proof of the inequality for $\widecheck{\tau}_{\ell}$ is similar. Fix $A$ and $0 \leq \ell <k_{1}$. We introduce the following classes of subsets
\begin{align*}
\mathcal{M}_1 &= \{B \subset A: |B| = k_1 - \ell\}, \\ 
\mathcal{M}_0 &= \left\{B \subset A^c: |B| = k_2, \;  \mathcal{I}_0^{i} \geq  \mathcal{I}_0^{(\ell+1)}(A^c) \; \forall \; i \in B \right\}.
\end{align*}
Clearly, we have $\widehat{\tau}_\ell \leq \tau'$, where 
\begin{align*}
\tau':=\inf \{ n \geq 1 &:  \min_{B \in  \mathcal{M}_1} \sum_{i \in B} \lambda^{i}(n) \geq b \; \text{ and } \; \min_{B \in  \mathcal{M}_0} \sum_{j \in B} \lambda^{j}(n) \leq -a,  \\
&\min_{i \in A} \lambda^{i}(n) > 0 \; \text{ and } \; \max_{j \notin A} \lambda^{j}(n) < 0\}.
\end{align*}
Thus, by an application of Lemma \ref{multiple_cross},  we have
$$
\Exp_A[\tau'] \leq  \max\left\{ \frac{b}{ \min_{B \in  \mathcal{M}_1}\sum_{j \in B} I_1^j}  \; , \; \frac{a}{\min_{B \in  \mathcal{M}_0}  \sum_{j \in B} I_0^j } \right\} (1+o(1)) . 
$$
By definition, for any $B_1 \in \mathcal{M}_1$ and $B_0 \in \mathcal{M}_0$  we have  
\begin{align*}
\sum_{j \in B} I_1^j &\geq  \mathcal{D}_1(A;1, k_1 - \ell), \quad
\sum_{j \in B} I_0^j  \geq \mathcal{D}_0(A^c; 1 + \ell, k_2 + \ell) ,
\end{align*}
therefore  we conclude that 
$$
\Exp_A[\tau'] \leq \max \left\{ \frac{b}{\mathcal{D}_1(A; 1,k_1 -\ell)}  \;  , \; \frac{a}{\mathcal{D}_0(A^c; 1 + \ell,k_2 + \ell )} \right\} (1+o(1))  ,
$$
which proves the inequality for $\widehat{\tau}_{\ell}$. 
\end{proof}

\subsection{An important lemma}\label{app:sub_important_lemma}
In this subsection we establish a lemma that is critical in establishing 
the lower bound in Theorem \ref{main}. To state the result, let us denote by 
\begin{align}  \label{class}
\MU_{k_1,k_2}(A) = \{C \subset [J]: |C \setminus A| < k_1 \;\text{ and }\; |A \setminus C| < k_2\},
\end{align}
the  collection of sets that are ``close" to $A$, according to the generalized familywise error rates.
Since $k_1, k_2$ are fixed integers, in order to lighten the notation,   in this subsection we write
$$
L(A;\alpha,\beta) \;\; \text{ for }\;\; L_A(k_1,k_2,\alpha,\beta)
.
$$

\begin{lemma} \label{Two_B_lemma}
Let $A \subset [J]$, $B \in \MU_{k_1,k_2}(A)$,  $\alpha,\beta \in (0,1)$.
\begin{enumerate}
\item If $|B| \geq k_1$ and $|B^c| \geq k_2$, then there exists $B_1^*, B_2^* \subset [J]$ such that 
\begin{align*}
(i)\; |B \setminus B_1^*| = k_1,\; |B_2^* \setminus B| = k_2,\;\quad (ii)\;\frac{|\log(\alpha)|}{\mathcal{I}^{A,B_1^*}} \bigvee \frac{|\log(\beta)|}{\mathcal{I}^{A,B_2^*}} \geq L(A; \alpha, \beta)
\end{align*}
\item If $|B| < k_1$, then there exists $B_2^* \subset [J]$ such that
\begin{align*}
(i)\; |B_2^* \setminus B| = k_2, \; \quad (ii)\;\frac{|\log(\beta)|}{\mathcal{I}^{A,B_2^*}} \geq L(A; \alpha, \beta).
\end{align*}
\item If $|B^c| < k_2 $,  there exists $B_1^* \subset [J]$ such that
\begin{align*}
(i)\; |B \setminus B_1^*| = k_1, \; \quad (ii)\;\frac{|\log(\alpha)|}{\mathcal{I}^{A,B_1^*}}  \geq L(A; \alpha, \beta).
\end{align*}
\end{enumerate}
\end{lemma}

The proof relies on the following two lemmas.

\begin{lemma} \label{G_A_F_lemma}
Let $G \subset A \subset F \subset [J]$. Denote $s_1 = |A \setminus G|$ and 
$s_2 = |F^c|$. Then, for any positive integer $n$ we have
\begin{align*}
D_1(G, 1, n) &\;\leq\; D_1(A, 1 + s_1, n + s_1), \\
D_0(F \setminus A, 1, n) &\;\leq\; D_0(A^c, 1 + s_2, n + s_2) 
\end{align*}
\end{lemma}
\begin{proof}  We start with the first inequality. We can assume $n \leq |G|$, since otherwise both sides are equal to $\infty$.

Fix some $1 \leq i \leq n$. Then clearly the $i^{th}$ smallest element in $\{\MI_1^{j}: j \in G\}$ is not larger than the $(i + |A\setminus G|)^{th}$ element in $\{\MI_1^{j}: j \in A\}$. Thus, the first inequality follows from the definition of the $D_1$ function.

For the second inequality, it follows from the previous argument by replacing $G$ by $F \setminus A$, $A$ by $A^c$, and $\MI_1^{j}$ by $\MI_0^{j}$.
\end{proof}

\begin{lemma}\label{2D_max_lemma}
Let $\ell_1, \ell_2$ be two non-negative integers such that $\ell_1 < k_1$ and $\ell_2 < k_2$.
Then for any $A \subset [K]$, and $\alpha,\beta > 0$, we have
$$
\frac{|\log(\alpha)|}{\mathcal{D}_1(A , 1 + \ell_2 , k_1  - \ell_1 + \ell_2)} \bigvee \frac{|\log(\beta)|}{\mathcal{D}_0(A^c , 1 + \ell_1, k_2  - \ell_2 + \ell_1)} 
\geq L(A; \alpha, \beta).
$$
\end{lemma}
\begin{proof} Let's consider the case that $\ell_1 \geq \ell_2$. When $\ell_1 \leq \ell_2$, the result can be proved in a similar way. Thus, denote $\ell = \ell_1 - \ell_2$. Then
\begin{align*}
&\frac{|\log(\alpha)|}{\mathcal{D}_1(A , 1 + \ell_2 , k_1  - \ell_1 + \ell_2)} \bigvee \frac{|\log(\beta)|}{\mathcal{D}_0(A^c , 1 + \ell_1, k_2  - \ell_2 + \ell_1)} \\
= \;\;& \frac{|\log(\alpha)|}{\mathcal{D}_1(A , 1 + \ell_2, k_1  - l)} \bigvee \frac{|\log(\beta)|}{\mathcal{D}_0(A^c , 1 + \ell + \ell_2, k_2  + \ell)} \\
\geq\;\; &\frac{|\log(\alpha)|}{\mathcal{D}_1(A ,  1, k_1  - \ell)} \bigvee \frac{|\log(\beta)|}{\mathcal{D}_0(A^c ,  1+\ell, k_2  + \ell)} = \widehat{L}_{A}(\ell;\alpha, \beta) \;\; \geq \;\; L(A; \alpha, \beta)
\end{align*}
where the last line used the definition of $\widehat{L}_{A}$ and $L$.
\end{proof}

With above two lemmas, we are ready to present the proof of Lemma \ref{Two_B_lemma}. We illustrate
 the intuition of this proof in Figure \ref{fig:set_vis}.

\begin{proof}  Fix $A$ and $B \in \MU_{k_1,k_2}(A)$.   By definition of the class $\MU_{k_1,k_2}(A)$,
$$ \ell_1 := |B \setminus A| < k_1, \quad \ell_2 := |A \setminus B| < k_2.
$$

First, consider the case that  $|B| \geq k_1$, which implies $|A \cap B| \geq k_1 - \ell_1$. Thus, we can find  
$\Gamma_1 \subset A \cap B$ such that
\begin{align*}
\Gamma_1 = k_1 - \ell_1, \quad
\sum_{i \in \Gamma_1} \MI_1^{i} =  \mathcal{D}_1(A \cap B, 1, k_1 - \ell_1)
\end{align*}
Set $B_1^{*} := A \setminus \Gamma_1$. It is easy to see that
$$
A \setminus B_1^{*} = \Gamma_1, \quad B \setminus B_1^{*} = \Gamma_1 \cup (B\setminus A).
$$
Thus, $|B \setminus B_1^{*}| = k_1$; further, viewing $A \cap B$ as $G$ in the Lemma \ref{G_A_F_lemma},
and since $\ell_2 = |A\setminus B| $, we have
\begin{align*}
\MI^{A,B_1^{*}} = \sum_{i \in \Gamma_1} \MI_1^{i} = \mathcal{D}_1(A \cap B, 1, k_1 - \ell_1) 
\leq  \mathcal{D}_1(A , 1 + \ell_2, k_1 - \ell_1 + \ell_2).
\end{align*}

%
%

Second, consider the case that $|B^c| \geq k_2$, which implies $|A^c \cap B^c| \geq k_2 - \ell_2$. Thus, there exists
$\Gamma_2 \subset A^c \cap B^c$ such that 
\begin{equation*}
\Gamma_2 = k_2 - \ell_2, \quad
\sum_{j \in \Gamma_2} I_0^{j} =  \mathcal{D}_0(A^c \cap B^c, 1, k_2 - \ell_2)
\end{equation*}
Set $B_2^{*} := A \cup \Gamma_2$.  It is easy to see
$$
B_2^{*} \setminus A = \Gamma_2,\quad B_2^{*} \setminus B = \Gamma_2 \cup (A \setminus B).
$$
Then $|B_2^{*} \setminus B| = k_2$. Further, viewing $A \cup (A^c \cap B^c)$ as $F$ in the Lemma \ref{G_A_F_lemma},
and since $\ell_1 = |B\setminus A| = |F^c|$, we have
\begin{align*}
\MI^{A,B_2^*} =  \sum_{j \in \Gamma_2} \MI_0^{j} = \mathcal{D}_0(A^c \cap B^c, 1, k_2 - \ell_2)
\leq \mathcal{D}_0(A^c, 1 + \ell_1, k_2 - \ell_2 + \ell_1)
\end{align*}

It remains to show  $B_1^*$ and $B_2^*$ satisfy the property $(ii)$ in each case.

\noindent\textbf{Case 1}: $|B| \geq k_1$ and $|B^c| \geq k_2$. By the  construction of $B_1^*$ and $B_2^*$ we have
\begin{align*}
&\frac{|\log(\alpha)|}{\mathcal{I}^{A,B_1^*}} \bigvee \frac{|\log(\beta)|}{I^{A,B_2^*}}  \\ 
\geq \;\;& \frac{|\log(\alpha)|}{\mathcal{D}_1(A , \ell_2 + 1, \ell_2 +k_1  - \ell_1)} \bigvee \frac{|\log(\beta)|}{\mathcal{D}_0(A^c , \ell_1 + 1, \ell_1 +k_2  - \ell_2)} \\
\geq \;\; & L(A; \alpha, \beta)
\end{align*}
where the last inequality is due to Lemma \ref{2D_max_lemma}.

%

\noindent\textbf{Case 2}: $|B| < k_1$, which implies the following:
\begin{align*}
|A| = |A\setminus B| + |A \cap B| = \ell_2 + |B| - \ell_1 < \ell_2 + k_1 - \ell_1,
\end{align*}
thus,  $D_1(A, \ell_2 + 1, \ell_2 + k_1 - \ell_1) = \infty$. As a result,
\begin{align*}
 \frac{|\log(\beta)|}{I^{A,B_2^*}}  
\geq & \;\frac{|\log(\alpha)|}{\mathcal{D}_1(A , \ell_2 + 1, \ell_2 +k_1  - \ell_1)} \bigvee \frac{|\log(\beta)|}{\mathcal{D}_0(A^c , \ell_1 + 1, \ell_1 +k_2  - \ell_2)} \\
\geq & \;L(A; \alpha ,\beta)
\end{align*}
where the last inequality is again due to Lemma \ref{2D_max_lemma}.

\noindent\textbf{Case 3}: $|B^c| < k_2$. It can be proved in the same way as in case 2.
%
\end{proof}

\begin{figure}
\centering
\includegraphics[width=0.75\textwidth]{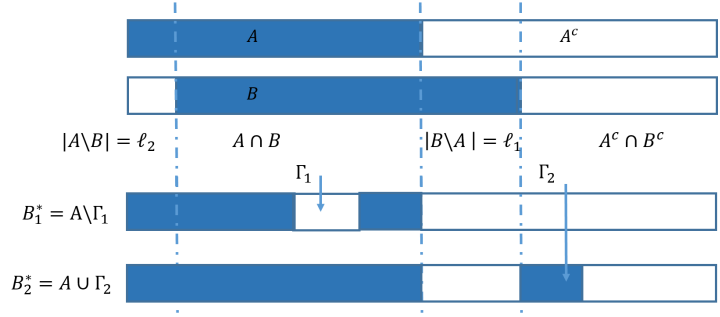}
\caption{The solid area are the streams with signal. The whole set $[J]$ is partitioned into four disjoint sets: $A\setminus B$, $A \cap B$, $B\setminus A$, $A^c \cap B^c$. If $B \in \MU_{k_1,k_2}(A)$, then $\ell_1 < k_1$ and $\ell_2 < k_2$. }
 \label{fig:set_vis}
\end{figure}

\subsection{Proof of Theorem \ref{main}}\label{proof_of_GF_lower_bound}
As explained in the discussion following Theorem \ref{main}, it suffices to show that for any $A \subset [J]$, as $\alpha, \beta \to 0$,
$$
N_A^*(k_1,k_2,\alpha,\beta) \geq  L_A(k_1,k_2,\alpha, \beta) \, (1 - o(1)).
$$ 
Since $k_1, k_2$ are fixed integers, in order  to simplify the   notation in   this subsection we write 
$$
L(A;\alpha,\beta) \;\; \text{ for }\;\; L_A(k_1,k_2,\alpha,\beta).
$$

\begin{proof}
Fix $A \subset [J]$.  By the same argument as in the proof of Theorem \ref{KM_lower_bound}, it suffices to show
 for every  $ q \in (0,1)$  we have:
\begin{align*}
\liminf_{\alpha, \beta \to 0}  \inf_{(T,D) \in \Delta_{k_1, k_2}(\alpha,\beta)} \Pro_A \left(T \geq q L(A; \alpha, \beta) \right) \geq 1.
\end{align*} 

Fix $ q \in (0,1)$ and  let  $(T,D)$ be any procedure in $\Delta_{k_1,k_2}(\alpha,\beta)$. Then, by the  definition of the class $\MU_{k_1,k_2}(A)$  in \eqref{class} we have 
$$
 1 - (\alpha + \beta) \leq \Pro_{A} \left(D \in \MU_{k_1,k_2}(\alpha,\beta) \right) = \sum_{B \in \MU_{k_1,k_2}(\alpha,\beta)}  \Pro_{A}(D=B).
$$
Fix  $B \in \MU_{k_1,k_2}(\alpha,\beta)$, and let $\eta > 0$.
First, we assume that $|B| \geq k_1$ and $|B^c| \geq k_2$. Then  $\Pro_A(D = B)$ is upper bounded by $\text{I}+\text{II}$, where 
\begin{align*} 
&\text{I} = \Pro_A\left( \lambda^{A,B_1^*}(T) < \log(\frac{\eta}{\alpha}), D = B \right)  + \Pro_A\left( \lambda^{A,B_2^*}(T) < \log(\frac{\eta}{\beta}), D = B \right) \\
&\text{II} = \Pro_A\left( \lambda^{A,B_1^*}(T) \geq \log(\frac{\eta}{\alpha} ), \lambda^{A,B_2^*}(T) \geq \log(\frac{\eta}{\beta}),  D = B \right),
\end{align*}
where the sets $B_1^*$ and $B_2^*$ are selected to  satisfy the conditions in Case 1 of Lemma \ref{Two_B_lemma}. Then,  $|B\setminus B_1^*| \geq k_1$ and $|B_2^* \setminus B| \geq k_2$, and consequently 
\begin{align*}
\Pro_{B_1^*}(D = B) \leq \alpha \; \text{ and } \; \Pro_{B_2^*}(D = B) \leq \beta.
\end{align*}
Thus, by change of measure $\Pro_A \rightarrow \Pro_{B_1^*}$ and $\Pro_A \rightarrow \Pro_{B_2^*}$, we have
\begin{align*}
\Pro_A\left( \lambda^{A,B_i^*}(T) < \log\left(\frac{\eta}{\alpha} \right), D = B \right) &\leq \eta, \; \text{ for }\; i = 1,2
\end{align*}
which  shows that $\text{I} \leq 2\eta$.  Moreover,  it is obvious that
\begin{align*}
&\text{II} \leq   \epsilon^{B}_{\alpha,\beta}(T) +   \Pro_A(T \geq q L(A; \alpha, \beta),\; D = B), \quad \text{where} 
\epsilon^{B}_{\alpha,\beta}(T) \; := \\
&\Pro_A \left(  T < q L(A; \alpha, \beta), \;  \lambda^{A,B_1^*}(T) \geq \log\left(\frac{\eta}{\alpha} \right),  \lambda^{A,B_2^*}(T) \geq \log\left(\frac{\eta}{\beta} \right) \right).
\end{align*}
But by the construction of  $B_1^*$ and $B_2^*$ we have
$$ L(A; \alpha, \beta) \leq \ell_{\alpha,\beta} :=  \frac{|\log(\alpha)|}{\mathcal{I}^{A,B_1^*}} \bigvee \frac{|\log(\beta)|}{\mathcal{I}^{A,B_2^*}},
$$
consequently
\begin{align*}
\epsilon^{B}_{\alpha,\beta}(T) \leq \Pro_A \left(T <  q\ell_{\alpha,\beta} ,\; \;  \lambda^{A,B_1^*}(T) \geq \log\left(\frac{\eta}{\alpha} \right),  \lambda^{A,B_2^*}(T) \geq \log\left(\frac{\eta}{\beta} \right) \right),
\end{align*}
and from  Lemma \ref{T_small_ST_larger_lemma} it follows that $\epsilon^{B}_{\alpha,\beta}(T) \to 0$  as $\alpha,\beta \to 0$ uniformly in $T$.   

Putting everything together, we have
\begin{equation}\label{FWER_D_B}
\Pro_A(D = B) \leq 2\eta +\epsilon^{B}_{\alpha,\beta}(T) + \Pro_A(T \geq q L(A; \alpha, \beta), D = B).
\end{equation}
In a similar way we can show that equation \eqref{FWER_D_B} remains valid when   $|B| < k_1$ or $|B^c| < k_2 $. 
Thus summing over $B \in \MU_{k_1,k_2}(A)$ we have
\begin{align*}
\Pro_A(D \in \MU_{k_1,k_2}(A)) &\leq 2Q \eta + \epsilon_{\alpha,\beta}(T) + \Pro_A(T \geq q L(A; \alpha, \beta) ,D \in \MU_{k_1,k_2}(A)),
\end{align*} 
where $Q = |\MU_{k_1,k_2}(A)|$ is a constant, and
$\epsilon_{\alpha,\beta}(T) = \sum_{B \in \MU_{k_1,k_2}(A)} \epsilon^{B}_{\alpha,\beta}(T).$
 Since each summand goes to $0$,  we have $\epsilon_{\alpha,\beta}(T) \to 0$ as $\alpha,\beta \to 0$  uniformly in $T$.   Therefore,
\begin{align*}
\Pro_A(T \geq q L(A; \alpha, \beta)) \geq 1 - (\alpha + \beta) - 2Q\eta - \epsilon_{\alpha,\beta}(T)
\end{align*}
The proof is complete after taking the infimum over the class $\Delta_{k_1,k_2}(\alpha,\beta)$, letting $\alpha,\beta \to 0$ and letting $\eta \to 0$.
\end{proof}

\subsection{Proof of Corollary \ref{delta_0_cor}} \label{proof_of_delta_0_cor}
\begin{proof}
The error control for $\delta_0$ follows by setting 
$\ell = 0$ in  Lemma \ref{taus_error_control}. 
The error control for the \procInt~$\delta_I$ can be established by a simple modification of the proof of Lemma \ref{taus_error_control}. If assumptions \eqref{sym_hom_KM} and \eqref{sym_hom_GF} hold, then from \eqref{fund_GF_limit} it follows that for every $A \subset [J]$ we have
\begin{align*}		
L_A(k_1,k_1, \alpha, \alpha) = \frac{|\log(\alpha)| }{k_1 \MI}.
\end{align*}
Further, setting $\ell = 0$ for $\tau_0$, and $k = 1$ for $T_I$ in the first inequality of Lemma \ref{taus_upper_bound_leap}, we have as $b \to \infty$
\begin{align*}
\Exp_A \left[ \tau_0(b,b) \right]  \;
\leq \frac{b }{k_1 \MI} (1 + o(1)), \quad \Exp_A \left[ \tau_I(b,b) \right]  \;
\leq \frac{b }{\MI} (1 + o(1)).
\end{align*}
Thus, if $b$ is selected as in the statement of  the corollary, then the quantity $L_A(k_1,k_1, \alpha, \alpha)$ provides an asymptotic power bound for both $\Exp_A \left[ \tau_0 \right]$  and  $\Exp_A \left[ \tau_I \right]$.  
Thus, the proof is complete.
\end{proof}

\subsection{Proof of Theorem \ref{GF_comp_fix_sample_rule}} \label{proof_of_GF_comp_fix_sample_rule}
\begin{proof}
Since $k_1, d$ are fixed,  for simplicity we write $n^*(\beta)$ and $\widehat{n}(\beta)$ for $n^*(k_1,k_1,\beta^d,\beta)$  and $\widehat{n}_{NP}(k_1,k_1,\beta^d,\beta)$, respectively.

\noindent (i) Let us first focus on $n^*(\beta)$. By its definition \eqref{def:fix_ns}, there exists some 
$$D^*(\beta) \in \Delta_{fix}(n^*(\beta)) \cap \Delta_{k_1,k_1}(\beta^d,\beta).$$

Fix  any $A \subset [J]$ such that $|A| = 2k_1-1$.
Denote $\Pro$ the probability measure for data in all streams, and consider the  simple versus simple testing problem \eqref{aux_simple_vs_simple} and the following procedure
$\widetilde{D}^*(\beta) := 
\begin{cases}
0 \;\; \text{ if } |D^*(\beta)| < k_1 \\ 
1 \;\; \text{ if } |D^*(\beta)| \geq k_1. 
\end{cases}
$
Then,  by definition of $D^*(\beta)$ we have
\begin{align*}
&\Pro_{\emptyset}(\widetilde{D}^*(\beta) = 1) = 
\Pro_{\emptyset}(|D^*(\beta)| \geq k_1) \leq \alpha = \beta^d, \\
&\Pro_{A}(\widetilde{D}^*(\beta) = 0) = 
\Pro_{A}(|D^*(\beta)| < k_1) \leq \beta,
\end{align*}
and by  the generalized Chernoff's Lemma \ref{generalized_chernoff_lemma},
\begin{align*}
\liminf_{\beta \to 0}\frac{\log(\beta)}{n^*(\beta)} &\geq 
\liminf_{\beta \to 0}\frac{1}{n^*(\beta)}
\log\left(\frac{1}{2} \Pro_{\emptyset}^{1/d}(\widetilde{D}^*(\beta) = 1) 
+\frac{1}{2}\Pro_{A}(\widetilde{D}^*(\beta) = 0)
\right) \\
&\geq - \frac{\Phi^{A}(\widetilde{h}_d^{A})}{d}.
\end{align*}
where $\widetilde{h}_d^{A}$ is a solution to $\Phi^{A}(z)/d = \Phi^{A}(z) - z$, and for any $z \in \bR$
\begin{align*}
\Phi^{A}(z) &:= \sup_{\theta \in \bR} 
 \left\{z\theta - \sum_{j \in A} \log \left(\Exp_{0}^{j}
 \left[ e^{\theta \lambda^j(1)} \right]
 \right)\right\} \\
&= \sup_{\theta \in \bR} \left\{z\theta - |A| \log \left(\Exp_{0}^{1}\left[
 e^{\theta \lambda^1(1)} \right]
\right)\right\} = |A| \, \Phi(z/|A|).
\end{align*}
Here, the second equality is due to homogeneity \eqref{homo_iid_case}.
By definition \eqref{h_hat_d},  ${\Phi({h}_d)}/{d} = \Phi({h}_d) - {h}_d$, which implies  
$${\Phi^{A}(|A| {h}_d)}/{d} = \Phi^{A}(|A| {h}_d) - (|A| {h}_d).$$
Thus,  $\widetilde{h}_d^{A} = |A| {h}_d$, and
$$
{\Phi^{A}(\widetilde{h}_d^{A})}/{d} = |A| \Phi({h}_d)/{d} = \frac{2k_1-1}{d} \Phi({h}_d),
$$
which completes the proof of (i).

\noindent (ii) We now focus on $\widehat{n}(\beta)$. 
By definition, there exists  $h_{\beta} \in \bR$ such that
$$
(\widehat{n}(\beta), \widehat{D}(\beta)) \in \Delta_{k_1,k_1}(\beta^d, \beta),
\text{ where }
 \widehat{D}(\beta) :=  D_{NP}(\widehat{n}(\beta), h_{\beta} \mathbf{1}_{J} ),
$$
where $\mathbf{1}_{J} \in  \bR^{J}$ is a vector of all ones.
Due to homogeneity \eqref{homo_iid_case}, set
\begin{align*}
p_{\beta} := \Pro_0^{1}(\widehat{D}^{1}(\beta) = 1) = 
\Pro_0^{1}\left( \lambda^1(\widehat{n}(\beta)) > h_{\beta} \,  \widehat{n}(\beta) \right) , \\
q_{\beta} := \Pro_1^{1}(\widehat{D}^{1}(\beta) = 0) = 
\Pro_1^{1}\left( \lambda^1(\widehat{n}(\beta)) \leq h_{\beta} \,  \widehat{n}(\beta) \right).
\end{align*}

For any $A \subset [J]$ such that $|A| = k_1 (= k_2)$,
\begin{align*}
\beta^d &\geq \Pro_{\emptyset}\left( \bigcap_{j \in A}\{\tilde{D}(\alpha)^{j} = 1\} \right)
= (p_{\beta})^{k_1},\;\; \\
 \beta &\geq \Pro_{[J]} \left( \bigcap_{j \in A}\{\tilde{D}(\alpha)^{j} = 0\} \right)
= (q_{\beta})^{k_1},
\end{align*}
which implies that 
\begin{align*}
\frac{1}{\widehat{n}(\beta)}\frac{\log(\beta)}{ k_1} 
\geq 
\frac{1}{\widehat{n}(\beta)}
\log\left(\frac{1}{2}
p_{\beta}^{1/d} + \frac{1}{2} q_{\beta} \right). 
\end{align*}
Then, again by the generalized Chernoff's lemma \ref{generalized_chernoff_lemma} we have
\begin{align*}
\lim\inf_{\beta \to 0}
\frac{\widehat{n}(\beta)}{|\log(\beta)|}
= \frac{d}{k_1 \Phi({h}_d)}.
\end{align*}
Further, the same argument shows that the equality is obtained with $h = {h}_d$, which completes the proof of (ii).
\end{proof}

\section{Sequential multiple testing with composite hypotheses}\label{sec:app_composte_case}
In this section, we prove  Theorem~\ref{thm:comp_main} in Section~\ref{sec:composite}. We  first establish a universal asymptotic lower bound on the expected sample size of  procedures that control  generalized familywise error rates under composite hypotheses (Subsec.~\ref{subsec:app_lower}).  Then, we show that this lower bound is achieved by the Leap rule with the adaptive log-likelihood statistics in \eqref{adaptive_stat} (Subsec.~\ref{Error control in composite case} and~\ref{app_composite_AO}). Further, we demonstrate numerically that  the Intersection rule \eqref{intersection_rule} and the asymmetric Sum-Intersection rule  \eqref{tau_0_def} with the adaptive statistics fail to achieve  asymptotic optimality in the composite case (Subsec.~\ref{app:simulation_composite}). We conclude this section with a discussion on the adaptive statistics and alternative local test statistics (Subsec.~\ref{comp_discussion}).

\subsection{Lower bound on the expected sample size}\label{subsec:app_lower}
Fix any $A \subset [J]$ and $\butheta = (\theta^1, \ldots, \theta^J)\in \boldsymbol{\Theta}_A$.

\textbf{Case 1:} Assume \textit{for now} that  the infima in~\eqref{comp_info_discrete} are attained, i.e.,  there exists $\widetilde{\butheta} = (\widetilde{\theta}^1, \ldots, \widetilde{\theta}^J) \in \boldsymbol{\Theta}_{A^c}$ such that
\begin{align*}
&\cI_0^j(\theta^j) = I^j(\theta^j, \widetilde{\theta}^j) \text{ for every } j \in A^c,\\
&\cI_1^j(\theta^j) = I^j(\theta^j, \widetilde{\theta}^j) \text{ for every } j \in A.
\end{align*}
Any procedure $(T,D)\in \Delta_{k_1,k_2}^{comp}(\alpha, \beta)$ controls the generalized familywise error rates 
below $\alpha$ and $\beta$ when applied to the multiple testing problem with the following \textit{simple} hypotheses for each stream:
\begin{align*}
&{\Hyp}_0^{j} \;: \gamma^{j} = \theta^{j} \;\; \text{ versus }\;\; {\Hyp}_1^{j}\;: \gamma^{j} = \widetilde{\theta}^{j}, \quad j \in A^c, \\
&{\Hyp}_0^{j} \;: \gamma^{j} = \widetilde{\theta}^{j} \;\; \text{ versus }\;\; {\Hyp}_1^{j}\;: \gamma^{j} = \theta^{j}, \quad j \in A, 
\end{align*}
where we write $\gamma^{j}$ for the generic local parameter in j-th stream to distinguish it from the j-th component of $\butheta$.

Then, under assumptions~\eqref{composite_complete_convergence} and~\eqref{comp_separable},  by Theorem~\ref{main} we have
\begin{align}\label{app_low_discrete}
\liminf_{\alpha, \beta \to 0} \; N^*_{A,\butheta}(k_1,k_2,\alpha,\beta)/L_{A,\butheta}(k_1,k_2,\alpha,\beta) \geq 1, 
\end{align}
where  
\begin{align}\label{comp_quantities_def}
\begin{split}
&L_{A,\butheta}(k_1,k_2,\alpha,\beta):= \min \left\{ \min_{0 \leq \ell <k_1}  \widehat{L}_{A,\butheta}(\ell; \alpha,\beta)   \; ,  \; \min_{0 \leq \ell <k_2} \widecheck{L}_{A,\butheta}(\ell;\alpha,\beta) \right\},\\
&\widehat{L}_{A,\butheta}(\ell; \alpha,\beta) := \max \left\{ 
\frac{|\log(\alpha)|}{\mathcal{D}_{1}(A,\butheta; 1, k_1 -\ell)} \; , \;  \frac{|\log(\beta)|}{\mathcal{D}_{0}(A^c,\butheta; \ell+1,  \ell+ k_2)} \right\}, \\
&\widecheck{L}_{A,\butheta}(\ell;\alpha,\beta) :=  \max \left\{ 
\frac{|\log(\alpha)|}{\mathcal{D}_1(A,\butheta;\ell+1, \ell+k_1)}  \; , \; \frac{|\log(\beta)|}{\mathcal{D}_0(A^c,\butheta;1 ,k_2 - \ell )}\right\},\\
&\mathcal{D}_{1}(A,\butheta;\ell, u) 
= \sum_{j=\ell}^{u} \cI^{(j)}_{1}(A,\butheta),\quad
\mathcal{D}_{0}(A^c,\butheta;\ell, u) 
= \sum_{j=\ell}^{u} \cI^{(j)}_{0}(A^c,\butheta),
\end{split}
\end{align}
and 
$$\cI^{(1)}_1(A,\butheta) \leq \ldots \leq  \cI^{(|A|)}_1(A,\butheta)$$
is  the  increasingly ordered sequence of $\{\cI_1^{j}(\theta^j), j \in A\}$, and
$$\cI^{(1)}_0(A^c,\butheta) \leq \ldots \leq  \cI^{(|A^c|)}_0(A^c,\butheta)$$
is  the  increasingly ordered sequence of $\{\cI_0^{j}(\theta^j), j \in A^c\}$.
As before, the convention is that
$$
\cI^{(j)}_1(A,\butheta) = \infty \; \text{ if } j > |A|,\quad
\cI^{(j)}_0(A^c,\butheta) = \infty \; \text{ if } j > |A^c|.
$$

\textbf{Case 2: } In general, the infima in~\eqref{comp_info_discrete} are not attained. However, under the separability assumption~\eqref{comp_separable}, for any $\epsilon > 0$ there exists $\widetilde{\butheta}_{\epsilon} 
= (\widetilde{\theta}^{1}_{\epsilon}, \ldots, \widetilde{\theta}^{J}_{\epsilon})
\in \boldsymbol{\Theta}_{A^c}$ such that
\begin{align*}
&I^j(\theta^j, \widetilde{\theta}_{\epsilon}^j) \leq (1+\epsilon) \,  \cI_0^j(\theta^j) \text{ for any } j \in A^c,\\
&I^j(\theta^j, \widetilde{\theta}_{\epsilon}^j) \leq (1+\epsilon) \, \cI_1^j(\theta^j)\text{ for any } j \in A.
\end{align*}
Applying again Theorem~\ref{main} to the following multiple testing problem with simple hypotheses:
\begin{align*}
&{\Hyp}_0^{j} \;: \gamma^{j} = \theta^{j} \;\; \text{ versus }\;\; {\Hyp}_1^{j}\;: \gamma^{j} = \widetilde{\theta}_{\epsilon}^{j}, \quad j \in A^c, \\
&{\Hyp}_0^{j} \;: \gamma^{j} = \widetilde{\theta}_{\epsilon}^{j} \;\; \text{ versus }\;\; {\Hyp}_1^{j}\;: \gamma^{j} = \theta^{j}, \quad j \in A, 
\end{align*}
we have
$$
\liminf \; N^*_{A,\butheta}(k_1,k_2,\alpha,\beta)/L_{A,\butheta}(k_1,k_2,\alpha,\beta) \geq 1/(1+\epsilon).
$$
Since $\epsilon$ is arbitrary,~\eqref{app_low_discrete} still holds.

Above discussions leads to the following theorem.
\begin{theorem} \label{comp_lower_bound}
If \eqref{composite_complete_convergence} and \eqref{comp_separable} hold,  then \eqref{app_low_discrete} holds  for every $A \subset [J]$ and $\butheta \in \boldsymbol{\Theta}_A$.  
\end{theorem}

\subsection{Error control of the Leap rule with adaptive log-likelihood ratios} \label{Error control in composite case}

We start with the following observation.
\begin{lemma}\label{comp_martingale}
Fix $A \subset [J]$, $\butheta=(\theta^1, \ldots, \theta^J) \in \Theta_{A}$. For each $j \in [J]$,
$$L_n^j := \exp \left( {\ell}^j_{*}(n) - \ell^j(n, \theta^j) \right), \;\; n \in \bN
$$
is an $\{\cF_n\}$-martingale under $\Pro_{A, \butheta}$ with expectation $1$.

\begin{proof}
By definition, 
\begin{align*}
L_n^j= L^j_{n-1} \cdot  \frac{p^j_{\widehat{\theta}_{n-1}^{j}}\left(X^j(n) \vert \cF^{j}_{n-1} \right)}
{p^j_{\theta^{j}}\left( X^j(n) \vert \cF^{j}_{n-1} \right)} \;\;.
\end{align*}
Clearly, $L_n^{j} \in \cF_n$ for any $n \in \bN$.
Further, since $\widehat{\theta}_{n-1}^{j} \in \cF_{n-1}$, 
\begin{align*}
\Exp_{A,\butheta} \left[ \left.
\frac{p^j_{\widehat{\theta}_{n-1}^{j}} \left(X^j(n) \vert \cF^{j}_{n-1} \right)}
{p^j_{\theta^{j}}\left(X^j(n) \vert  \cF^{j}_{n-1}  \right)} \right\vert \cF_{n-1}
\right] 
= \int \, \frac{p^j_{\widehat{\theta}_{n-1}^{j}} \left(z \vert \cF^{j}_{n-1} \right)}
{p^j_{\theta^{j}}\left(z \vert  \cF^{j}_{n-1}  \right)} 
p^j_{\theta^{j}}\left(z \vert  \cF^{j}_{n-1}  \right)
= 1,
\end{align*}
which implies $\Exp_{A,\butheta}[L^j_n \vert \cF_{n-1}] = L^j_{n-1}$. Further, since $\widehat{\butheta}_0$ is deterministic, $\Exp_{A,\butheta}[L^j_1] = 1$, which completes the proof.
\end{proof}

By Lemma~\ref{comp_martingale} and due to independence across streams, 
for any subset $M \subset [J]$, 
there exists a probability measure $Q_{A,\butheta,M}$ such that for any $n \in \bN$,
\begin{equation}
\label{composite_change_of_measure}
\frac{d Q_{A,\butheta,M}}{d \Pro_{A,\butheta}}(\cF_n)
\;=\;
\prod_{j \in M} \exp\left({\ell}^j_{*}(n) - \ell^j(n,\theta^j) \right).
\end{equation}

Next, we establish the error control of the Leap rule with adaptive log-likelihood ratios. The proof is 
almost identical to  Theorem \ref{leap_error_control}.

\begin{theorem}\label{comp_error_control}
Assume \eqref{comp_info_discrete} and \eqref{comp_separable} hold. For any $\alpha, \beta \in (0,1)$ we have that  the Leap rule $\delta_L^{*}(a,b) \in \Delta_{k_1,k_2}^{comp}(\alpha,\beta)$ when the thresholds are selected as follows:
 \begin{equation*}
  a = |\log(\beta)| + \log(2^{k_2} C^{J}_{k_2}), \quad b = |\log(\alpha)| + \log(2^{k_1} C^{J}_{k_1}).
 \end{equation*}
\end{theorem}
\begin{proof}
Just as  Theorem \ref{leap_error_control} follows from Lemma~\ref{taus_error_control} (see the proof in Appendix~\ref{proof:leap_error_control}), in the same way  Theorem  \ref{comp_error_control} follows by the next Lemma.
\end{proof}

\begin{lemma} \label{app_comp_taus_error_control}
Assume \eqref{comp_info_discrete}, \eqref{comp_separable} hold. 
Fix  $A \subset [J]$, $\butheta = (\theta^1,\ldots, \theta^J) \in \boldsymbol{\Theta}_A$. Let $B_1 \subset A^c$ with $|B_1| = k_1$, and $B_2 \subset A$ with $|B_2| = k_2$.
\begin{enumerate}[label={(\roman*)}]
\item Fix any $0 \leq \ell < k_1$. For any event $\Gamma \in \MF_{\widehat{\tau}_\ell}$,  we have
\begin{equation*}
\Pro_{A,\butheta}(B_1 \subset \widehat{D}^{*}_{\ell}) \leq  C_{\ell}^{k_1} e^{-b},\quad
\Pro_{A,\butheta}(B_2 \subset (\widehat{D}^{*}_{\ell})^c,\, \Gamma) \leq  e^{-a} Q_{A,\butheta, B_2} (\Gamma).
\end{equation*}
\item Fix any $0 \leq \ell < k_2$. For any event $\Gamma \in \MF_{\widecheck{\tau}_\ell}$,   we have
\begin{equation*}
\Pro_{A,\butheta}(B_1 \subset \widecheck{D}^{*}_\ell,\, \Gamma) \leq  e^{-b}Q_{A,\butheta, B_1} (\Gamma) ,\quad
\Pro_{A,\butheta}(B_2 \subset (\widecheck{D}^{*}_\ell)^c) \leq  C_{\ell}^{k_2} e^{-a}.
\end{equation*}
\end{enumerate}
\end{lemma}

\begin{proof}
The proof is similar to that of Lemma~\ref{taus_error_control}. We only indicate the differences by working out the first inequality in (i).

As in the proof of Lemma~\ref{taus_error_control}, by definition, $\widehat{D}^{*}_{\ell}$ rejects the nulls in the $\ell$ streams with the least significant non-positive LLR, in addition to the nulls in the streams with positive LLR. Thus, 
$$
\{B_1 \subset \widehat{D}^{*}_{\ell}\} \subset \bigcup_{M \subset B_1, |M| = k_1 - \ell} \Pi_{M} \;, \quad \text{ where }\;
\Pi_{M} := \{\lambda^j_{*}(\widehat{\tau}_\ell) > 0 \;\; \forall j \; \in M\},
$$
and by Boole's inequality it suffices to show that $\Pro_{A,\butheta}(\Pi_{M}) \leq e^{-b}$ for every
$M \subset B_1$ with $|M| = k_1 - \ell$.

By definition, for any $j \in M \subset B_1 \subset A^c$,  since $\theta^j \in \Theta^j_0$,
$$
{\ell}^j_0(n) \geq \ell^j(n, \theta^j) \;\; \text{ for any } n \in \bN.
$$
Then,  by the definition of the adaptive log-likelihood ratio statistics~\eqref{adaptive_stat}, we have
$$
\Pi_{M} \; \subset \; \left\{ \sum_{j \in M} \left({\ell}^j_{*}(\widehat{\tau}_\ell) - {\ell}^j_0(\widehat{\tau}_\ell)\right) \geq b \right\}
\; \subset \;
\left\{ \sum_{j \in M} \left({\ell}^j_{*}(\widehat{\tau}_\ell) - \ell^j(\widehat{\tau}_\ell,\theta^j) \right) \geq b \right\}.
$$

By the above observation, the definition of $Q_{A,\butheta,M}$~\eqref{composite_change_of_measure}, and likelihood ratio identity, on the event $\Pi_M$,
$$
\frac{d Q_{A,\butheta,M}}{d \Pro_{A,\butheta}}(\cF_{\widehat{\tau}_{\ell}}) \geq e^{b},
$$
and the proof is complete by changing the measure from $\Pro_{A,\butheta}$ to 
$Q_{A,\butheta,M}$.
\end{proof}
\end{lemma}

\subsection{Asymptotic optimality of the Leap rule with adaptive log-likelihood ratios}
\label{app_composite_AO}
The asymptotic optimality follows after we establish an asymptotic upper bound on the expected sample size of the Leap rule. The following result is similar to Lemma~\ref{taus_upper_bound_leap}.

\begin{lemma}\label{comp_leap_upper_bound}
Assume~\eqref{comp_separable} and~\eqref{composite_uniform_convergence} hold. For any $A \subset [J]$ and 
$\butheta \in \Theta_A$, as $a,b \to \infty$,
\begin{align*}
\Exp_{A,\butheta}[\widehat{\tau}_\ell] &\leq \max\left\{ \frac{b(1 + o(1)) }{\mathcal{D}_1(A,\butheta; 1, k_1 -\ell)} \;  , \;  \frac{a(1 + o(1))}{\mathcal{D}_0(A^c, \butheta; \ell+1 ,\ell+ k_2)} \right\} ,
  0 \leq \ell <k_{1},  \\
\Exp_{A,\butheta}[\widecheck{\tau}_\ell] &\leq  \max\left\{ \frac{b(1 + o(1))}{\mathcal{D}_1(A, \butheta; \ell+1,\ell+k_1)} \,  , \,  \frac{a(1 + o(1))}{\mathcal{D}_0(A^c, \butheta; 1 , k_2 - \ell ) } \right\},  0 \leq \ell <k_{2},
\end{align*}
where the denominators are defined in~\eqref{comp_quantities_def}.
\end{lemma}
\begin{proof}
Under assumption~\eqref{composite_uniform_convergence}, the proof uses the same argument
as in that  for Lemma~\ref{taus_upper_bound_leap} in Subsection~\ref{proof:taus_upper_bound}.
\end{proof}

Now Theorem~\ref{thm:comp_main} follows from  Theorem \ref{comp_lower_bound}, Lemma \ref{comp_error_control} and Lemma \ref{comp_leap_upper_bound}.

\subsection{Simulations for composite case}\label{app:simulation_composite}
Here we consider a ``homogeneous" multiple testing problem on the normal means with known variance.
Specifically, we assume that for each $j \in [J]$, the observations in the $j$-th stream, $\{X^j(n): n\in \bN\}$,  are i.i.d. with common distribution  $\mathcal{N}(\theta^j, 1)$, and for a given constant $\mu > 0$, that does not depend on $j$, we want to test
\begin{align}\label{composite_normal_mean}
{\Hyp}_0^j: \;  \theta^j \leq 0 \;\; \text{ versus }\;\; {\Hyp}_1^j: \; \theta^j \geq \mu.
\end{align}
Using  $\mathcal{N}(0,1)$ as our reference measure,   for each $j \in [J]$ we have
\begin{equation}
\label{normal_mean_loglikelihood}
\ell^{j}(n, \theta^j) = n\left( \theta^{j} \, \overline{X}^{j}(n)  - \frac{1}{2}(\theta^{j})^2 \right), \quad
\text{ where } \overline{X}^{j}(n) := \frac{1}{n} \sum_{i=1}^{n} X^{j}(i).
\end{equation}
Further, for any $\theta^{j}, \widetilde{\theta}^{j}$, we have $I^{j}(\theta^j, \widetilde{\theta}^j) = \frac{1}{2}(\theta^j - \widetilde{\theta}^{j})^2$, and 
$$
I^j_0(\theta^j) = \frac{1}{2}(\theta^{j} - \mu)^2 \text{ for } \theta^j \leq 0, \qquad
I^j_1(\theta^j) = \frac{1}{2}(\theta^{j})^2 \text{ for } \theta^j \geq \mu.
$$
Clearly, the null and the alternative hypotheses are separated in the sense of~\eqref{comp_separable}. Further,  condition~\eqref{composite_complete_convergence} is satisfied due to~\cite{hsu1947complete}.

The adaptive log-likelihood process~\eqref{adaptive_LL} for the j-th stream in this context takes the following form:
$\ell^j_{0} = 0$, and for $n \geq 1$,
\begin{equation}
\label{normal_mean_adaptive}
\ell^{j}_{*}(n) = \sum_{i=1}^{n}\left({X}^{j}(i) \, \widehat{\theta}_{i-1}^{j} - \frac{1}{2}(\widehat{\theta}_{i-1}^{j})^2 \right).
\end{equation}
If we choose to use the maximum likelihood estimators $\{\widehat{\butheta}_n\}$ in above definition, i.e., $\widehat{\theta}^{j}_n = \overline{X}^{j}(n)$, the one-sided complete convergence condition~\eqref{composite_uniform_convergence} is established in~\cite{tartakovsky2014sequential} (Page 278-279). Thus, by Theorem~\ref{thm:comp_main}, the Leap rule is asymptotically optimal in this setup.

To distinguish from the simulations in the simple versus simple setup, we refer to the Leap rule with adaptive statistics as ``Leap*" rule. We will compare the Leap* rule with the following procedures:
\begin{enumerate}
\item \textit{Asymmetric Sum-Intersection* rule}: replace the log-likelihood ratio statistics $\lambda^{j}(n)$, in the definition of the asymmetric Sum-Intersection rule~\eqref{tau_0_def}, by the adaptive version $\lambda_*^{j}(n)$~\eqref{adaptive_LL}.

\item \textit{Intersection* rule}: replace the log-likelihood ratio statistics $\lambda^{j}(n)$, in the definition of the Intersection rule~\eqref{intersection_rule}, by the adaptive version $\lambda_{*}^{j}(n)$~\eqref{adaptive_LL}.

\item \textit{MNP rule}: for a fixed-sample size $n$, in each stream, we run the Neyman-Pearson rule with the same threshold $h >0$, which is the most powerful test for each stream due to the monotone likelihood ratio property. Formally,
\begin{align*}
\delta_{NP}(n,h):= (n,D_{NP}(n,h)), \; \; D_{NP}(n,h):= \{j \in [J]:  \overline{X}^{j}(n)>  h \},
\end{align*}

\end{enumerate}

For simulation purposes, we assume that the tolerance on the two types of mistakes is the same, in the sense that~\eqref{sym_hom_GF} holds. As in Section~\ref{sec:composite}, we denote the true parameter as $(A, \butheta)$, where
$\butheta =(\theta^1, \ldots, \theta^J) \in \boldsymbol{\Theta}_A$.

\subsubsection{Thresholds selection via simulation}
For each $j \in [J]$ and $\theta^j \leq 0$
the distribution of $\{\lambda^{j}_{*}(n): n \in \bN\}$ 
under $\Pro^{j}_{\mu - \theta^j}$ is the same as the distribution of $\{-\lambda^{j}_{*}(n): n \in \bN\}$ 
under $\Pro^{j}_{\theta^j}$.  Since~\eqref{sym_hom_GF} holds, we should equate the thresholds $a$ and $b$ 
in the Leap* rule. Further, we only need to focus on the generalized familywise error rate of Type I.

For a fixed parameter $a$ ($= b$),  we  use simulation to find out the maximal probability of the Leap* rule committing $k_1$ false positive mistakes, i.e.
$$
\max_{(A, \butheta): A \subset [J], \butheta \in \boldsymbol{\Theta}_A} \Pro_{A,\butheta}(\vert D_L^{*} \setminus A \vert \geq k_1).
$$
Then we try different values for $a$ and  select the one for which  the  above quantity is equal to $\alpha$.  Note that the maximum is over $\butheta \in \boldsymbol{\Theta}$. However, for $\theta^j \leq \widetilde{\theta^{j}}$,
$\{\lambda^{j}_{*}(n): n \in \bN\}$ under $\Pro_{\widetilde{\theta}^j}$ is stochastically larger than
$\{\lambda^{j}_{*}(n): n \in \bN\}$ under $\Pro_{{\theta}^j}$, in the sense that for any $n \in \bN$ and $x \in \bR$,
$$
\Pro_{{\theta}^j}^{j}(\lambda^{j}_{*}(n) \leq x) \; \geq \; 
\Pro^{j}_{\widetilde{\theta}^j}(\lambda^{j}_{*}(n) \leq x).
$$
As a result, the maximal probability is achieved by the boundary cases, i.e., $\butheta \in \{0,\mu\}^{J}$. 

The same discussion applies to the other two sequential procedures. For the MNP rule,~\eqref{sym_hom_GF} implies that 
$h = \frac{1}{2}\mu$, and for a fixed $n$, the maximal probability of making $k_1$ false positives is also achieved by 
$\butheta \in \{0,\mu\}^{J}$. 

\subsubsection{Practical considerations}
The first few estimators of $\butheta$ will typically be
quite noisy, since they are estimated based on only a few observations.  However,
from~\eqref{adaptive_LL} or \eqref{normal_mean_adaptive}  we observe that their effect will persist. Thus, in practice it is preferable to take an initial sample of fixed size, say $n_0$, and use these observations \textit{only} to obtain good initial estimates of  the unknown parameter.

Specifically,  we assume that for each $j \in [J]$, $X^{j}(-n_0),\ldots, X^{j}(-1)$ are i.i.d. with distribution $\mathcal{N}(\theta^j,1)$, and  we define for $n \geq 0$  the following maximum likelihood estimator 
$$
\widehat{\theta}^j_n := \frac{\sum_{i=-n_0}^{-1}X^{j}(i) + \sum_{i=1}^{n} X^{j}(i)}{n_0 + n},
$$
which includes the initial samples. The definitions of the log-likelihood process~\eqref{normal_mean_loglikelihood} and the adaptive log-likehood process~\eqref{normal_mean_adaptive} remain unchanged. By taking an initial sample of fixed size, the asymptotic expected sample size of the Leap* rule is not affected. Further, if we enlarge the $\sigma$-field by including the initial samples,  i.e.,
$$
\widetilde{\cF}_n :=  \cF_n \; \vee \; \sigma\left(X^{j}(i): j \in [J], i \in \{-n_0,\ldots, -1 \} \right),
$$
then the key Lemma~\ref{comp_martingale}, used to establish the error control of Leap* rule, still holds. Thus,
taking an initial sample does not affect the asymptotic optimality of the Leap* rule.

\subsubsection{Simulation results}
We consider the problem~\eqref{composite_normal_mean} with $J =20$,  $\mu = 0.2$, $k_1 =k_2 = 2$ and the initial sample size $n_0 = 10$. Based on the previous discussion, we set $a = b$ for the sequential methods.
For a fixed threshold $a$, we use simulation to find out the maximal probability (over $\butheta \in \boldsymbol{\Theta}$) of committing $k_1$ false positives  (Err), and the expected sample size (ESS) under a particular $\Pro_{A, \butheta}$, where 
$A = \{1,\ldots, 10\}$ and
\begin{equation}
\label{simulation choice}
\butheta = (\theta^1,\ldots, \theta^J),\;\;
\theta^{j} = \begin{cases}
0.7 \;\;\;\;\; \text{ if } j = 1,\ldots,10 \\
-0.3 \;\; \text{ if } j = 11,\ldots, 19 \\
0 \;\;\;\;\;\;\;\; \text{ if } j = 20.
\end{cases}
\end{equation}

For the MNP rule, we set $h = \frac{1}{2}\mu$, and use simulation to find out the maximal probability of committing $k_1$ false positives for each fixed $n \in \bN$. The results are shown in Figure~\ref{fig:composite}.

From Figure~\ref{fig:composite}, we observe that the other procedures have a different ``slope'' compared to the  asymptotically optimal Leap* rule, which indicates that they  fail to be asymptotically optimal. 
Further, since sequential methods are adaptive to the true $\butheta$, the gains over fixed-sample size procedures increase  as $\butheta$  is farther from the boundary cases.

\begin{figure}[htbp]
\subfloat{
\includegraphics[width=0.42\linewidth]{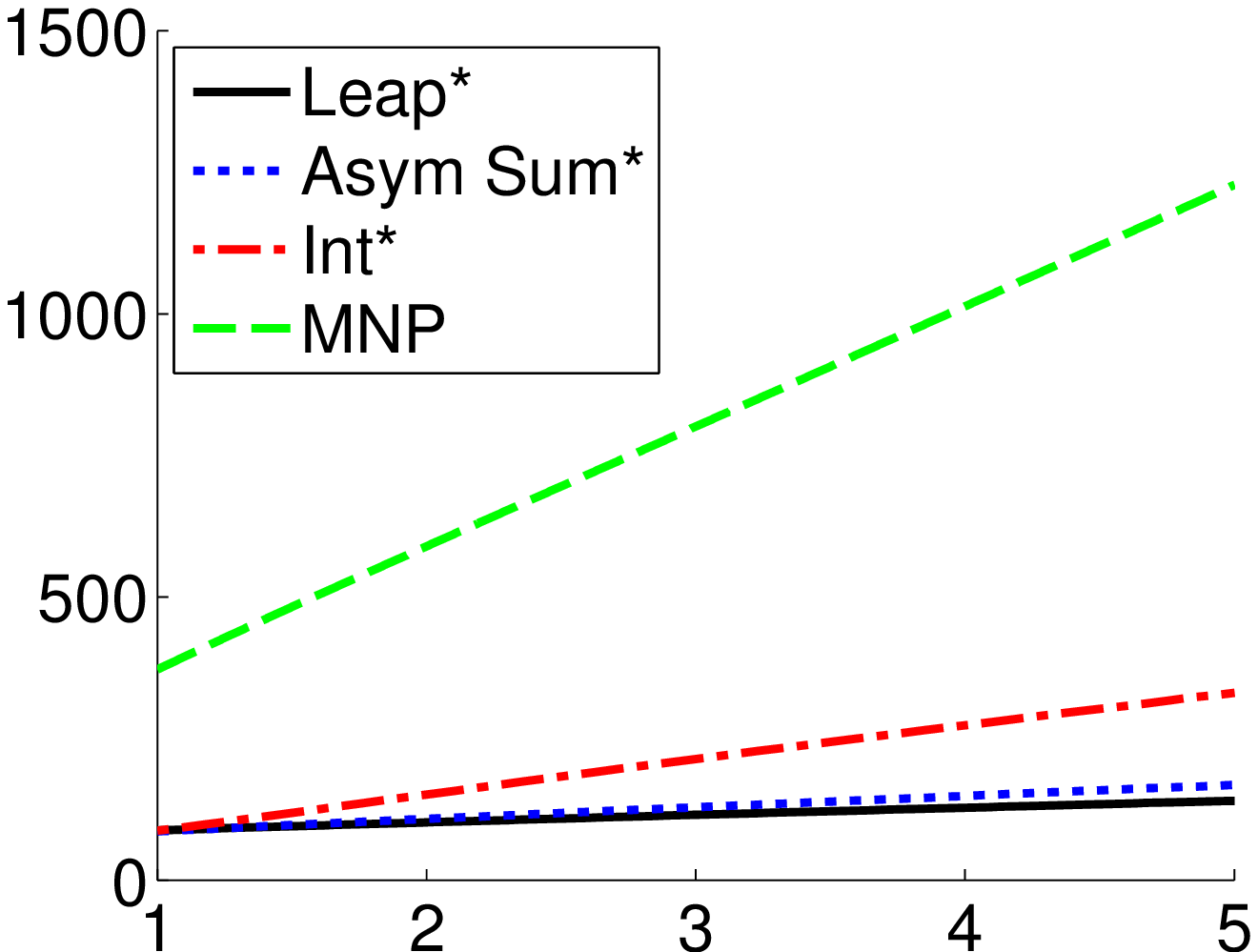} 
}\hspace{1cm}
\subfloat{
\includegraphics[width=0.42\linewidth]{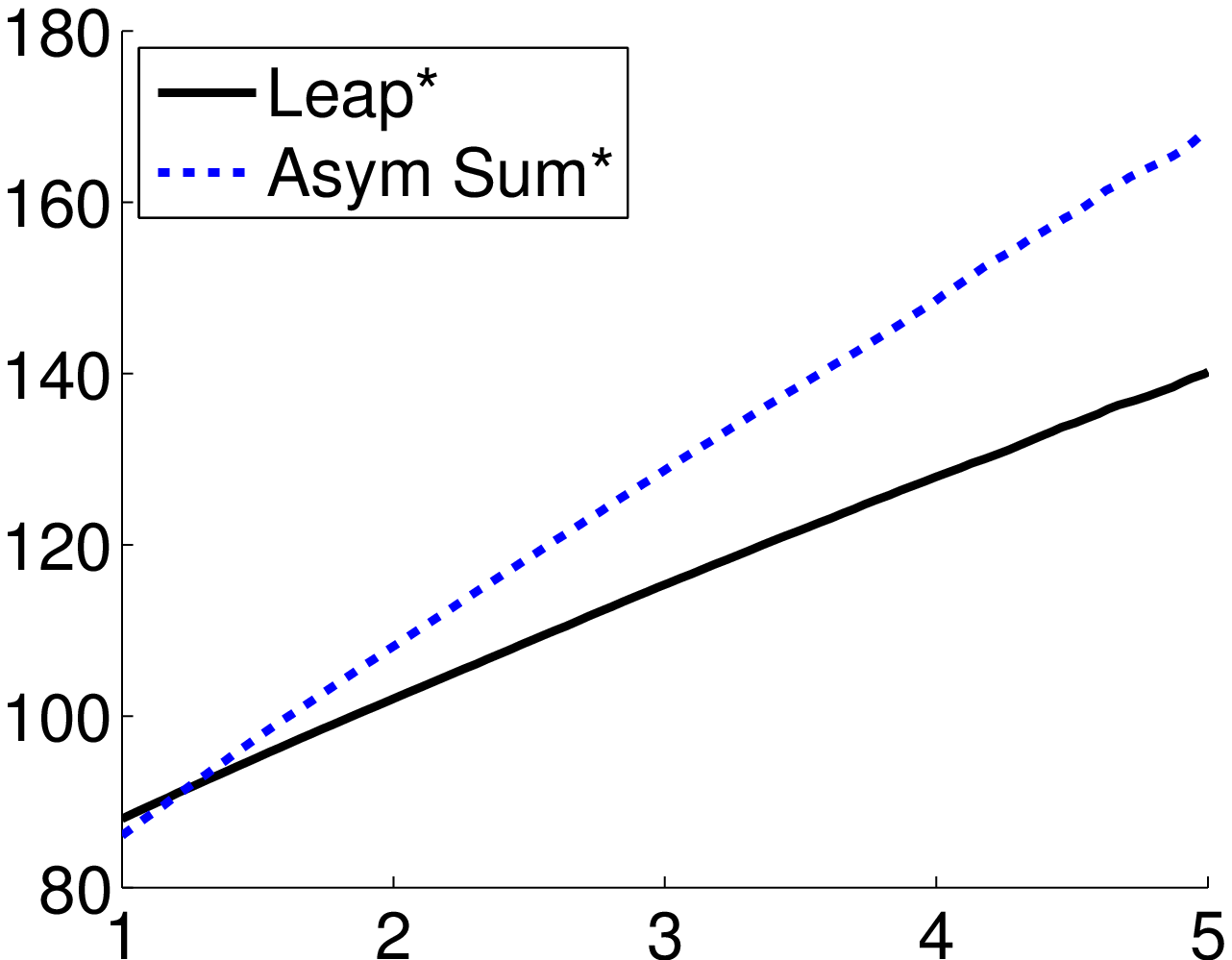} 
} 
\caption{The testing problem~\eqref{composite_normal_mean} with  $J = 20, \mu = 0.2, k_1=k_2 = 2$ and the initial sample size $n_0 = 10$. 
The x-axis in both graphs  is  $|\log_{10}(\text{Err})|$. 
The y-axis  is the corresponding ESS under $\butheta$ given by~\eqref{simulation choice}.
The second figure plots  two of the lines in the first figure. Note that for the sequential procedures, the initial sample size $n_0$ is added to the ESS.
}
\label{fig:composite}
\end{figure}

\subsection{Discussion on the local test statistics}\label{comp_discussion}

When there is only one stream (i.e. $J = 1$), the adaptive log-likelihood ratio statistic~\eqref{adaptive_stat} was first proposed in~\cite{robbins1974} in the context of power one tests, and later extended by~\cite{pavlov1991sequential} to sequential multi-hypothesis testing. There are two other popular choices for the local test statistics in  the case of composite hypotheses. 

The first one is to follow the approach suggested by Wald~\cite{wald1945sequential}  and replace $\lambda^j(n)$ in the Leap rule~\eqref{leap_def} by
the following \textit{mixture}  log-likelihood ratio statistic:
$$
\log\left(
\frac{\int_{\Theta^j_1}\, \exp\left(\ell(n, \theta^j)\right) \; \omega_1^{j}(d\theta^j)}
{\int_{\Theta^j_0}\, \exp\left(\ell(n, \theta^j)\right) \;  \omega_0^{j}(d\theta^j)}
\right),
$$
where $\omega^{j}_0, \omega^{j}_1$ are two probability measures on $\Theta^j_0$ and $\Theta^j_1$ respectively.  The second is to replace $\lambda^j(n)$ in the Leap rule~\eqref{leap_def} by the  \textit{generalized} log-likelihood ratio (GLR) statistic
${\ell}^j_{1}(n) - {\ell}^j_{0}(n)$.
When there is only one stream (i.e. $J = 1$), the corresponding sequential test  has been studied in \cite{lorden1973} for one-parameter exponential family, in~\cite{chan2000asymptotic} for multi-parameter exponential family, and in~\cite{li2014generalized} for separate families of hypotheses.  

We have chosen  the adaptive log-likelihood ratio statistics~\eqref{adaptive_stat} in this paper mainly because they allow for explicit and universal error control. Indeed, with this choice of statistics, the upper bounds on the error probabilities rely on a change-of-measure argument, in view of  Lemma~\ref{comp_martingale},
whereas this argument breaks down when we  use GLR or mixture  statistics.

\section{Sequential testing of two composite hypotheses in exponential family} \label{sub:exp_family}

In this section, we show that~\eqref{composite_uniform_convergence} holds if each stream has i.i.d. observations from an exponential family distribution, both the null and alternative parameter spaces are compact,
and the maximal likelihood estimator is used in the adaptive log-likelihood statistics \eqref{adaptive_stat}.
Note that~\eqref{composite_uniform_convergence} is a condition on each \textit{individual stream}, thus in this section we  drop the superscript $j$.

Let $\{X_n: n \in \bN\}$ be a sequence of i.i.d. random vectors in $\bR^d$ with common density
$$
p_{\theta}(x) = \exp\left(\theta^T x - b(\theta) \right)
$$ 
with respect to some measure $\nu$,  where superscript $T$ means transpose. We assume  that the natural parameter space 
$$
\Theta := \{\theta \in \bR^d: \int p_{\theta}(x) \nu(dx) < \infty \}
$$
is an open subset of $\bR^d$. For any $\theta, \widetilde{\theta} \in \Theta$,  the Kullback-Leibler divergence between 
$p_{\theta}$ and $p_{\widetilde{\theta}}$ is denoted by
$$
I(\theta, \widetilde{\theta}) := \Exp_{\theta} \left[ \log \frac{p_{\theta}(X_1)}{p_{\widetilde{\theta}}(X_1)}\right] = (\theta - \widetilde{\theta})^{T} \nabla b(\theta) - (b(\theta) - b(\widetilde{\theta})),
$$
where $\nabla$ stands for the gradient. We denote by $\{\ell(n,\theta): n \in \bN\}$ the log-likelihood process:
$$
\ell(n,\theta) := \sum_{i=1}^{n} \log p_{\theta}(X_i) 
= \sum_{i=1}^{n} (\theta^{T}X_i - b(\theta)) \quad\text{ for } n \in \bN.
$$
We assume that $\Theta_0, \Theta_1$ are two \textit{disjoint, compact} subsets of $\Theta$,
and denote  by
$$\widehat{\theta}_n := \arg\max_{\theta \in \Theta_0 \cup \Theta_1} \ell(n,\theta)$$  the maximum likelihood estimator based on the data up to time $n$ over the set $\Theta_0 \cup \Theta_1$. 
Picking any deterministic $\widehat{\theta}_0 \in \Theta$,  we define
$$
\ell_*(n) := \sum_{i=1}^{n} \log p_{\widehat{\theta}_{i-1}}(X_i) 
= \sum_{i=1}^{n} (\widehat{\theta}_{i-1}^{T}X_i - b(\widehat{\theta}_{i-1})) \quad\text{ for } n \in \bN.
$$

The main result of this subsection is summarized in the following theorem.

\begin{theorem}\label{exp_fam_complete_convergence}
Let $\theta \in \Theta_1$ and set $I(\theta) := \inf_{\theta_0 \in  \Theta_0} I(\theta, \theta_0)$.
Then, for any $\epsilon > 0$,
$$\sum_{n=1}^{\infty}\Pro_{\theta}\left( 
\frac{\ell_{*}(n) - \ell_{0}(n)}{n} - I(\theta) < \epsilon \right) < \infty,
$$
where $\ell_{0}(n) := \sup_{\theta_0 \in \Theta_0} \ell(n,\theta_0)$.
\end{theorem}
\begin{proof}
Observe that for any $\theta_0 \in \Theta_0$,
$$
\ell_{*}(n) - \ell(n,\theta_0) =
\ell_{*}(n) - \ell(n,\theta) + \ell(n,\theta) - \ell(n,\theta_0) - nI(\theta, \theta_0)
+ nI(\theta, \theta_0),
$$
which implies that
\begin{align*}
&\ell_{*}(n) - \ell_{0}(n) \\
=\;& 
\ell_{*}(n) - \ell(n,\theta) + \inf_{\theta_0 \in \Theta_0}\left(  \ell(n,\theta) - \ell(n,\theta_0) - nI(\theta, \theta_0) + nI(\theta, \theta_0)\right) \\
\geq \;&\ell_{*}(n) - \ell(n,\theta) + \inf_{\theta_0 \in \Theta_0}\left( \ell(n,\theta) - \ell(n,\theta_0) - nI(\theta, \theta_0)\right) + nI(\theta).
\end{align*}
As a result, it suffices to show that
\begin{align}
&\frac{1}{n}(\ell_{*}(n) - \ell(n,\theta)) 
\xrightarrow[n \to \infty]{\Pro_{\theta} \text{ completely } } 0,
\label{expfamily_to_show_1}
\\
&\frac{1}{n}\inf_{\theta_0 \in \Theta_0}\left( \ell(n,\theta) - \ell(n,\theta_0) - nI(\theta, \theta_0)\right)
\xrightarrow[n \to \infty]{\Pro_{\theta} \text{ completely } } 0,
\label{expfamily_to_show_2}
\end{align}
which are the content of the next two lemmas.
\end{proof}
\begin{remark}
The sequence in~\eqref{expfamily_to_show_1} concerns the behavior of the maximal likelihood estimator for the exponential family distribution, while the sequence in~\eqref{expfamily_to_show_2} concerns the uniform behavior over $\Theta_0$.
\end{remark}

\begin{lemma}
For any $\theta \in \Theta$, as $n \to \infty$, $\frac{1}{n}(\ell_{*}(n) - \ell(n,\theta))$ converges completely to zero under $\Pro_{\theta}$.
\end{lemma}
\begin{proof}
Since $\Theta_0$ and $\Theta_1$ are compact, there exists $K > 0$ such that
$$
\max\{\lVert \widetilde{\theta} \rVert,\;
I(\theta, \widetilde{\theta})
\}
 < K \;\;\text{ for any } \widetilde{\theta} \in \Theta_0 \cup \Theta_1, 
$$
where we use $\lVert \cdot \rVert$ to denote the Euclidean distance.

Observe that  $\frac{1}{n}(\ell_{*}(n) - \ell(n,\theta)) = \frac{1}{n} M_n - \frac{1}{n} R_n$, where
\begin{align*}
&M_n := \ell_{*}(n) - \ell(n,\theta) + \sum_{i=1}^{n} I(\theta, \widehat{\theta}_{i-1})
= \sum_{i=1}^{n}(\widehat{\theta}_{i-1} - \theta)^{T}(X_i - \nabla b(\theta)),\\
& R_n := \sum_{i=1}^{n} I(\theta, \widehat{\theta}_{i-1})
\end{align*}

Denote $\cF_{n} := \sigma(X_1, \ldots, X_{n})$ the $\sigma$-field generated by the first $n$ observations. Then $\{M_n: n \in \bN\}$ is an $\{\cF_n\}$-martingale, since $\Exp[X_1] = \nabla b(\theta)$ due to the property of the exponential family and 
$\widehat{\theta}_{n-1} \in \cF_{n-1}$.
Further, the martingale difference sequence $\{(\widehat{\theta}_{i-1} - \theta)^{T}(X_i - \nabla b(\theta)): i \in \bN\}$ is bounded in $L^{p}$ for any $p >2$. Indeed, by Cauchy–-Schwarz inequality,
$$
\sup_{i \in \bN}\Exp |(\widehat{\theta}_{i-1} - \theta)^{T}(X_i - \nabla b(\theta))|^p \leq (2K)^p \Exp \lVert X_1 - \nabla b(\theta)\rVert^p < \infty.
$$
Then by~\cite{stoica2007baum}, we conclude $\frac{1}{n}M_n$ converges completely to zero under $\Pro_{\theta}$.

It remains to show that $\frac{1}{n} R_n$ converges completely to zero under $\Pro_{\theta}$.
Fix any $\epsilon> 0$. Since $I(\theta, \widetilde{\theta})$ is continuous in $\widetilde{\theta}$, there exists 
$\delta > 0$ such that if $\lVert \widetilde{\theta} - \theta \rVert \leq \delta$, $I(\theta,\widetilde{\theta}) \leq \epsilon/2$.
Define three random times
\begin{align*}
\eta_1 &:= \sup\{n \in \bN: |R_n| > n \, \epsilon \}, \\
\eta_2 &:= \sup\{n \in \bN: |I(\theta, \widehat{\theta}_n)| > \epsilon/2 \},\quad
\eta_3 := \sup\{n \in \bN: \lVert \widehat{\theta}_n - \theta \rVert > \delta\}
\end{align*}
By Theorem 5.1 in~\cite{pavlov1991sequential}, there exist constant $c_1$ and $c_2$ such that
$\Pro_{\theta}(\eta_3 > n) \leq c_1 \exp(-c_2 n)$ for any $n \in \bN$. In particular,
$$
\Exp_{\theta} [\eta_3] < \infty.
$$
Clearly, $\eta_2 \leq \eta_3$, which implies that $\Exp_{\theta}[\eta_2] < \infty$.
We next show that $\eta_1 \leq 2\epsilon K \eta_2$. Indeed,
for $n \geq 2 K \eta_2/\epsilon$,
$$
\frac{1}{n} |R_n|
\leq \frac{1}{n}\left(\sum_{i=1}^{\eta_2} I(\theta, \widehat{\theta}_{i-1})
+ \sum_{i=\eta_2+1}^{n} I(\theta, \widehat{\theta}_{i-1}) \right)
\leq \frac{K \eta_2 + n \, \epsilon/2 }{n} \leq \epsilon.
$$
Thus $\Exp_{\theta}[\eta_1] < \infty$, which implies $\frac{1}{n}R_n$ converges to zero quickly. (See Chapter 2.4.3 in~\cite{tartakovsky2014sequential} for formal definition of quick convergence.) Due to Lemma 2.4.1 in~\cite{tartakovsky2014sequential}, quick convergence implies complete convergence, and thus 
$\frac{1}{n}R_n$ converges to zero completely.
\end{proof}

\begin{lemma}
Assume the conditions in Theorem~\ref{exp_fam_complete_convergence} hold. Then
$$
\frac{1}{n}\inf_{\theta_0 \in \Theta_0}\left( \ell(n,\theta) - \ell(n,\theta_0) - nI(\theta, \theta_0)\right)
\xrightarrow[n \to \infty]{\Pro_{\theta} \text{ completely } } 0.
$$
\end{lemma}
\begin{proof}
By definition, we have
\begin{align*}
&\frac{1}{n}\inf_{\theta_0 \in \Theta_0}\left( \ell(n,\theta) - \ell(n,\theta_0) - nI(\theta, \theta_0)\right) \\
=\;& \frac{1}{n} \inf_{\theta_0 \in \Theta_0} \sum_{i=1}^{n}(\theta - \theta_0)^{T}(X_i - \nabla b(\theta)) \\
=\; &\inf_{\theta_0 \in \Theta_0} (\theta - \theta_0)^T (\frac{1}{n}\sum_{i=1}^{n}(X_i - \nabla b(\theta))).
\end{align*}
Denote $\theta_j$, $\theta_{0,j}$, $X_{i,j}$ and $\nabla_{j} b(\theta)$ the $j^{th}$ dimension of the $\bR^d$ vectors 
$\theta$,
$\theta_0$, $X_{i}$ and $\nabla b(\theta)$.
Since $\Theta_0, \Theta_1$ is compact, there exists $K > 0$ such that
$$
|\theta_j|, |\theta_{0,j}|  \leq K, \text{ for any } 1 \leq j \leq d, \;\;\theta_0 \in \Theta_0.
$$
By triangle inequality,
\begin{align*}
&\left\vert \frac{1}{n}\inf_{\theta_0 \in \Theta_0}\left(  \ell(n,\theta) - \ell(n,\theta_0) - nI(\theta, \theta_0)\right)  \right\vert 
\leq\;\;  2K \sum_{j=1}^{d} \left|\frac{1}{n}\sum_{i=1}^{d}(X_{i,j} - \nabla_j b(\theta))\right|.
\end{align*}
But for each $1 \leq j \leq d$,  since $\Exp_{\theta} [X_{i,j}^2] <\infty$, by~\cite{hsu1947complete},
$$
\frac{1}{n}\sum_{i=1}^{d}(X_{i,j} - \nabla_j b(\theta)) 
\xrightarrow[n \to \infty]{\Pro_{\theta} \text{ completely } } 0,
$$
which completes the proof.
\end{proof}

\section{Two renewal-type lemmas}
In this section, we present two renewal-type lemmas about general discrete stochastic process, 
which may be of independent interest.

\begin{lemma}\label{T_small_ST_larger_lemma}
Let $\{\xi_i(n): n \in \bN\}$ ($i=1,2$) be two stochastic processes on some probability space $(\Omega, \mathcal{F}, \Pro)$. Suppose that for some positive $ \mu_1, \mu_2$,
\begin{equation*}
\Pro\left(\lim_{n \to \infty}\; \frac{1}{n} \xi_i(n) = \mu_i\right) = 1\quad \text{ for }\; i = 1,2.
\end{equation*}
Let $c$ be a fixed constant. Then for any $q \in (0,1)$,
\begin{align}
\label{one_process}
\lim_{b \to \infty} \;&\sup_{T}\;\Pro \left( T \leq q \frac{b}{\mu_1},\; \xi_1(T) \geq b + c \right) = 0,\\
\label{two_processes}
\lim_{a,\,b \to \infty} \;&\sup_{T}\;\Pro \left(T \leq q (\frac{a}{\mu_1} \bigvee \frac{b}{\mu_2}),\; \xi_1(T) \geq a + c,\; \xi_2(T) \geq b + c \right) = 0.
\end{align}
where the supremum is taken over all random time $T$.
\end{lemma}

\begin{proof} Since $c$ is fixed, we assume $c=0$ without loss of generality. Denote $N_b = \lfloor q\frac{b}{\mu_1} \rfloor$, and 
$\epsilon_q = \frac{1}{q} - 1 > 0$. Notice that $\Pro(T \leq q \frac{b}{\mu_1},\; \xi_1(T) \geq b)$ is upper bounded by
\begin{align*}
 \Pro \left(\max_{1 \leq n  \leq N_b} \xi_1(n) \geq b \right)\; \leq\;  \Pro \left(\frac{1}{N_b}\max_{1 \leq n \leq N_b} \xi_1(n) \geq (1 + \epsilon_q) \mu_1 \right) \to 0
\end{align*}
where the convergence follows directly from  \cite[Lemma A.1]{fellouris2017multichannel}. Thus, the proof of \eqref{one_process} is complete.

For the second part, assume \eqref{two_processes} does not hold. Then, there exists some $\epsilon > 0$ and	 a sequence $(a_n, b_n)$ with $a_n, b_n \to \infty$ such that $p_n \geq \epsilon$  for large  $n \in \bN$, where 
$$
p_n := \sup_{T} \Pro \left(T \leq q (\frac{a_n}{\mu_1} \bigvee \frac{b_n}{\mu_2}),\, \xi_1(T) \geq a_n, \, \xi_2(T) \geq b_n \right).
$$
We can assume $a_n / \mu_1 \geq b_n/ \mu_2$ for every $n \in \bN$, since otherwise we can take a subsequence and the following argument will still go through. Thus,
$$
\epsilon \;\leq \;	p_n \; \leq \; \sup_{T} \Pro\left(T \leq q \frac{a_n}{\mu_1}, \xi_1(T) \geq a_n \right),
$$
which contradicts with \eqref{one_process}. Thus the proof is complete.
\end{proof}
\begin{remark}
Note that in \eqref{two_processes} there is no restriction on the way $a,\,b$ approach infinity, and that $T$ is not required to be a stopping time.
\end{remark}

The next lemma provides an upper bound on the expectation of the first time when multiple processes simultaneous cross given thresholds. 
\begin{lemma}\label{multiple_cross}
Let $L \geq 2$ and $\{\xi_\ell(n): n \in \bN \}_{\ell \in [L]}$ be $L$ stochastic processes on some probability space $(\Omega, \mathcal{F}, \Pro)$. Define the stopping time
$$
\nu(\vec{b}) \;:= \; \inf\{n \geq 1:\; \xi_{\ell}(n) \geq b_{\ell}\; \text{ for every }\; \ell \in [L]\}
$$
where $\vec{b} = \{b_1,\ldots, b_L\}$. Then for some positive $\mu_1,\ldots, \mu_L$, we have
\begin{align}\label{eq:upper_bound_general_lemma}
\Exp[\nu(\vec{b})]\; \leq \;  \max_{\ell \in [L]}\;\left\{ \frac{b_{\ell}}{\mu_\ell} \right\} (1 + o(1))
 \; \text{ as }\; \min_{\ell \in [L]} \{b_\ell \} \to \infty
\end{align}
if \textbf{one} of the following conditions holds:
(i). For each $\ell \in [L]$ and any $\epsilon > 0$, 
\begin{equation*} 
\sum_{n=1}^{\infty} \Pro\left( \big|\frac{1}{n} \xi_\ell(n) - \mu_\ell \big| \geq \epsilon \right) < \infty.
\end{equation*}
(ii). For each $\ell \in [L]$, $\{\xi_\ell(n): n \in \bN\}$ has independent and identically distributed increment, and 
\begin{equation*}
\Pro\left(\lim_{n \to \infty}\; \frac{1}{n} \xi_\ell(n) = \mu_\ell\right) = 1.
\end{equation*}
\end{lemma}
\begin{proof} Denote $N(\vec{b}) = \max_{\ell \in [L]} \left\{ b_{\ell}/\mu_\ell \right\}$, and $\vec{b}_{min} = \min\{b_1,\ldots, b_L\}$.  

First, assume condition (i) holds. Fix $\epsilon \in (0,1)$, and denote $N_{\epsilon}(\vec{b}) = \big\lfloor  N(\vec{b})/(1-\epsilon) \big\rfloor$. By definition of $\nu(\vec{b})$, we have
$$
\{\nu(\vec{b}) > n\} \; \subset \;\bigcup_{\ell \in [L]} \{\xi_\ell(n) < b_\ell \}
$$
By Boole's inequality, for $n > N_{\epsilon}(\vec{b})$,
\begin{align*}
\Pro(\nu(\vec{b}) > n) & 
\leq \sum_{\ell \in [L]} \Pro(\xi_\ell(n) < b_\ell ) 
\leq \sum_{\ell \in [L]}  \Pro\left(\frac{1}{n} \xi_\ell(n) < \frac{b_\ell}{N_{\epsilon}(\vec{b})+1} \right) \\
&\leq \sum_{\ell \in [L]}  \Pro\left(\frac{1}{n} \xi_\ell(n) < (1-\epsilon)\mu_\ell \right)  \\
& \leq \sum_{\ell \in [L]}  \Pro\left(\big|\frac{1}{n} \xi_\ell(n) - \mu_\ell| > \epsilon \mu_\ell \right),
\end{align*}
where we used the fact that $n \geq N_{\epsilon}(\vec{b})+1 \geq \frac{N(\vec{b})}{1-\epsilon} \geq 
\frac{b_{\ell}}{(1-\epsilon) \mu_\ell}$. Thus
\begin{align*}
\Exp[\nu(\vec{b})] &= \int_{0}^{\infty} \Pro(\nu(\vec{b}) > t)\,dt \;\leq N_{\epsilon}(\vec{b}) + 1 + \sum_{n > N_{\epsilon}(\vec{b})} \Pro(\nu(\vec{b}) > n) \\
&\leq N_{\epsilon}(\vec{b}) + 1 + \sum_{\ell \in [L]} \sum_{n > N_{\epsilon}(\vec{b})}    \Pro\left(\big|\frac{1}{n} \xi_\ell(n) - \mu_\ell \big| > \epsilon \mu_\ell \right)
\end{align*}
Due to condition (i), we have
\begin{align*}
\limsup_{\vec{b}_{min} \to \infty} \;
\frac{\Exp[\nu(\vec{b})]}{ N(\vec{b}) }
\; = \;
\limsup_{\vec{b}_{min} \to \infty} \;
(1-\epsilon)\frac{\Exp[\nu(\vec{b})]}{ N_{\epsilon}(\vec{b}) }
\;\leq\; 1 - \epsilon
\end{align*}
Since $\epsilon \in (0,1)$ is arbitrary, \eqref{eq:upper_bound_general_lemma} holds.

Now assume that condition (ii) holds. Clearly, $\nu(\vec{b}) \geq \nu_{\ell}(b_\ell)$, where
$$
\nu_{\ell}(b_\ell) := \inf\{n \geq 1: \xi_\ell(n) \geq b_\ell \}\; \text{ for }\; \ell \in [L].
$$
Due to condition (ii), we have
$$
\liminf_{b_\ell \to \infty} \, \frac{\nu(\vec{b})}{b_{\ell}/\mu_\ell} \; \geq \; 
\lim_{b_\ell \to \infty} \frac{\nu_{\ell}(b_\ell)}{b_{\ell}/\mu_\ell} = 1 \; \text{ for } \; \ell \in [L], 
$$
which implies $\liminf_{\vec{b}_{min} \to \infty} \nu(\vec{b}) / N(\vec{b}) \geq 1$. On the other hand, by the definition of $\nu(\vec{b})$, there exists $\ell' \in [L]$ such that
\begin{align*}
\xi_{\ell'}(\nu(\vec{b}) - 1) < b_{\ell'} \; \iff \; 
\frac{\xi_{\ell'}(\nu(\vec{b})) - b_{\ell'}}{\nu(\vec{b}) \mu_{\ell'}} \leq
\frac{\xi_{\ell'}(\nu(\vec{b})) - \xi_{\ell'}(\nu(\vec{b}) -1)}{\nu(\vec{b}) \mu_{\ell'}}.
\end{align*}
Taking the minimum on the l.h.s., and maximum on the right, we have
\begin{align*}
\min_{\ell \in [L]} \frac{\xi_{\ell}(\nu(\vec{b})) - b_{\ell}}{\nu(\vec{b}) \mu_{\ell}} \leq
\max_{\ell \in [L]} \frac{\xi_{\ell}(\nu(\vec{b})) - \xi_{\ell}(\nu(\vec{b}) -1)}{\nu(\vec{b}) \mu_{\ell}}.
\end{align*}
which  implies
\begin{align*}
\frac{N(\vec{b})}{\nu(\vec{b})} \,=\, \max_{\ell \in [L]} \frac{b_\ell}{\nu(\vec{b}) \mu_{\ell}} \,\geq \, 
\min_{\ell \in [L]} \frac{\xi_{\ell}(\nu(\vec{b}))}{\nu(\vec{b}) \mu_{\ell}} -
\max_{\ell \in [L]} \frac{\xi_{\ell}(\nu(\vec{b})) - \xi_{\ell}(\nu(\vec{b}) -1)}{\nu(\vec{b}) \mu_{\ell}} 
\end{align*}
where the last term will goes to $1$ as $\vec{b}_{min} \to \infty$ due to condition (ii). Thus, $\liminf N(\vec{b})/\nu(\vec{b})  \geq 1$ as $\vec{b}_{min} \to \infty$, which together with previous reverse inequality, shows that $\nu(\vec{b}) / N(\vec{b}) \to 1$ almost surely as $\vec{b}_{min} \to \infty$. Thus, the proof would be
complete if we can show the following:
\begin{align*}
(*)\quad \quad \mathcal{C}_1 = \left\{\frac{\nu(\vec{b})}{N(\vec{b})}:\; b_1,\ldots, b_L > 0 \right\} \; \text{is uniformly integrable}
\end{align*}
Define $\mu_{max} = \max\{\mu_1,\ldots, \mu_L\} > 0$, $b_{max} = \max\{b_1,\ldots, b_L\}$ and 
$$\nu'(c) = \inf \{n\geq 1: \xi_\ell \geq c \;\;\text{for every }\; \ell \in [L]\} \; \text{ for }\; c>0.
$$ 
By Theorem 3 of \cite{farrell1964}, $\mathcal{C}_2 = \{\nu'(c)/c: c > 0\}$ is uniformly integrable. Observe that 
$$
\nu(\vec{b}) \leq \nu'(b_{max})\;,\;
N(\vec{b}) \geq \frac{b_{max}}{\mu_{max}}
\; \Rightarrow \;
\frac{\nu(\vec{b})}{N(\vec{b})}  \,\leq\,
\mu_{max} \frac{\nu'(b_{max})}{b_{max}} \; \in \; \mu_{max}\; \mathcal{C}_2.
$$
Since $\mu_{max}$ is a constant, $\mathcal{C}_1$ is dominated by a uniformly integrable family. Thus condition $(*)$ holds, and the proof is complete.
\end{proof}

\section{Generalized Chernoff's lemma}
In this section we present a generalization of Chernoff's lemma  \citep[Corollary 3.4.6]{dembo1998large} that allows for different requirements on the type I and type II errors. Let $\{X_n, n \in \bN\}$ be a sequence of independent random variables with common density $f$ relative to some $\sigma$-finite measure $\nu$ and consider the following simple versus simple testing problem:
$$
{\Hyp_0}: f = f_0 \quad \text{ vs. } \quad
{\Hyp_1}: f = f_1.
$$
Let $\mathcal{S}_n$ be the class of $\cF_n$-measurable random variables taking value in $\{0,1\}$,
where $\cF_n = \sigma(X_1,\ldots,X_n)$.  
For any procedure $D_n \in \mathcal{S}_n$, denote
$$
p_n(D_n) := \Pro_0(D_n = 1),\quad
q_n(D_n) := \Pro_1(D_n = 0),
$$
where $\Pro_i$ is the probability measure under $\Hyp_i$ for $i = 0,1$.
Further, denoting $Y := {f_1(X_1)}/{f_0(X_1)}$, we define
$$
\Phi(z) := \sup_{\theta \in \bR}\left\{
z\theta - \log\left(\Exp_0
[Y^{\theta}]
\right)
\right\}, \;
I_0 := \Exp_0[-\log(Y)], \;
I_1 := \Exp_1[\log(Y)],
$$
with the possibility that either $I_0$ or $I_1$ is equal to $\infty$. We assume that there exists  $h_d \in (-I_0, I_1)$ such that
\begin{equation} \label{hd}
\Phi(h_d) / d = \Phi(h_d) - h_d.
\end{equation}
In particular, if $d = 1$, we can set $h_d = 0$. 

\begin{lemma}(Generalized Chernoff's Lemma)\label{generalized_chernoff_lemma}
For any $d > 0$,
$$
\lim_{n \to \infty} \inf_{D_n \in \mathcal{S}_n}\frac{1}{n} 
\log\left(p_n^{1/d}(D_n) + q_n(D_n) \right)
= - \frac{\Phi(h_d)}{d}.
$$
\end{lemma}
\begin{remark}
When $d = 1$, since we can select $h_d = 0$, it reduces to Chernoff's Lemma \citep[Corollary 3.4.6]{dembo1998large}.
For $d \neq 1$, the proof is essentially the same, and we present it here for completeness. 
\end{remark}

\begin{proof}[Proof of Lemma \ref{generalized_chernoff_lemma}]
For fixed $n \in \bN$, due to the Neyman-\-Pearson Lemma, it suffices to consider the tests of the following form:
$$
\delta_n(h) := 1 \;\; \Leftrightarrow  \;\;\frac{1}{n}\lambda(n) \geq h, \quad \text{where} \quad \lambda(n) := \sum_{i=1}^{n} \log\frac{f_1(X_i)}{f_0(X_i)}.
$$
Then, we have
$$
\inf_{D_n \in \mathcal{S}_n}
\log\left(p_n^{1/d}(D_n) + q_n(D_n) \right)
=
\inf_{h \in \bR} 
\log\left(p_n^{1/d}(\delta_n(h)) + q_n(\delta_n(h)) \right)
$$
Since $p_n(\delta_n(h))$ is decreasing in $h$ and $q_n(\delta_n(h))$  increasing in $h$, 
for any $h \in \bR$,
$\text{ either } p_{n}(\delta_n(h)) \geq p_{n}(\delta_n(h_d))
\text{ or } q_{n}(\delta_n(h)) \geq q_{n}(\delta_n(h_d))$.
Thus 
\begin{align*}
\inf_{D_n \in \mathcal{S}_n}
\log\left(p_n^{1/d}(D_n) + q_n(D_n) \right)
\geq 
\log \min \left\{p_n^{1/d}(\delta_n(h_d)),q_n(\delta_n(h_d)) \right\}.
\end{align*}
By   \citep[Theorem 3.4.3]{dembo1998large}, as $n \to \infty$,
\begin{align*}
\frac{1}{n}\log(p_n^{1/d}(\delta_n(h_d)) \to -\frac{\Phi(h_d)}{d},\quad
\frac{1}{n}\log(q_n(\delta_n(h_d))) \to -(\Phi(h_d) - h_d).
\end{align*}
Thus, by the definition of $h_d$ in \eqref{hd} and letting $n \to \infty$ we obtain
$$
\liminf_{n \to \infty} \inf_{D_n \in \mathcal{S}_n}\frac{1}{n} \log(
p_n^{1/d}(D_n) +  q_n(D_n))
\geq - \frac{\Phi(h_d)}{d}.
$$
Clearly, the lower bound is attained by the Neyman-Pearson rule with threshold $h_d$,  $\delta_n(h_d)$, which completes the proof.
\end{proof}